\documentclass[12pt, a4paper, reqno]{amsart}

\usepackage[english]{babel}
\usepackage[T1]{fontenc}
\usepackage{lmodern}
\usepackage{diagbox}
\usepackage{amsmath}
\usepackage{amssymb}
\usepackage{amsthm}
\usepackage{multirow}
\usepackage{anyfontsize}  
\usepackage{url}
\usepackage{cases}
\usepackage{hyperref} 
\usepackage{enumerate}
\usepackage[normalem]{ulem}
\usepackage{float}
\usepackage{mathtools, stmaryrd, latexsym, todonotes, color}
\usepackage{amsfonts, stmaryrd, latexsym, mathrsfs}

\usepackage[normalem]{ulem}
\newcommand{\AECh}[2]{\textcolor{magenta}{\sout{#1}}\textcolor{blue}{#2}}

\makeatletter
\def\namedlabel#1#2{\begingroup
   \def\@currentlabel{#2}%
   \label{#1}\endgroup
}
\makeatother

\newcommand{\mmid}[0]{\parallel}

\newcommand{\gammatilde}[0]{\tilde{\gamma}}

\newcommand{\cC}{\mathcal{C}}

\newcommand{\ppar}{\ \par}
\newcommand{\Size}[1]{\left\lvert #1 \right\rvert}

\newcommand{\Span}[1]{\left\langle\, #1 \,\right\rangle}
\newcommand{\Set}[1]{\left\{ #1 \right\}}

\newcommand{\F}{\mathbb{F}}

\newcommand{\Gal}{\mathop{\mathrm{Gal}}\nolimits}
\newcommand{\ord}{\mathop{\mathrm{ord}}\nolimits}

\renewcommand{\epsilon}{\varepsilon}
\renewcommand{\theta}[0]{\vartheta}
\renewcommand{\phi}[0]{\varphi}

\DeclareMathOperator{\GL}{GL}

\DeclareMathOperator{\End}{End}
\DeclareMathOperator{\Aut}{Aut}
\DeclareMathOperator{\Inn}{Inn}

\DeclareMathOperator{\Hol}{Hol}

\DeclareMathOperator{\diag}{diag}
\DeclareMathOperator{\Stab}{Stab}
\DeclareMathOperator{\Orb}{Orb}
\DeclareMathOperator{\Norm}{Norm}
\DeclareMathOperator{\tr}{tr}
\DeclareMathOperator{\id}{id}

\DeclareMathOperator{\Perm}{Perm}

\newtheorem{dummy}{Dummy}
\numberwithin{dummy}{section}
\numberwithin{figure}{section}

\newtheorem{theorem}[dummy]{Theorem}

\newtheorem{remark}[dummy]{Remark}
\newtheorem{lemma}[dummy]{Lemma}
\newtheorem{recap}[dummy]{Recap}
\newtheorem{proposition}[dummy]{Proposition}
\newtheorem{prop}[dummy]{Proposition}

\theoremstyle{definition}
\newtheorem{definition}[dummy]{Definition}

\newtheorem*{remark*}{Remark}
\theoremstyle{remark}

\makeatletter
\def\imod#1{\allowbreak\mkern10mu({\operator@font mod}\,\,#1)}
\makeatother

\newcounter{todocounter}

\numberwithin{equation}{section}

\definecolor{MyGreen}{RGB}{20, 140, 60}

\definecolor{airforceblue}{rgb}{0.36, 0.54, 0.66}
\definecolor{celestialblue}{rgb}{0.29, 0.59, 0.82}
\definecolor{amaranth}{rgb}{0.9, 0.17, 0.31}
\definecolor{antiquefuchsia}{rgb}{0.57, 0.36, 0.51}
\definecolor{awesome}{rgb}{1.0, 0.13, 0.32}
\definecolor{atomictangerine}{rgb}{1.0, 0.6, 0.4}
\definecolor{brickred}{rgb}{0.8, 0.25, 0.33}
\definecolor{palered-violet}{rgb}{0.86, 0.44, 0.58}
\definecolor{palerobineggblue}{rgb}{0.59, 0.87, 0.82}
\definecolor{caribbeangreen}{rgb}{0.0, 0.8, 0.6}
\definecolor{dodgerblue}{rgb}{0.12, 0.56, 1.0}

\hypersetup{
  linkbordercolor=dodgerblue
}

\begin{document}

\date{24 March 2023, 15:50 CET -- version 5.12}
      
\title[$p^{2} q$]%
      {Hopf-Galois structures on\\
        extensions of degree $p^{2} q$\\
        and skew braces of order $p^{2} q$:\\
        the elementary abelian Sylow $p$-subgroup case}      
      
\author{E. Campedel}

\address[E.~Campedel]%
{Dipartimento di Matematica e Applicazioni\\
Edificio U5\\
Universit\`a degli Studi di Milano-Bicocca\\
via Roberto Cozzi, 55\\
20126 Milano}

\email{e.campedel1@campus.unimib.it}

\author{A. Caranti}

\address[A.~Caranti]%
 {Dipartimento di Matematica\\
  Universit\`a degli Studi di Trento\\
  via Sommarive 14\\
  I-38123 Trento\\
  Italy} 

\email{andrea.caranti@unitn.it} 

\urladdr{https://caranti.maths.unitn.it/}

\author{I. Del Corso}

\address[I.~Del Corso]%
        {Dipartimento di Matematica\\
          Universit\`a di Pisa\\
          Largo Bruno Pontecorvo, 5\\
          56127 Pisa\\
          Italy}
\email{ilaria.delcorso@unipi.it}

\urladdr{http://people.dm.unipi.it/delcorso/}

\begin{abstract}    
  Let $p, q$  be distinct primes, with  $p > 2$. In a  previous paper we
  classified the  Hopf-Galois structures on Galois  extensions of degree
  $p^{2}  q$, when  the  Sylow  $p$-subgroups of  the  Galois group  are
  cyclic.  This  is equivalent to  classifying the skew braces  of order
  $p^2q$, for which the Sylow  $p$-subgroups of the multiplicative group
  is cyclic.   In this paper  we complete the classification  by dealing
  with the  case when the  Sylow $p$-subgroups  of the Galois  group are
  elementary abelian.

  According to  Greither and  Pareigis, and  Byott, we  will do  this by
  classifying the regular subgroups of the holomorphs of the groups $(G,
  \cdot)$ of order  $p^{2} q$, in the case when  the Sylow $p$-subgroups
  of $G$ are elementary abelian.
  
  We  rely   on  the  use   of  certain  gamma   functions  $\gamma:G\to
  \Aut(G)$. These  functions are  in one-to-one correspondence  with the
  regular subgroups  of the holomorph  of $G$, and are  characterised by
  the  functional equation  $\gamma(g^{\gamma(h)} \cdot  h) =  \gamma(g)
  \gamma(h)$, for $g,  h \in G$.  We develop methods  to deal with these
  functions, with  the aim of  making their enumeration easier  and more
  conceptual.
\end{abstract}



\thanks{This paper is part of the thesis submitted by E. Campedel as a
  requirement for the PhD degree at the University of Milano--Bicocca.
  \\
  The authors are members of INdAM---GNSAGA. The authors
  gratefully acknowledge support from the 
  Departments of  Mathematics of  the Universities  of Milano--Bicocca,
  Pisa, and Trento.}

\maketitle

\thispagestyle{empty}

\section{Introduction}
\label{sec:intro}

\subsection{The general problem, and the classical approach}

Let   $L/K$   be   a   finite  Galois   field   extension,   and   let
$\Gamma=\Gal(L/K)$.  The  group  algebra  $K[\Gamma]$  is  a  $K$-Hopf
algebra,  and  its   natural  action  on  $L$  endows   $L/K$  with  a
Hopf--Galois structure. In  general this is not  the only Hopf--Galois
structure  on  $L/K$,  and  the  study  of  Galois  module  structures
other than the  classical one  is  relevant, for  example in  the
context  of  algebraic  number  theory  where  different  Hopf--Galois
structures  may  behave differently  at  integral  level (see  Child's
book~\cite{ChildsBook} for an overview and \cite{Byott97} for a specific
result).
 
 This motivated the study  of the  Hopf-Galois structures on a finite
 field extension $L/K$ and their classification. 

On   the  other   hand,  the   celebrated  result   by  Greither   and
Pareigis~\cite[Theorem~2.1]{GP}    showed   that    all   Hopf--Galois
structures  on $L/K$  can be  described  in a  purely group  theoretic
way.  In the  reformulation  due to  Byott~\cite{Byott96} this  result
states that to each Hopf--Galois structure  on $L/K$ one can associate
a group $G$, with the same  cardinality as $\Gamma$, and such that the
holomorph  $\Hol(G)$,  regarded  as  a subgroup  of  $\Perm(G)$,
contains a regular subgroup isomorphic to $\Gamma$.
 
As first noticed by  Bachiller in~\cite{Bac16}, and
clearly  explained  in the  appendix  to  \cite{SV2018} by  Byott  and
Vendramin,  classifying  the  regular  subgroups   of  $\Hol(G)$  is
equivalent to  determining the operations  ``$\circ$'' on $G$  such that
$(G, \cdot, \circ)$ is a (right) skew brace (in the relevant literature it
is more common to use left skew  braces; we have translated the statements
in the  literature from  left to  right). Therefore,  the Hopf--Galois
structures on  an extension  with Galois group  isomorphic to  a group
$\Gamma$ correspond  bijectively to  the skew  braces $(G,  \cdot, \circ)$
with $(G,\circ)=\Gamma$; see also the recent work~\cite{st23}.  This has further motivated  the study of
this context, which in recent years has been deeply investigated.

\begin{definition}
\label{def:e-e'}
  Let $\Gamma$, $G=(G, \cdot)$ be finite groups with $\Size{G} =  \Size{\Gamma}$. We define the following numbers.
  \begin{enumerate}
  \item $e(\Gamma, G)$, the number of Hopf-Galois structures of type $G$ on a Galois extension with group $\Gamma$,   
  \item $e'(\Gamma,G)$, the number of regular subgroups of $\Hol(G)$ isomorphic to $\Gamma$, 
  \item $e''(\Gamma, G)$, the total number of (right) skew braces $(G, \cdot, \circ)$ such that $\Gamma     \cong (G,\circ)$.
  \item $f'(\Gamma,G)$, the number of classes of regular subgroups of $\Hol(G)$ isomorphic to $\Gamma$, under conjugation by elements of $\Aut(G)$,
  \item $f''(\Gamma, G)$, the number of isomorphism classes of skew braces $(G, \cdot, \circ)$ such that $\Gamma \cong (G,\circ)$.
  \end{enumerate}
\end{definition}

\begin{remark*}
  Recall that, given a group $(G, \cdot)$, by the \emph{(total) number
    of skew  braces on $(G,  \cdot)$} we  mean the number  of distinct
  operations ``$\circ$'' on the set  $G$ such that $(G, \cdot, \circ)$
  is a skew brace.
\end{remark*}

\begin{theorem}
 \label{th:hgs-regular-skew}
  Let  $L/K$ be  a finite  Galois  field extension  with Galois  group
  $\Gamma$.  For  any group $G=(G, \cdot)$  with $\Size{G} =  \Size{\Gamma}$,
it is equivalent to determine 
\begin{description}
\item[{{\cite[Theorem 4.2]{skew}}}\namedlabel{item:skew-GV17}{\cite[Theorem 4.2]{skew}}] $e'(\Gamma,G)$ and $e''(\Gamma, G)$;
\item[{{\cite[Proposition A.3]{SV2018}}}\namedlabel{item:classes-BV18}{\cite[Proposition A.3]{SV2018}}] $f'(\Gamma,G)$ and $f''(\Gamma, G)$.
\end{description}
The number $e(\Gamma, G)$ is given by
\begin{description}
  \item[{{\cite[Corollary p.~3220]{Byott96}}}\namedlabel{item:hgs-byott96}{\cite[Corollary p.~3220]{Byott96}}]
  \begin{equation}
    \label{eq:e_vs_e-prime}
    e(\Gamma,   G)  =   \frac{\Size{\Aut(\Gamma)}}{\Size{\Aut(G)}}  \,
    e'(\Gamma, G).
  \end{equation}
  Moreover $e(\Gamma)$, the total  number of Hopf-Galois structures on
  $L/K$, is given by $\sum_{G}e(\Gamma, G)$  where the sum is over all
  isomorphism types $G$ of groups of order $\Size{\Gamma}$.
\end{description}  
 \end{theorem}

In the paper~\cite{p2qciclico}, to the  introduction of which we refer
for   more   details   on    the   literature,   we   classified   the
\emph{Hopf-Galois  structures} on  a Galois  extension $L/K$  of order
$p^2q$, where $p$, $q$ are distinct  primes, $p$ is odd, and the Sylow
$p$-subgroups of $\Gamma = \Gal(L/K)$ are cyclic. In the same paper we
also computed, for $(G, \cdot)$ a group of order $p^{2} q$ with cyclic
Sylow  $p$-subgroups,  the \emph{number  of  skew  braces} $(G,  \cdot,
\circ)$, that  is, the number of  group operations ``$\circ$'' on  the set
$G$, such that  $(G, \cdot, \circ)$ is a skew  brace. We also computed
the   number  of \emph{isomorphism  classes} of  such skew  braces,
together with the cardinality of each such  class.

Acri e Bonatto  in~\cite{AcriBonatto_p2q_I, AcriBonatto_p2q_II}, using
a different method, determine  the number of \emph{isomorphism classes
  of \emph{all} skew braces} $(G, \cdot, \circ)$  of order $p^{2} q$; for the
case of groups $(G, \cdot)$ with cyclic Sylow 
$p$-subgroups, their results coincide with ours.

The  classification  of~\cite{p2qciclico}  has been  extended  to  all
groups  of   order  $p^{2}  q$  in   the  PhD  thesis  of   the  first
author~\cite{cam-thesis}.     The    present   paper  completes the work
of~\cite{p2qciclico}, by determining     the
classification of Hopf--Galois structures  on a Galois extension $L/K$
of order $p^2q$, where $p$, $q$  are distinct primes, in the remaining
cases  where $p > 2$  and  the Sylow  $p$-subgroups  of $\Gamma  =
\Gal(L/K)$  are elementary  abelian.   Our
methods work also in the case $p=2$ (see~\cite{cam-thesis}), but we do not
include  this case here,  since the  classification in  the case  $4
q$ has already appeared in~\cite{Ko07, SV2018}.

As  in~\cite{p2qciclico}, we explicitly determine  the number
of  Hopf--Galois  structures  on  $L/K$  for  each  type,  showing  in
particular that  each $G$ with elementary  abelian Sylow $p$-subgroups
defines some structure.

We accomplish this as follows. For any given group $G  = (G, \cdot)$ of
order   $p^{2} q$, with  $p > 2$, and for each group $\Gamma$ of order
$p^{2} q$ with elementary abelian Sylow $p$-subgroups, we determine
the following numbers.

\begin{enumerate}
  \item\label{en:1} The  total number of regular  subgroups of $\Hol(G)$
    isomorphic to $\Gamma$ (Theorem~\ref{number_regular}).
    \\
    This is the
    same, according to Theorem~\ref{th:hgs-regular-skew}~\ref{item:skew-GV17},
    as the total  number of  (right) skew braces  $(G, \cdot,  \circ)$ such
    that $\Gamma \cong (G,\circ)$.
  \item\label{en:2} The  number of  isomorphism classes of  (right) skew
    braces  $(G, \cdot,  \circ)$ such  that $\Gamma  \cong (G,
    \circ)$ (Theorem~\ref{th:number_skew}).
    \\
    This is the same, according to Theorem~\ref{th:hgs-regular-skew}~\ref{item:classes-BV18},
    as the number of  conjugacy  classes  in  $\Hol(G)$  of  regular  subgroups
    isomorphic to $\Gamma$; our numbers here coincide  with the numbers found by
    Acri e Bonatto in~\cite{AcriBonatto_p2q_II}.
    \\
    Additionally, in~Theorem~\ref{th:number_skew} we also
    determine the length of each such conjugacy class.
  \item The  number of Hopf--Galois structures  of type $G$ on  a Galois
    extension    with    Galois    group    isomorphic    to    $\Gamma$
    (Theorem~\ref{number_hgs}).
\end{enumerate}

\subsection{The methods}

As  in   our  previous  paper~\cite{p2qciclico},  we   follow  Byott's
approach, that is, for each group $G  = (G, \cdot)$ of order $p^{2} q$
with elementary abelian Sylow  $p$-subgroups, we determine the regular
subgroups of  $\Hol(G)$ isomorphic  to $\Gamma$.   As we  noted above,
this is in  turn equivalent to determining the right  skew braces $(G,
\cdot, \circ)$ such that $(G, \circ)\cong \Gamma$.

Our method relies on the use  of the alternate brace operation $\circ$
on $G$ mainly through the use of the function
\begin{align*}
  \gamma :\ &G  \to \Aut(G)\\
            &g \mapsto (x \mapsto  (x \circ g) \cdot g^{-1}),
\end{align*}
which is characterised by the functional equation
\begin{equation}
  \label{eq:GFE0}
  \gamma(g^{\gamma(h)} \cdot h) = \gamma(g) \gamma(h).
\end{equation}
(See \cite[Theorem 2.2]{p2qciclico} and the ensuing discussion for the
details.)   The  functions   $\gamma$  satisfying~\eqref{eq:GFE0}  are
called    \emph{gamma   functions    (GF)}   and    we   will    refer
to~\eqref{eq:GFE0} as the \emph{gamma  functional equation (GFE)}. The
GF's are  in one-to-one correspondence  with the regular  subgroups of
$\Hol(G)$,  and occur  naturally in  the  theory of  skew braces.   It
follows  that  to determine  the  number  $e'(\Gamma, G)$  defined  in
Definition~\ref{def:e-e'}
we can  count the number of
functions $\gamma :\ G \to \Aut(G)$ verifying \eqref{eq:GFE0} and such
that, for the operation $\circ$ defined on $G$ by
$$g\circ h=g^{\gamma(h)} h,$$ we have $(G,\circ)\cong\Gamma$.

The classification of the groups of order $p^2q$ is known after
H{\"o}lder \cite{Holder1893}; we have recorded the classification of these groups and of their automorphism
groups in \cite{classp2q}.

To enumerate the gamma functions we  use the general results listed in
Section~\ref{sec:tools}  below; some  of these  had been  developed in
\cite{p2qciclico}. In the course of  our discussion, we will appeal to
some ad-hoc  arguments; indeed,  counting the  gamma functions  in the
case of  groups with  elementary abelian Sylow  $p$-subgroups presents
several additional difficulties compared to the cyclic case.

The first is due to the sheer number of groups involved. Following the
notation  in Section~\ref{sec:the-groups},  we distinguish  the groups
into \textit{types} indexed  by the numbers 5,  6, 7, 8, 9,  10 and 11
(see   Table~\ref{table:grp_aut}).  Each   type   corresponds  to   an
isomorphism class of groups of order  $p^2q$, except for type 8, which
correspond  to $\frac{q-3}{2}$  isomorphism  classes  $G_k$, where  $k
\in\cC_q$, $k\ne 0,\pm 1$ and $G_k\simeq G_{k^{-1}}$.

Our analysis is further complicated by the difficulty of proving the
existence of a Sylow $q$-subgroup which is invariant under its image
under $\gamma$ (see
Subsections~\ref{ssec:G9_kerp}, \ref{ssec:G8_kerp},
\ref{ssec:G7_kerp-inn}, \ref{sub:G7-non-int}), and %
by having to deal
with the case when $\gamma(G)$ is not contained in the group of inner
automorphisms of $G$ (e.g. see Section~\ref{sec: 7}). 
\\

We now sketch the main tools and arguments we will use in counting the
gamma functions. 

We make use of  an argument of  duality, as introduced by A.~Koch and
P.J.~Truman in~\cite{KochTruman}, which we employ in the form spelled out
in~\cite{p2qciclico}. Each gamma function $\gamma$ can 
be  paired   with  a  gamma  function   $\gammatilde$,  which  defines
Hopf-Galois   structures   of    the   same   type   (\cite[Subsection
  2.8]{p2qciclico}).   Under suitable  assumptions,  we  can use  this
pairing to  halve the number of  GF we have to  consider. Moreover, in
some circumstances, the duality argument allows us to choose a GF with
a    kernel     that    is    more    suitable     for    calculations
(Lemma~\ref{lemma:duality} and Proposition~\ref{prop:duality}).

The theory developed in Section~\ref{sec:tools} offers some methods to
build gamma functions on $G$ piecewise. The first tool is
Proposition~\ref{prop:lifting}, which is sort of  a homomorphism theorem
for gamma functions. Under suitable assumptions, it gives a one-to-one
correspondence between certain gamma functions defined on $G$ and the
(relatives) gamma functions defined on a quotient of $G$, which is
smaller and then easier to investigate. We refer to this method as
\textit{lifting and restriction}. 

A further tool is Proposition~\ref{prop:more-lifting}
(\textit{gluing}), which is a generalisation of
Proposition~\ref{prop:lifting}, and describes a way to construct GF's
on $G$, when $G$ is of the form $G=AB$, for $A, B\leq G$, starting
from a \textit{relative gamma function} defined on $A$ (see Definition~\ref{def:GF}) and a relative gamma function defined on $B$. 
\\

Let $r\in\Set{p ,q}$. To apply the tools above it will be useful to know when there exists a Sylow $r$-subgroup $H$ which is $\gamma(H)$-invariant (\textit{invariant} for short).

In~\cite[Theorem 3.3]{p2qciclico} we prove that for $G$ a group of order $p^2q$ and $\gamma$ a GF on $G$, there always exists an invariant Sylow $p$-subgroup.

If $G$ is a group of order $p^2q$, then either $G$ has a unique Sylow $q$-subgroup or it has $p^{f}$ Sylow $q$-subgroups, where $f=1,2$.

In the first case, since the unique Sylow $q$-subgroup $B$ is characteristic, it is invariant. 

In the second case, there are $p$ Sylow $q$-subgroups when $G$ is of type 6, and $p^2$ when $G$ is of type 7, 8, 9, 10 and also %
4 (for the last one we refer to \cite{classp2q},\cite{p2qciclico}).

Let $\gamma$ be a GF in $G$, and consider the action of $\gamma(G)$ on the set $\mathcal{Q}$ of the Sylow $q$-subgroups of $G$. If $p^{2} \mid \Size{\ker(\gamma)}$, then $\Size{\gamma(G)}=1$ or $q$, so that there exists at least one orbit of length $1$, namely there exists $B \in \mathcal{Q}$ which is $\gamma(G)$-invariant.

Moreover, if $q \mid \Size{\ker(\gamma)}$, then there exists a Sylow $q$-subgroup $B$ contained in $\ker(\gamma)$, therefore it is $\gamma(B)$-invariant.

In the remaining cases, namely when $\Size{\ker(\gamma)}=1$ or $p$, we will prove for some specific type of group $G$ that we can find such a Sylow $q$-subgroup (see~\ref{subsub:G6-p-divides-ker}, \ref{ssec:q-sylow-G789}, \ref{ssec:G9-ker-1}, \ref{sub:G7-non-int}, and~\ref{ssec:G-10-ker-1}).
(For the type 4 see~\cite[Subsection 4.4]{p2qciclico}).

Therefore, we obtain the following.
\begin{prop}
\label{remark:invariant-q-Sylow}
If $G$ is a group of order $p^2q$ and $\gamma$ is a GF on $G$, then there always exist both an invariant Sylow $p$-subgroup and an invariant Sylow $q$-subgroup of $G$.
\end{prop}

\subsection{Hopf-Galois structures of order $p^2q$}

\begin{theorem}
  \label{number_hgs}
  Let $L/K$ be a Galois field extension of order $p^{2} q$, where $p$ and $q$ are two distinct primes with  $p > 2$, and let $\Gamma = \Gal(L/K)$.

  Let $G$ be a group of order $p^{2} q$.

  If the Sylow $p$-subgroups of $G$ and $\Gamma$ are not isomorphic, then there are no Hopf-Galois structures of type $G$ on $L/K$.

  If the Sylow $p$-subgroups of $\Gamma$ and $G$ are elementary abelian, then the numbers $e(\Gamma, G)$ of Hopf-Galois structures of type $G$ on $L/K$ are given in the following tables.
  \begin{enumerate}[(i)]
  \item For $q \nmid p-1$:
      \begin{center}
     \scalebox{0.9}{      
      \begin{tabular}{|c | c c|}
      \hline
        \diagbox[width=3.0em]{$\Gamma$}{$G$} & 5  & 11  \\
        \hline
        5 & $p^{2}$ & $2p(p^2-1)$ \\
        11 & $p^2q$ & $2p(1+qp^2-2q)$  \\
        \hline
      \end{tabular}
     } 
    \end{center}
    where the upper left sub-tables of sizes  $1\times 1$ and $2\times 2$ give
    respectively the cases $p \nmid q-1$ and  $p \mid q-1$.
\item For $q \nmid p-1$and  $q \mid p+1$: 
    \begin{center}
     \scalebox{0.9}{
      \begin{tabular}{|c|cc|}
      \hline
        \diagbox[width=3.0em]{$\Gamma$}{$G$}  & 5 & 10 \\
        \hline
        5 & $p^2$ & $p(p-1)(q-1)$ \\
       10 & $p^2$ & $2+2p^2(q-3)-p^3+p^4$ \\
        \hline
      \end{tabular}
     } 
   \end{center}               
  \item For $q \mid p-1$:\\
If $q=2$,
    \begin{center}
     \scalebox{0.9}{ 
      \begin{tabular}{|c | c c c|}
      \hline
        \diagbox[width=3.0em]{$\Gamma$}{$G$} & 5  & 6 & 7  \\
        \hline
        5 & $p^2$ & $2p(p+1)$   & $p(3p+1)$  \\
        6 & $p^2$ & $2p(p+1)$   & $p(3p+1)$  \\
        7 & $p^2$ & $2p^2(p+1)$ & $2+p(p+1)(2p-1)$  \\
        \hline
      \end{tabular}
     } 
    \end{center}
If $q=3$,
    \begin{center}
     \scalebox{0.9}{    
      \begin{tabular}{|c | c c c c|}
      \hline
        \diagbox[width=3.0em]{$\Gamma$}{$G$} & 5  & 6 & 7 & 9 \\
        \hline
        5 & $p^2$ & $4p(p+1)$ & $2p(3p+1)$ 			& $4p(p+1)$ \\
        6 & $p$ & $2p(p+3)$ & $4p(p+1)$ 		& $p(3p+5)$  \\
        7 & $p^2$ & $2p^2(p+1)^2$  & $2+p^2(2p^2+3p+2)$  & $p(p+1)^3$ \\
        9 & $p^2(2p-1)$ & $4p(p^2+1)$ & $2(2p^3+3p^2-2p+1)$ & $2+2p+p^3(p +3)$ \\
        \hline
      \end{tabular}
     } 
    \end{center}    
If $q>3$,
     \begin{center}
     \scalebox{0.9}{
      \begin{tabular}{|l | c c |}
      \hline
        \diagbox[width=7.0em]{$\Gamma$}{$G$} & 5  & 6  \\
        \hline
        5 					  		& $p^2$ & $2p(p+1)(q-1)$  \\
        6 					  		& $p$   & $2p(p+2q-3)$ \\
        7 					  		& $p^2$ & $2p^2(p+1)(pq-2p+1)$  \\
        8, $G_2$ 			  		& $p^3$ & $4p(p^2+pq-3p+1)$  \\        
        8, $G_k \not\simeq G_2$    	& $p^2$ & $4p(p^2+pq-3p+1)$  \\                
        9 					  		& $p^2$ & $4p(p^2+pq-3p+1)$  \\
        \hline
      \end{tabular}
     } 
    \end{center}
    \begin{center}
     \scalebox{0.9}{
      \begin{tabular}{|l | cc |}
      \hline
        \diagbox[width=5.9em]{$\Gamma$}{$G$} & 7  & 9 \\
        \hline
        5  						& $p(3p+1)(q-1)$ & $2p(p+1)(q-1)$  \\
        6 						& $4(p^2+pq-2p)$  & $p(4q+3p-7)$ \\
        7 						& $2+p^2(2p^2+pq+2q-4)$ & $p(p+1)(p^2(2q-5)+2p+1)$  \\
        8, $G_2$ 				& $2p(p^2q-4p+pq+2)$ & $p(p^3+3p^2-14p+4pq-6)$  \\
        8, $G_k \not\simeq G_2$ & $4p(2p^2-5p+pq+2)$ & $p(p^3+5p^2-18p+4pq+8)$ \\        
        9 						& $2(4p^3-9p^2+2p^2q+2p+1)$  & $2+4p+p^2(p^2+5p+4q-16)$ \\
        \hline
      \end{tabular}
      }
    \end{center} 
    \begin{center}
     \scalebox{0.84}{ 
      \begin{tabular}{|c | ccc|}
      \hline
        \diagbox[width=3.0em]{$\Gamma$}{$G$8} & $G\not\simeq G_{\pm 2}$  & $G \simeq G_{\pm 2}$, $q>5$ &  $G \simeq G_{2}$, $q=5$ \\
        \hline
        5 & $4p(p+1)(q-1)$ & $4p(p+1)(q-1)$ & $16p(p+1)$ \\
        6 & $8p(q+p-2)$ & $8p(q+p-2)$ & $8p(p+3)$ \\
        7 & $4p^2(p+1)(pq-3p+2)$  & $4p^2(p+1)(pq-3p+2)$ & $8p^2(p+1)^2$ \\
        8 & Table~\ref{table:G-8-G-circ-8-1} & Table~\ref{table:G-8-G-circ-8-2} & $4(1+p+3p^{2}(p+1))$ \\ 
        9 & $8p(2p^2+pq-5p+2))$ & $4p(3p^2+2pq-8p+3)$ &  $16p(2p^3-2p+p+1)$ \\
        \hline
      \end{tabular}
     } 
    \end{center}    
      \begin{table}[H]
    \caption{$G$ and $\Gamma$ of type 8, $G \simeq G_k \not\simeq G_{\pm 2}$ } 
	\label{table:G-8-G-circ-8-1}
    \centering
	\begin{center}
     \scalebox{0.88}{
      \begin{tabular}{|l | c |}
      \hline		
        \multirow{2}{*}{$\Gamma$} 
        & \multirow{2}{*}{if either $k$ or $k^{-1}$ is a solution of $x^2-x-1=0$:}    \\
		& \\
        \hline
        $G_{k}$, $G_{1-k}$ & $2(1+5p+4p^2q-17p^2+7p^3)$ \\
        $G_{1+k}$ & $4(3p+2p^2q-8p^2+3p^3)$  \\
        $G_s \not\simeq G_k, G_{1+k}, G_{1-k}$  & $8(2p+p^2q-5p^2+2p^3)$ \\
        \hline
        \multirow{2}{*}{$\Gamma$} & \multirow{2}{*}{if $k$ and $k^{-1}$ are the solutions of $x^2+x+1=0$:}    \\
		& \\
        \hline
        $G_{k}$ & $2(1+6p+4p^2q-19p^2+8p^3)$ \\
        $G_{1-k}$, $G_{1-k^{-1}}$ & $2(7p+4p^2q-18p^2+7p^3)$ \\
        $G_{1+k}$ & $2(1+4p+4p^2q-15p^2+6p^3)$  \\
        $G_{s} \not\simeq G_k, G_{1+k}, G_{1-k}, G_{1-k^{-1}}$  & $8(2p+p^2q-5p^2+2p^3)$ \\
        \hline        
        \multirow{2}{*}{$\Gamma$} & \multirow{2}{*}{if $k$ and $k^{-1}$ are the solutions of $x^2-x+1=0$: \phantom{b}}    \\
		& \\
        \hline
        $G_{-k}$ & $2(1+6p+4p^2q-19p^2+8p^3)$ \\
        $G_{1+k}$, $G_{1+k^{-1}}$ & $2(7p+4p^2q-18p^2+7p^3)$ \\
        $G_{1-k}$ & $2(1+4p+4p^2q-15p^2+6p^3)$  \\
        $G_{s} \not\simeq G_{-k}, G_{1-k}, G_{1+k}, G_{1+k^{-1}}$  & $8(2p+p^2q-5p^2+2p^3)$ \\
        \hline        
        \multirow{2}{*}{$\Gamma$} & \multirow{2}{*}{if $k$ and $k^{-1}$ are the solutions of $x^2+1=0$: \phantom{b}}    \\
		& \\
        \hline
        $G_{k}$ & $4(1+2p+2p^2q-9p^2+4p^3)$ \\
        $G_{1+k}$, $G_{1-k}$ & $4(3p+2p^2q-8p^2+3p^3)$ \\
        $G_{s} \not\simeq G_{k}, G_{1+k}, G_{1-k}$  & $8(2p+p^2q-5p^2+2p^3)$ \\
        \hline        
\multirow{2}{*}{$\Gamma$} & \multirow{2}{*}{if $k^2 \ne \pm k \pm 1, -1$:}  \\
		& \\
        \hline
        $G_k$, $G_{-k}$ & $2(1+6p+4p^2q-19p^2+8p^3)$ \\
        $G_{1+k}$, $G_{1+k^{-1}}$, $G_{1-k}$, $G_{1-k^{-1}}$ & $2(7p+4p^2q-18p^2+7p^3)$ \\
        $G_{s} \not\simeq G_{\pm k}, G_{1\pm k}, G_{1\pm k^{-1}}$ \phantom{ a}   
        & $8(2p+p^2q-5p^2+2p^3)$  \\
        \hline
      \end{tabular}
     } 
   \end{center} 
    \end{table}

   \begin{table}[H]
    \caption{$G$ and $\Gamma$ of type 8, $G \simeq G_k$ for $k=\pm 2$,} 
	\label{table:G-8-G-circ-8-2}
    \centering   
     \scalebox{0.9}{ 
      \begin{tabular}{|l | c |}
      \hline
		$\Gamma $ & if $q>7$: \\
        \hline
        $G_{2}$ & $2(1+5p+4p^2q-17p^2+7p^3)$ \\
        $G_{3}$, $G_{\frac{3}{2}}$ & $2(7p+4p^2q-18p^2+7p^3)$ \\
        $G_{-2}$ & $2(1+6p+4p^2q-19p^2+8p^3)$ \\
        $G_{s} \not\simeq G_2, G_3, G_{\frac{3}{2}}, G_{-2}$  & $8(2p+p^2q-5p^2+2p^3)$ \\
        \hline
		$\Gamma $ & if $q=7$: \\
        \hline
        $G_{2}$ & $2(1+5p+11p^2+7p^3)$ \\
        $G_{3}$ & $2(1+4p+13p^2+6p^3)$ \\
        \hline
      \end{tabular}
     } 
    \end{table}
  \end{enumerate}
\end{theorem}

\begin{theorem}
  \label{number_regular}
  Let $G = (G, \cdot)$ be a group of order $p^{2} q$, where $p, q$ are distinct primes, with $p > 2$.

  If $\Gamma$ is a group of order $p^{2} q$ and the Sylow $p$-subgroups of $G$ and $\Gamma$ are not isomorphic, then no regular subgroup of $\Hol(G)$ is isomorphic to $\Gamma$.
  
  If $G$ and $\Gamma$ have elementary abelian Sylow $p$-subgroups, then the following tables give equivalently
  \begin{enumerate}
  \item
    the number $e'(\Gamma, G)$ of regular subgroups of $\Hol(G)$ isomorphic to $\Gamma$;
  \item
    the number of (right) skew braces $(G,  \cdot, \circ)$ such that $\Gamma \cong (G, \circ)$.
  \end{enumerate}
 
  \begin{enumerate}[(i)]
  \item For $q \nmid p-1$:
    \begin{center}
     \scalebox{0.9}{      
      \begin{tabular}{|c | c c|}
      \hline
        \diagbox[width=3.0em]{$\Gamma$}{$G$} & 5  & 11  \\
        \hline
        5 & $p^{2}$ & $2pq$ \\
        11 & $p^{2}(p^{2}-1)$ & $2p(1+qp^2-2q)$  \\
        \hline
      \end{tabular}
     } 
    \end{center}
    where the upper left sub-tables of sizes  $1\times 1$ and $2\times 2$ give
    respectively the cases $p \nmid q-1$ and  $p \mid q-1$.
  \item For $q \nmid p-1$ and $q \mid p+1$: 
    \begin{center}
     \scalebox{0.9}{      
      \begin{tabular}{|c|cc|}
      \hline
        \diagbox[width=3.0em]{$\Gamma$}{$G$}  & 5 & 10 \\
        \hline
        5 & $p^2$ & $2p^2$ \\
       10 & $\frac{1}{2}p(p-1)(q-1)$ & $2+2p^2(q-3)-p^3+p^4$ \\
        \hline
      \end{tabular}
     } 
    \end{center}                 
  \item For $q \mid p-1$:\\
If $q=2$,
    \begin{center}
     \scalebox{0.9}{      
      \begin{tabular}{|c | c c c|}
      \hline
        \diagbox[width=3.0em]{$\Gamma$}{$G$} & 5  & 6 & 7  \\
        \hline
        5 & $p^2$ & $2p$ & $p^3(3p+1)$  \\
        6 & $p^2(p+1)$ & $2p(p+1)$ & $p^3(p+1)(3p+1)$  \\
        7 & $1$ & $2$ & $2+p(p+1)(2p-1)$  \\
        \hline
      \end{tabular}
	 }
    \end{center}
If $q=3$,
    \begin{center}
     \scalebox{0.89}{      
      \begin{tabular}{|c | c c c c |}
      \hline
        \diagbox[width=3.0em]{$\Gamma$}{$G$} & 5  & 6 & 7 & 9 \\
        \hline
        5 & $p^2$ & $2p$ & $p^3(3p+1)$ & $4p^2$  \\
        6 & $2p(p+1)$ & $2p(p+3)$ & $4p^3(p+1)^2$ & $2p^2(3p+5)$ \\
        7 & $2$ & $2(p+1)$ & $2+p^2(2p^2+3p+2)$ & $2p^2+4p+2$  \\
        9 & $2p^3+p^2-p$ & $2(p^2+1)$ & $2p^5+5p^4+p^3-p^2+p$ & $p^4+3p^3+2p+2$ \\
        \hline
      \end{tabular}
     } 
    \end{center}
If $q>3$, 
    \begin{center}
     \scalebox{0.9}{      
      \begin{tabular}{|l | c c |}
      \hline
        \diagbox[width=7.0em]{$\Gamma$}{$G$} & 5  & 7  \\
        \hline
        5 					  		& $p^2$ & $p^3(3p+1)$  \\
        6 					  		& $p(p+1)(q-1)$ & $4p^2(p+1)(p^2+pq-2p)$  \\
        7 					  		& $q-1$ & $2+p^2(2p^2+pq+2q-4)$  \\
        8, $G_2$ 			  		& $p^2(p+1)(q-1)$ & $2p^2(p+1)(p^2q-4p+pq+2)$  \\        
        8, $G_k \not\simeq G_2$    	& $p(p+1)(q-1)$ & $4p^2(p+1)(2p^2-5p+pq+2)$  \\                
        9 					  		& $\frac{1}{2}p(p+1)(q-1)$ & $4p^5+p^4(q-2)+p^3(2q-7)+3p^2+p$  \\
        \hline
      \end{tabular}
	 }
    \end{center}
    \begin{center}
     \scalebox{0.9}{      
      \begin{tabular}{|l | cc |}
      \hline
        \diagbox[width=7.0em]{$\Gamma$}{$G$} & 6 & 9 \\
        \hline
        5  						& $2p$ & $4p^2$  \\
        6 						& $2p(p+2q-3)$ & $2p^2(4q+3p-7)$  \\
        7 						& $2+2p(q-2)$ & $2+4p+2p^2(2q-5)$ \\
        8, $G_2$ 				& $4(1+p(p+q-3))$ & $2p(p^3+3p^2-14p+4pq-6)$ \\
        8, $G_k \not\simeq G_2$ & $4(1+p(p+q-3))$ & $2p(p^3+5p^2-18p+4pq+8)$ \\        
        9 						& $2+2p(p+q-3)$ & $2+4p+p^2(p^2+5p+4q-16)$ \\
        \hline
      \end{tabular}
     } 
    \end{center}        
    \begin{center}
     \scalebox{0.84}{      
      \begin{tabular}{|c | ccc|}
      \hline
        \diagbox[width=3.0em]{$\Gamma$}{$G$8} & $G\not\simeq G_{\pm 2}$  & $G \simeq G_{\pm 2}$, $q>5$ &  $G \simeq G_{2}$, $q=5$ \\
        \hline
        5 & $4p^2$ & $4p^2$ & $4p^2$ \\
        6 & $8p^2(q+p-2)$ & $8p^2(q+p-2)$ & $8p^2(p+3)$ \\
        7 & $8p+4p^2(q-3)$  & $8p+4p^2(q-3)$ & $8p+8p^2$ \\
        8 & Table~\ref{table:G-8-G-circ-8-1} & Table~\ref{table:G-8-G-circ-8-2} & $4(1+p+3p^{2}(p+1))$ \\ 
        9 & $4p(2+p(q+2p-5))$ & $2p(3+p(2q+3p-8))$ &  $8p(1+p+2p(p^2-1))$ \\
        \hline
      \end{tabular}
     } 
    \end{center}
  \end{enumerate}
\end{theorem}

As a consequence, we are able to compute the numbers of isomorphism
classes of skew braces of size $p^2q$; these numbers coincide with those
given in~\cite{AcriBonatto_p2q_I, AcriBonatto_p2q_II} for  $q>2$,
and~\cite{Cre2021} for $q=2$.

\begin{theorem}  
\label{th:number_skew}
  Let $G = (G, \cdot)$ be a group of order $p^{2} q$, where $p, q$ are
  distinct primes, with $p > 2$.
  For each group $\Gamma$ of order $p^{2} q$ with elementary abelian Sylow $p$-subgroups the following tables give equivalently
  \begin{enumerate}
  \item
    the number of conjugacy classes within $\Hol(G)$ of regular
    subgroups isomorphic to $\Gamma$;
  \item
    the number of isomorphism classes of skew braces $(G, \cdot,
    \circ)$ such that $\Gamma \cong (G, \circ)$.
  \end{enumerate} 
  \begin{enumerate}[(i)]
  \item For $q \nmid p-1$:
    \begin{center}         
     \scalebox{0.9}{      
      \begin{tabular}{|c | c c |}
        \hline
        \diagbox[width=3.0em]{$\Gamma$}{$G$} & 5 & 11  \\
        \hline
        5 & $2$ & $4$  \\
        11 & $4$ & $6p-4$  \\
        \hline
       \end{tabular}
      } 
    \end{center}
    where the upper left sub-tables of sizes  $1\times 1$ and $2\times 2$ give
    respectively the cases $p \nmid q-1$ and  $p \mid q-1$.
\item For $q \nmid p-1$ and $q \mid p+1$: 
    \begin{center}
     \scalebox{0.9}{      
      \begin{tabular}{|c|cc|}
      \hline
        \diagbox[width=3.0em]{$\Gamma$}{$G$}  & 5 & 10 \\
        \hline
        5 & $2$ & $2$ \\
       10 & $1$ & $p+2q-4$ \\
        \hline
      \end{tabular}
     } 
    \end{center} 
\item For $q \mid p-1$: \\
If $q=2$,
    \begin{center}
     \scalebox{0.9}{      
      \begin{tabular}{|c | c c c|}
      \hline
        \diagbox[width=3.0em]{$\Gamma$}{$G$} & 5  & 6 & 7  \\
        \hline
        5 & $2$ & $2$ & $5$  \\
        6 & $2$ & $8$ & $p+10$  \\
        7 & $1$ & $2$ & $5$  \\
        \hline
      \end{tabular}
     } 
    \end{center} 
If $q=3$,
    \begin{center}
     \scalebox{0.9}{      
      \begin{tabular}{|c | c c c c|}
      \hline
        \diagbox[width=3.0em]{$\Gamma$}{$G$} & 5  & 6 & 7 &  9 \\
        \hline
        5 & $2$ & $2$ & $5$ & $3$ \\
        6 & $1$ & $12$ & $16$ & $p+14$  \\
        7 & $1$ & $4$ & $8$ & $4$  \\
        9 & $2$ & $6$ & $10$ & $p+8$ \\
        \hline
      \end{tabular}
     } 
    \end{center}
If $q>3$,
    \begin{center}
     \scalebox{0.9}{      
      \begin{tabular}{|l | c c c c c|}
      \hline
        \diagbox[width=7.0em]{$\Gamma$}{$G$} & 5  & 6 & 7 & 8, $G_k$ & 9 \\
        \hline
        5 						& $2$ & $2$ & $5$ & $4$ & $3$ \\
        6 						& $1$ & $4q$ & $4(q+1)$ & $8(q+1)$ & $4q+p+2$ \\
        7 						& $1$ & $2(q-1)$ & $3q-1$ & $4(q-1)$ & $2(q-1)$ \\
        8, $G_s\not\simeq G_2$  & $1$ & $4q$ & $4(q+1)$ & $8(q+1)$ & $4q+p+2$ \\
        8, $G_2$ 				& $2$ & $4q$ & $6q$ & $8(q+1)$ & $4q+p+2$ \\
        9 						& $1$ & $2q$ & $2(q+1)$ & $4(q+1)$ & $3q+p-1$ \\
        \hline
      \end{tabular}
     } 
    \end{center}  
  \end{enumerate}
\end{theorem}  
The lengths of the conjugacy classes are spelled out in Propositions~\ref{prop:G5}, \ref{prop:G6}, \ref{prop:G9}, \ref{prop:G8}, \ref{prop:G7}, \ref{prop:G10} and \ref{prop:G11}.

\section{Tools}
\label{sec:tools}

\begin{definition}
  \label{def:GF}
  Let  $G$ be  a group,  $A  \le G$,  and  $\gamma: A  \to \Aut(G)$  a
  function.

  $\gamma$ is said to satisfy the \emph{gamma functional equation} (or
  \emph{GFE} for short) if
  \begin{equation*}
    \gamma( g^{\gamma(h)} h )  = \gamma(g) \gamma(h), \qquad \text{for
      all $g, h \in A$.}
  \end{equation*}

  We will say that $A$ is \emph{invariant} if it is invariant under
  the action of $\gamma(A)$.
  
  $\gamma$  is  said  to  be  a  \emph{relative  gamma  function}  (or
  \emph{RGF} for  short) on $A$  if it satisfies the  gamma functional
  equation, and $A$ is $\gamma(A)$-invariant.

  If $A = G$, a relative gamma function is simply called a \emph{gamma
    function} (or \emph{GF} for short) on $G$.
\end{definition}

We will make use of the following results from~\cite{p2qciclico}; some
of  these   results  were  stated  in   ~\cite{p2qciclico}  under  the
assumption that the relevant group $G$ is finite, but the proofs stand
verbatim for arbitrary groups.

\begin{lemma}[{{\cite[Lemma 2.13]{p2qciclico}}}]\ppar
  \label{Lemma:gamma_morfismi}
  Let $G$ be a  group, $A \le G$ and $\gamma : A \to \Aut(G)$ be a
  function such that $A$ is invariant under $\gamma(A)$.

  Then any two of the following conditions imply the third one.
  \begin{enumerate}
  \item $\gamma([A, \gamma(A)]) = \Set{1}$.
  \item $\gamma : A \to \Aut(G)$ is a morphism of groups.
  \item $\gamma$ satisfies the GFE.
  \end{enumerate}
\end{lemma}

\begin{prop}[{{\cite[Proposition 2.14]{p2qciclico}}}]\ppar
  \label{prop:lifting}
  Let $G$ be a  group and  let $A$, $B$ be subgroups of $G$ such
  that    $G = A B$.

  If $\gamma$ is a GF on $G$, and $B \le \ker(\gamma)$, then
  \begin{equation}    
    \gamma(a b) = \gamma(a), \text{ for } a \in A, b \in B,
  \end{equation}
  so that $\gamma(G) = \gamma(A)$.
  
  Moreover, if $A$  is  $\gamma(A)$-invariant, then 
  \begin{equation}
    \gamma'=\gamma_{\restriction A}\colon A\to\Aut(G)
  \end{equation}  
  is a RGF on $A$ and  $\ker(\gamma)$  is  invariant  under  the
  subgroup   
  \begin{equation*}
    \Set{ \gamma'(a) \iota(a) : a \in A }
  \end{equation*}
  of   $\Aut(G)$.

  Conversely, let $\gamma':A\to \Aut(G)$ be a RGF such that
  \begin{enumerate}
  \item
    $\gamma'(A \cap B) \equiv 1$,
  \item
    $B$ is invariant under $\{ \gamma'(a) \iota(a) : a \in A \}$.
  \end{enumerate}
  Then the map 
  \begin{equation*}
    \gamma(ab) = \gamma'(a), \text{ for } a \in A, b \in B,
  \end{equation*}
  is a well defined GF on $G$, and $\ker(\gamma) = \ker(\gamma') B$.
\end{prop}
In this situation we will say that $\gamma$ is a \emph{lifting} of
$\gamma'$.

\begin{lemma}[{{\cite[Lemma 2.23]{p2qciclico}}}]\ppar
  \label{lemma:duality}
  Let $G$ be a non-abelian group.
  Let $C$ be a non-trivial subgroup of $G$ such that:
  \begin{enumerate}
  \item
    $C$ is abelian;
  \item
    $C$ is characteristic in $G$;
  \item $C\cap Z(G)=\Set{1}$.
  \end{enumerate}
  Let $\gamma\colon G\to\Aut(G)$  be a GF, and suppose  that for every
  $c\in C$  we have  $\gamma(c)=\iota(c^{-\sigma})$ for  some function
  $\sigma\colon C\to C$.  

  Then  $\sigma\in\End(C)$, and the following relations hold in $\End(C)$:
  \begin{equation}
    \label{eq:sigma}
    \sigma \, \gamma(g)_{\restriction C} \, (\sigma - 1)
    =
    (\sigma - 1) \, \gamma(g)_{\restriction C} \, \iota(g)_{\restriction
      C} \, \sigma,
    \quad\text{for $g \in G$.}
  \end{equation}
\end{lemma}

Note that $\gamma(g) \iota(g) = \iota(g^{\gamma(g)^{-1}})
\gamma(g)$. Setting $g' = g^{\ominus 1} = g^{-\gamma(g)^{-1}}$, we see
that $\gamma(g) \iota(g) = \iota(g')^{-1}
\gamma(g')^{-1}$. Therefore~\eqref{eq:sigma} can be rewritten as
\begin{equation}
  \label{eq:sigmaprime}
  \sigma \, \gamma(g)_{\restriction C}^{-1} \, (\sigma - 1)
  =
  (\sigma - 1) \,
  \iota(g)_{\restriction C}^{-1} \, \gamma(g)_{\restriction C}^{-1} \,  \sigma,
  \quad\text{for $g \in G$.}
\end{equation}

Lemma~\ref{lemma:duality} admits a converse, which partly extends Proposition~\ref{prop:lifting}. 

\begin{proposition}
  \label{prop:more-lifting}
Let $G = A B$ be a group, where $A, B$ are subgroups of $G$, such that
 \begin{enumerate}
  \item $A \cap B = \Set{1}$,
  \item $A$ is abelian,
  \item $A$ is characteristic in $G$,
  \item $A \cap Z(G) = \Set{1}$.
 \end{enumerate}
 
If moreover there exists a RGF $\gamma:B \to \Aut(G)$ and $\sigma \in \End(A)$ which satisfy~\eqref{eq:sigma}, 
then the extension of $\gamma$ to the function $\gamma : G \to \Aut(G)$ defined by, for $a \in A$ and $b \in B$,
 \begin{equation}
  \label{eq:formula}
    \gamma(a b) =  \iota(a^{-\gamma(b)^{-1} \sigma}) \gamma(b)
  \end{equation} 
is a GF on $G$.

Conversely, every GF $\gamma$ on $G$ such that $\gamma(a)=\iota(a^{-\sigma})$, where $\sigma\in \End(A)$, 
and for which $B$ is $\gamma(B)$-invariant, is obtained as the extension of a RGF on $B$ as in~\eqref{eq:formula}.
\end{proposition}

In this situation we will say  that $\gamma$ is the \emph{gluing} of a
RGF on  $A$ and  a RGF on  the $\gamma(B)$-invariant  
subgroup $B$.

\begin{proof}
Assume   first  that   $\gamma$   is  a   GF  on   $G$   and  $B$   is
$\gamma(B)$-invariant. Then  $\gamma_{|B}$ is RGF  on $B$, and  for $a
\in A$ and $b\in B$, $ \gamma(a b)=\gamma(a^{\gamma(b)^{-1}})\gamma(b)
$.   By Lemma~\ref{lemma:duality}  there exists  $\sigma \in  \End(A)$
such   that  $\gamma(a)=\iota(a^{-\sigma})$   for   $a   \in  A$;   we
obtain~\eqref{eq:formula}.

Conversely, let $\gamma$ be a RGF on $B$ and $\sigma \in \End(A)$ such that~\eqref{eq:sigma} is satisfied. We now show that the function defined in~\eqref{eq:formula} satisfies the GFE. Let $a_{1}, a_{2} \in A$ and $b_{1}, b_{2} \in B$.

We have
  \begin{align*}
     \gamma(a_{1} b_{1}) \gamma(a_{2} b_{2}) 
    &=
    \iota(a_{1}^{-\gamma(b_{1})^{-1} \sigma}) \gamma(b_{1})
    \iota(a_{2}^{- \gamma(b_{2})^{-1} \sigma}) \gamma(b_{2}) \\
    &=
    \iota(a_{1}^{-\gamma(b_{1})^{-1} \sigma}
    a_{2}^{-\gamma(b_{2})^{-1} \sigma\gamma(b_{1}^{-1})} )
    \gamma(b_{1}) \gamma(b_{2}).
  \end{align*}
  On the other hand
  \begin{align*}
    & \gamma(
    ( a_{1} b_{1} )^{\gamma( a_{2} b_{2})} a_{2} b_{2} ) =\\
    &
    \gamma(
    a_{1}^{\gamma(b_{2})}
    b_{1}^{ \iota(a_{2}^{-\gamma(b_{2})^{-1} \sigma})  \gamma(b_{2}) }
    a_{2} b_{2}
    ) =\\
    &
    \gamma(
    a_{1}^{\gamma(b_{2})}
    (
    a_{2}^{\gamma(b_{2})^{-1} \sigma}
    b_{1}
    a_{2}^{-\gamma(b_{2})^{-1} \sigma}
    )^{\gamma(b_{2}) }
    a_{2} b_{2}
    ) =\\
    &
    \gamma(
    a_{1}^{\gamma(b_{2})}
    a_{2}^{\gamma(b_{2})^{-1} \sigma \gamma(b_{2})}
    b_{1}^{\gamma(b_{2}) }
    a_{2}^{-\gamma(b_{2})^{-1} \sigma \gamma(b_{2})}
    a_{2} b_{2}
    ) =\\
    &
    \gamma(
    a_{1}^{\gamma(b_{2})}
    a_{2}^{\gamma(b_{2})^{-1} \sigma \gamma(b_{2})}
    a_{2}^{-\gamma(b_{2})^{-1} \sigma \iota(b_{1})^{-1} \gamma(b_{2})}
    a_{2}^{\gamma(b_{2})^{-1} \iota(b_{1})^{-1} \gamma(b_{2}) }
    b_{1}^{\gamma(b_{2}) } b_{2}
    ) =\\
    &
    \iota(
    (
    a_{1}^{\gamma(b_{2})}
    a_{2}^{\gamma(b_{2})^{-1} \sigma \gamma(b_{2})}
    a_{2}^{-\gamma(b_{2})^{-1} \sigma \iota(b_{1})^{-1} \gamma(b_{2})}
    a_{2}^{\gamma(b_{2})^{-1} \iota(b_{1})^{-1} \gamma(b_{2}) }
    )^{- \gamma(b_{2})^{-1} \gamma(b_{1})^{-1} \sigma }
    ) \\
    &\qquad
    \gamma(b_{1}) \gamma(b_{2}) =\\ 
    &
    \iota(
    a_{1}^{- \gamma(b_{1})^{-1} \sigma }
    a_{2}^
    {
      -\gamma(b_{2})^{-1} \sigma \gamma(b_{1})^{-1} \sigma
      + \gamma(b_{2})^{-1} \sigma \iota(b_{1})^{-1} \gamma(b_{1})^{-1} \sigma
      - \gamma(b_{2})^{-1} \iota(b_{1})^{-1} \gamma(b_{1})^{-1} \sigma
    }
    ) \\
    &\qquad
    \gamma(b_{1}) \gamma(b_{2}).
  \end{align*}
  Now~\eqref{eq:sigmaprime} shows that the two expressions
  \begin{equation*}
    - \sigma \gamma(b_{1})^{-1}_{\restriction A}
  \end{equation*}
  and
  \begin{equation*}
    -\sigma \gamma(b_{1})^{-1}_{\restriction A} \sigma
    + \sigma \iota(b_{1})^{-1}_{\restriction A}
    \gamma(b_{1})^{-1}_{\restriction A} \sigma 
    - \iota(b_{1})^{-1}_{\restriction A} \gamma(b_{1})^{-1}_{\restriction A} \sigma
  \end{equation*}
  coincide.
\end{proof}

\begin{proposition}[{{\cite[Proposition 2.24]{p2qciclico}}}]\ppar
  \label{prop:duality}
  Let $G$ be a finite non-abelian group.
  Let $C$ be a subgroup of $G$ such that:
  \begin{enumerate}
  \item 
    $C = \Span{c}$ is cyclic, of order a power of the prime $r$,
  \item \label{i2_duality}
    $C$ is characteristic in $G$,
  \item 
    $C \cap Z(G) = \Set{1}$, and
  \item 
    there  is  $a \in  G$  which  induces  by  conjugation on  $C$  an
    automorphism whose order is not a power of $r$.
  \end{enumerate}
  Let $\gamma\colon G\to\Aut(G)$  be a GF, and suppose  that for every
  $c\in C$  we have  $\gamma(c) = \iota(c^{-\sigma})$, for  some function
  $\sigma \colon C \to C$.  

  Then
  \begin{enumerate}
  \item
    either $\sigma = 0$, that is, $C \le \ker(\gamma)$, 
  \item
    or $\sigma = 1$, that is, $\gamma(c) = \iota(c^{-1})$, so that $C
    \le \ker(\tilde\gamma)$.
  \end{enumerate}
\end{proposition}

\begin{remark}
\label{remark:duality-G-7}
In Proposition~\ref{prop:duality}, once $\gamma$ has been chosen, we can replace the hypothesis~\eqref{i2_duality} with the slightly more general hypothesis that $C$ is normal and $\gamma(G)$-invariant.
In fact, in that case $\gamma(g)_{\restriction C} \in \Aut(C)$. The proof of Lemma~\ref{lemma:duality} still stands, as $c^{\iota(g)}, c^{\gamma(g)} \in C$, so that equation~\eqref{eq:sigma} is satisfied.
\end{remark}

\begin{lemma}[{{\cite[Lemma 2.9]{p2qciclico}}}]\ppar
  \label{lemma:conjugacy}
  Let $G$ be a group, $N$ a regular subgroup of $\Hol(G)$, and
  $\gamma$ the associated gamma function.

  Let $\phi \in \Aut(G)$.
  \begin{enumerate}
  \item 
    The gamma function $\gamma^{\phi}$ associated to the regular subgroup
    $N^{\phi}$  is given by
    \begin{equation}
      \label{eq:conjugacy}
      \gamma^{\phi}(g)
      =
      \gamma(g^{\phi^{-1}})^{\phi}
      =
      \phi^{-1} \gamma(g^{\phi^{-1}}) \phi,
    \end{equation}
    for $g \in G$.
  \item
    \label{item:invariance-under-conjugacy}
    If $H \le G$ is invariant under $\gamma(H)$, then $H^{\phi}$ is
    invariant under $\gamma^{\phi}(H^{\phi})$.
  \end{enumerate}  
\end{lemma}

We will refer to the action~\eqref{eq:conjugacy} of $\Aut(G)$ on
$\gamma$ of the Lemma as 
\emph{conjugation}.

\begin{lemma}
\label{lemma:generatori}
Let $G$ be a group of order $p^2q$, $p>2$, and assume that the Sylow $p$-subgroup $A$ of $G$ is normal. Let $b \in G$ an element of order $q$.

\begin{enumerate} 
\item If $A=\Span{a}$ is cyclic, then $\Set{a, b}$ is a set of generators for both $G$ and $(G, \circ)$, for each possible operation $\circ$ on $G$.
\item If $A=\Span{a_1, a_2}$ is elementary abelian and $\gamma$ is a GF on $G$ such that $\Span{a_1}$ is $\gamma(\Span{a_1})$-invariant, then $\Set{a_1, a_2, b}$ is a set of generators for both $G$ and $(G, \circ)$.
\end{enumerate}
\end{lemma}

\begin{proof}
Clearly the generator(s) of $A$ together with the element $b$ generate $G$. 

Let $\gamma$ be a GF on $G$ and let $\circ$ be the corresponding operation on $G$. 
Since $A$ is $\gamma(A)$-invariant, $A$ is a subgroup of $(G,\circ)$, and since $p>2$, $A\simeq (A, \circ)$ (see \cite[Theorem 3.3]{p2qciclico}). 

If $A=\Span{a}$ is cyclic then $\ord_{A}(a)=\ord_{(A,\circ)}(a)$ (take $\gamma_{|A}$ in \cite[Corollary 2.18]{p2qciclico}), therefore $a$ generates $(A, \circ)$ too. 

If $A$ is elementary abelian then every non-trivial element of $(A,\circ)$ has order $p$. Moreover if $A_1:=\Span{a_1}$ is $\gamma(A_1)$-invariant then $a_2 \not\in A_1=(A_1, \circ)$, so that $a_1, a_2$ generate $(A, \circ)$ too.

Now, since $b \in G\setminus A$ and $[G:A]=q$, 
the generator(s) of $A$ together with the element $b$ generate also $(G,\circ)$.
\end{proof}


\section{Groups of order \texorpdfstring{$p^2q$}{p2q}}
\label{sec:the-groups}

We briefly describe  the groups of order $p^{2} q$,  and list them and
their automorphisms  in the table below,  referring to~\cite{classp2q}
for the details.

We will  say that two  groups have the  same \emph{type} if  they have
isomorphic automorphism  groups.  For groups  of order $p^{2}  q$ each
type corresponds  to an  isomorphism class, except  for type  8, which
corresponds to $\frac{q-3}{2}$ isomorphism classes.

We use the notation $\cC_{n}$ for a cyclic group of order $n$.

\begin{description}
\item[Type 5] Abelian group.
\item[Type 6] This is the non-abelian group with centre of order $p$ for $q \mid p - 1$, which we denote  by $\cC_{p} \times (\cC_{p} \rtimes \cC_{q})$.  It can be described as
$$\Span{a_1, a_2, b : a_1^p = a_2^p = b^q = 1,\ a_2^{\iota(b)} = a_2^{\lambda} },$$
where $\lambda$ is an element of order $q$ in $\cC_p^{\ast}$, $\lambda \ne 1$. 
\item[Type 7] This is the non-abelian group for $q\mid p-1$ in which a generator of $\cC_q$ acts on $\cC_p \times \cC_p$ as a non-identity scalar matrix. We denote it by $(\cC_{p} \times \cC_{p}) \rtimes_{S} \cC_{q}$, and it can be described as
$$\Span{a_1, a_2, b : a_1^p = a_2^p = b^q = 1,\ a_1^{\iota(b)} = a_1^{\lambda},\ a_2^{\iota(b)} = a_2^{\lambda} },$$
where $\lambda$ is an element of order $q$ in $\cC_p^{\ast}$, $\lambda \ne 1$. 
\item[Type 8] These are the non-abelian groups for $q\mid p-1$, $q>3$, in which a generator of $\cC_q$ acts on $\cC_p \times \cC_p$ as a diagonal, non-scalar matrix with no eigenvalue $1$, and determinant different from $1$. 
We denote this type by $(\cC_{p} \times \cC_{p}) \rtimes_{D1} \cC_{q}$, and it consists of the groups $G_k$ which can be described as
$$G_k = \Span{a_1, a_2, b : a_1^p = a_2^p = b^q = 1,\ a_1^{\iota(b)} = a_1^{\lambda},\ a_2^{\iota(b)}=a_2^{\lambda^{k}} }, $$
where $\lambda$ is an element of order $q$, $\lambda \ne 1$, and $k$ is an integer modulo $q$, $k \ne 0, \pm 1$. 

Since for each $k \ne 0, \pm 1$ we have that $G_{k} \simeq G_{k^{-1}}$, the type 8 includes $\frac{q-3}{2}$ isomorphism classes of groups.

We will denote by $\mathcal{K}$ the set of the elements $k \ne 0, \pm 1$ for which $\Set{G_k : k \in \mathcal{K}}$ is a set of representatives of the isomorphism classes of groups of type 8.

\item[Type 9] This is the non-abelian group for $q\mid p-1$, $q>2$, in which a generator of $\cC_q$ acts on $\cC_p \times \cC_p$ as a diagonal, non-scalar matrix with no eigenvalue $1$, and determinant $1$. We denote it by $(\cC_{p} \times \cC_{p}) \rtimes_{D1} \cC_{q}$, and it can be described as
$$\Span{a_1, a_2, b : a_1^p = a_2^p = b^q = 1,\ a_1^{\iota(b)} = a_1^{\lambda},\ a_2^{\iota(b)}=a_2^{\lambda^{-1}} }, $$
where $\lambda$ is an element of order $q$ in $\cC_p^{\ast}$, $\lambda \ne 1$.

\item[Type 10] This is the non-abelian group group for $q \mid p+1$, $q>2$, in which a generator of $\cC_q$ acts on $\cC_p \times \cC_p$ as a matrix $C$ with $\det(C)=1$ and $\tr(C)=\lambda + \lambda^{-1}$, where $\lambda\ne 1$ is a $q$-th root of unity in a quadratic extension of $\mathbb{F}_p$.
We denote it by $(\cC_{p} \times \cC_{p}) \rtimes_{C} \cC_{q}$, and it can be described as
$$\Span{a_1, a_2, b : a_1^p = a_2^p = b^q = 1,\ a_1^{\iota(b)}=a_1^{\lambda+\lambda^{-1}}a_2,\ a_2^{\iota(b)}= a_1^{-1}} .$$
\item[Type 11] This is the non-abelian group with centre of order $p$ for $p \mid q - 1$, which we denote  by $(\cC_{q} \rtimes \cC_{p}) \times \cC_{p}$. It can be described as
$$\Span{a_1, a_2 , b : a_1^p = a_2^p = b^q = 1,\ b^{\iota(a_1)}=b^{u}},$$ 
where $u$ is an element of order $p$ in $\cC_{q}^{\ast}$.
\end{description}

\begin{center}
  \begin{table}[H]
    \caption{Groups of order $p^2q$ and their automorphisms} 
    \label{table:grp_aut}
    \centering
    \begin{tabular}{|c|c|c|c|}
    \hline
    Type & Conditions & $G$ & $\Aut(G)$  \\
    \hline
    5 & & $\cC_{p} \times \cC_{p} \times \cC_{q}$ & $\GL(2, p) \times \cC_{q-1}$  \\
    6 & $q \mid p - 1$ & $\cC_{p} \times (\cC_{p} \rtimes \cC_{q})$ & $\cC_{p-1} \times \Hol(\cC_{p})$ \\
    7 & $q \mid p - 1$ & $(\cC_{p} \times \cC_{p}) \rtimes_{S} \cC_{q}$ & $\Hol(\cC_{p} \times \cC_{p})$ \\
    8 & $3 < q \mid p - 1$ & $(\cC_{p} \times \cC_{p}) \rtimes_{D0} \cC_{q}$ & $\Hol(\cC_{p}) \times \Hol(\cC_{p})$ \\
    9 & $2 < q \mid p - 1$ & $(\cC_{p} \times \cC_{p}) \rtimes_{D1} \cC_{q}$ & $ (\Hol(\cC_{p}) \times \Hol(\cC_{p})) \rtimes  \cC_{2} $ \\
    10 & $2 < q \mid p + 1$ & $(\cC_{p} \times \cC_{p}) \rtimes_{C} \cC_{q}$ & $(\cC_{p} \times \cC_{p})  
    \rtimes (\cC_{p^{2}-1} \rtimes \cC_{2}) $ \\
    11 & $p \mid q - 1$ & $(\cC_{q} \rtimes \cC_{p}) \times \cC_{p}$ & $\Hol(\cC_{p}) \times \Hol(\cC_{q})$ \\
    \hline
    \end{tabular}
  \end{table}
\end{center}


\section{The main case distinction}
\label{sec:case-distinction}

In this section we spell out the case distinction we will pursue in the
following sections, and collect a few  facts that will be
useful at several points in the classification.

For type 6, we can apply Lemma~\ref{lemma:duality}~and
Proposition~\ref{prop:duality} to a non-central subgroup of order
$p$. 
For type 11, we can apply these to the Sylow $q$-subgroup $C$.

For $G$ of the  remaining types, namely 7, 8, 9 and  10, denote by $A$
the elementary abelian Sylow $p$-subgroup of  $G$.  We will show in
Sections~\ref{sec: 7,8,9} and~\ref{sec : 10} that for the types 8, 9 and
10  (when  $p>2$) one has
\begin{equation}
\label{eq: RGF-on-A-by-sigma}
\forall a \in A, \  \gamma(a)=\iota(a^{-\sigma}) ,
\end{equation}
for some $\sigma \in \End(A)$.  If $G$  is  of  type  7  then
it is not always the case that $\gamma(A) \le \Inn(G)$; we will treat the case
$\gamma(A)\not\leq   \Inn(G)$  separately in
Section~\ref{sec:  7}).  
Therefore,  for $G$  of types 8, 9, 10, or of type 7 and $\gamma(A) \le \Inn(G)$, equation \eqref{eq: RGF-on-A-by-sigma} holds and we can apply Lemma~\ref{lemma:duality} with $C=A$, getting equation~\eqref{eq:sigma}.

We have the following case distinction.

\subsection{\texorpdfstring{$\sigma, 1 - \sigma$ are}{} not both invertibile}
\label{subsec:sigma}
This means that $\sigma$ has an eigenvalue $0$ or $1$. If it is $0$, then $p\mid \Size{\ker(\gamma)}$. If it is $1$, consider the dual gamma function defined as $\gammatilde(g)=\gamma(g^{-1})\iota(g^{-1})$ (see \cite[Proposition 2.22]{p2qciclico}). Then for $a \in A$, $\gammatilde(a)=\gamma(a^{-1})\iota(a^{-1})=\iota(a^{\sigma-1})$, so that $p\mid \Size{\ker(\gammatilde)}$. Therefore, up to switch $\gamma$ with $\gammatilde$, we can assume the eigenvalue is $0$, so that $p$ divides the order of the kernel of $\gamma$.

\subsection{\texorpdfstring{$\sigma, 1 - \sigma$ are}{} both invertibile}
\label{subsec:sigma_1-sigma_inv}
This means that $\sigma$ has no eigenvalues $0, 1$. Then equation~\eqref{eq:sigma} yields
\begin{equation}
  \label{eq:conjugation}
  (\sigma^{-1} - 1)^{-1} \gamma(b)_{\restriction C} (\sigma^{-1} - 1)
  = \gamma(b)_{\restriction C} \iota(b)_{\restriction C},
\end{equation}
where $b \ne 1$ is a $q$-element. 
Thus $\gamma(b)_{\restriction C}$ and $\gamma(b)_{\restriction C} \iota(b)_{\restriction C}$ are conjugate, and this yields some information about the eigenvalues of $\gamma(b)_{\restriction C}$.

For type 7, if $q>2$~\eqref{eq:conjugation} is plainly impossible, as
\begin{equation*}
  \iota(b) = 
  \begin{bmatrix}
    \lambda &  \\
    & \lambda
  \end{bmatrix},
\end{equation*}
for some $\lambda \ne 1$, $\lambda$ of order $q$.

For type 8, the two normal subgroups of order $p$ are
characteristic, so $\gamma(b)_{\restriction C}$ and
$\iota(b)_{\restriction C}$ commute, as they are simultaneously diagonal. Let 
\begin{equation*}
  \iota(b)_{\restriction C}
  =
  \begin{bmatrix}
    \lambda_{1} & \\
     & \lambda_{2}\\
  \end{bmatrix},
  \qquad
  \gamma(b)_{\restriction C}
  =
  \begin{bmatrix}
    \alpha_{1} & \\
     & \alpha_{2}\\
  \end{bmatrix},
\end{equation*}
with $\lambda_{i} \ne 1$. This implies $\alpha_{1} = \lambda_{2}
\alpha_{2}$ and $\alpha_{2} = \lambda_{1} \alpha_{1}$, so that
$\alpha_{1} = \lambda_{1} \lambda_{2} \alpha_{1}$ and $\lambda_{1}
\lambda_{2} = 1$, against the assumption of type $8$.

For type 9, however, this is well possible. This time
there is an automorphism of order two exchanging the two eigenspaces,
but since $\gamma(b)_{\restriction C}$ has odd order $q$, it leaves
them invariant, so that once more $\gamma(b)_{\restriction C}$ and
$\iota(b)_{\restriction C}$ commute, as they are simultaneously
diagonal.

In the same notation as for type 8, here we get
$\lambda_{1} = \lambda$, $\lambda_{2} = \lambda^{-1}$, $\alpha_{1} =
\alpha$ and $\alpha_{2} = \lambda \alpha$. We get 
\begin{equation*}
  \sigma^{-1} - 1
  =
  \begin{bmatrix}
     & s_{1}\\
    s_{2} & \\
  \end{bmatrix}
\end{equation*}
(with $s_{1} s_{2} \ne 1$), or
\begin{equation*}
  \sigma
  =
  (1 - s_{1} s_{2})^{-1}
  \begin{bmatrix}
    1 & - s_{1}\\
    -s_{2} & 1
  \end{bmatrix}.
\end{equation*}

For type 10, the eigenvalues of $\iota(b)$ are not in the base field,
but in a quadratic extension of it. Still, this is similar to
case $9$.

\subsection{Some results on \texorpdfstring{$\GL(2,p)$}{GL(2,p)}}

We collect here some information about $\GL(2,p)$, which will be useful for the groups $G$ of type 5 or 7.
We will denote by $A$ and $B$, the Sylow $p$-subgroup (which is unique in both cases) and a Sylow $q$-subgroup of $G$, respectively.

\subsubsection{Sylow $p$-subgroups}
\label{sss:GL2_1}
$\GL(2,p)$ has $p+1$ Sylow $p$-subgroups and each of them fixes a $p$ subgroup of $\cC_{p} \times \cC_{p}$. In the following we will denote by $\alpha$ an element of order $p$ of $\GL(2,p)$.

\subsubsection{Elements of order $p$ when $p\mmid \Size{\ker(\gamma)}$}
\label{sss:GL2_2}
Suppose that $G$ is of type 5 or 7, and let $\gamma$ be a GF on $G$ such that $\Span{a_1}\leq \ker(\gamma) \neq A$, where $a_1 \in A$, $a_1 \neq 1$. Let $a_2 \in A \setminus \Span{a_1}$, then $\gamma(a_2)=\alpha$  (possibly modulo $\iota(A)$), where $\alpha \in \GL(2,p)$ has order $p$. Then
\begin{equation}
\label{eq:betafixedpoint}
  a_{1}^{\alpha} a_{2} = a_{1} \circ a_{2} = a_{2} \circ a_{1} = a_{2} a_{1},
\end{equation}
so that $a_{1}$ is fixed by $\alpha$. This means that $\ker(\gamma)$ determines $\Span{\alpha}$, which is the Sylow $p$-subgroup of $\GL(2,p)$ fixing $\ker(\gamma)$.

\subsubsection{Sylow $q$-subgroups}
\label{sss:GL2_3}

Suppose that $q\mid p-1$ and recall that $\Size{\GL(2, p)} = (p - 1)^{2} p (p + 1)$. 

If $q>2$ a Sylow $q$-subgroup of $\GL(2, p)$ has order $q^{2 e}$, where $q^{e} \mmid p - 1$.
Every Sylow $q$-subgroup of $\GL(2, p)$ is of the form

Every Sylow $q$-subgroup of $\GL(2, p)$ is of the form
\begin{align*}
  Q_{A_{1}, A_{2}}
  &=
  \begin{aligned}[t]
  \{
  \beta \in \GL(2, p)
  \colon&
  \text{$A_{1}, A_{2}$ are eigenspaces of $\beta$ with respect}
    \\&\text{to
      eigenvalues of order dividing $q^{e}$}
    \}
  \end{aligned}
  \\&\cong
  \cC_{q^{e}} \times \cC_{q^{e}},
\end{align*}
for any choice of a pair $\Set{ A_{1}, A_{2} }$ of distinct one-dimensional subspaces  of $A$. Thus there are $\frac{p(p+1)}{2}$ Sylow $q$-subgroups.

Moreover, each Sylow $q$-subgroup of $\GL(2,p)$ has $q^{2} - 1$ elements of order $q$. However, the scalar elements are common to all the Sylow $q$-subgroups. Hence $\GL(2,p)$ has
\begin{equation*}
  (q^{2} - q) \cdot \frac{(p + 1) p}{2} + q - 1
\end{equation*}
elements of order $q$.

If $q=2$, the Sylow $2$-subgroups of $\GL(2,p)$ are described in \cite{CartFong}. Note that in this case if $\theta$ has order $2$, then its minimal polynomial divides $x^2-1$, and therefore its egeinvalues belong to $\Set{\pm 1}$. Moreover all the elements with eigenvalues $1, -1$ are conjugate, and such an element, say $\theta$, is stabilised by the diagonal matrices, therefore $\Size{\Orb(\theta)}=p(p+1)$. 
Thus there are $p(p+1)$ non-scalar elements of order $2$, plus the scalar matrix $\diag(-1, -1)$.

\subsubsection{Elements of order $q$ when $\Size{\ker(\gamma)}=p$}
\label{sss:GL2_4}
Suppose that $q \mid p-1$ and $G$ is of type 5 or 7. Let $\gamma$ be a GF on $G$ with kernel $\Span{a_1}$, where $a_1 \in A$. Let $b \in G$ be such that $\gamma(b)=\beta$ (possibly modulo $\iota(A)$), where $\beta$ is an element of order $q$ in the normaliser of $\alpha$. 
Then $\alpha^{\beta}=\alpha^{t}$ for a certain $t$, and Subsection~\ref{sss:GL2_2} yields that $\Span{a_1}$ is fixed by $\alpha$, so that
$$ a_1^{\beta \alpha} = a_1^{\alpha^{t^{-1}}\beta} = a_1^{\beta} ,$$
namely $a_1^{\beta}$ is fixed by $\alpha$ as well. Therefore $a_1^{\beta} \in \Span{a_1}$, so that $\Span{a_1}$ is an eigenspace for $\beta$ too.

Let  $\Span{a_3}$  be  another  eigenspace for  $\beta$.  Then,  since
$\det(\alpha)^{p}=1$, replacing $a_3$  with a suitable  element in
$\Span{a_3}$  we can  write,  with respect  to  the basis  $\Span{a_1,
  a_3}$,
$$\alpha=\begin{bmatrix}
1 & 0 \\
1 & 1
\end{bmatrix}, \ 
\beta = \begin{bmatrix}
\lambda^{x_1} & 0 \\
0 & \lambda^{x_2}
\end{bmatrix},
$$
where $\lambda$ has order $q$, and $x_1, x_2$ are not both $0$. 

Note that if $\beta$ is a scalar matrix, there are $q-1$ elements $\beta$ as above. If $\beta$ is non-scalar, taking into account the choice of $\Span{a_3}$, there are $q(q-1)p$ possibilities for $\beta$.

\section{Type 5}

Here $G = (\cC_{p} \times \cC_{p}) \times \cC_{q}$, and $\Aut(G) = \GL(2, p) \times \cC_{q-1}$.

Let $A$ be the Sylow $p$-subgroup, and $B = \Span{b}$ the Sylow $q$-subgroup.

In the following we will denote by $\alpha$ an element of order $p$ of $\GL(2,p)$.
If $p\mid q-1$ we will denote by $\eta$ an element of order $p$ of $\cC_{q-1}$: clearly $\eta$ fixes $A$ point-wise. If $q\mid p-1$ we will denote by $\beta$ an element of order $q$ of $\GL(2,p)$.

\subsection{Abelian groups}
\label{sssec:Gcirc5}
Assume here  $(G, \circ)$  abelian. These are  in particular  the only
cases  when  there  are  no divisibilities.

$\gamma(b)$ will have order dividing $q$, so it is an element in $\GL(2,
p)$ of order dividing $q$. Then for $a \in A$ we have
\begin{equation*}
  b^{\ominus 1} \circ a \circ b
  =
  b^{-\gamma(b)^{-1} \gamma(a) \gamma(b)} b a^{\gamma(b)}
  =
  b^{-\gamma(a)} b a^{\gamma(b)}
  =
  a,
\end{equation*}
from which we get that $\gamma(b) = 1$, and also that
$\gamma(a)_{\restriction  B} = 1$. 

Thus $B \le \ker(\gamma)$.
If $\gamma(G)=\Set{1}$, then we obtain the right regular representation, which corresponds to one group of type 5. Otherwise $\gamma(G) \ne \Set{1}$, and we can only have $\gamma(G) = \gamma(A) = \Span{\alpha}$, where $\alpha \in \GL(2, p)$ has order $p$.
Therefore, each GF on $G$ is the lifting of a RGF  $\gamma\colon A\to \Aut(G)$ with $\Size{\gamma(A)}=p$. 

Let  $1 \ne a_{1} \in A$ and let $\ker(\gamma)=\Span{a_1}$ (we have $p+1$ choices for such a subgroup); the argument in~\ref{sss:GL2_2} shows that 
$a_{1}$ is fixed by $\alpha$, so that $\ker(\gamma)$ determines $\gamma(A)$, and we have $\gamma(a_2)=\alpha^{i}$, for $1\leq i \leq p-1$.
Note that for each $i$ the unique morphism defined by $\gamma(a_1)=1$ and $\gamma(a_2)=\alpha^i$ is such that $[A,\gamma(A)]=\ker(\gamma)$, so by Lemma~\ref{Lemma:gamma_morfismi}, these morphisms coincide with the RGF's. Therefore here we have $(p+1)(p-1)=p^2-1$ different GF's on $G$ giving groups $(G,\circ)$ of type 5.

As to the conjugacy classes, since $B\leq \ker(\gamma)$ is characteristic, every automorphism $\phi$ of $G$ stabilises $\gamma_{|B}$. Moreover, if $\mu \in \Aut(B) \cong \cC_{q-1}$, then $\Span{\mu}$ centralises $a$ and $\gamma(a)$, so that it centralises $\gamma$.

Now, let $\delta \in \Aut(A) \cong \GL(2,p)$. If $\delta$ stabilises $\gamma$, then $\gamma^{\delta}(a_1)=1$, namely $\gamma(a_1^{\delta^{-1}})=1$. Therefore $\delta^{-1}$ fixes $\Span{a_1}$, and writing $\delta=(\delta_{ij})_{i,j}$ with respect to the basis $\Set{a_1, a_2}$, this implies that $\delta_{12}=0$.

As for $a_2$, we have
\begin{equation*}
\gamma^{\delta}(a_2)
= \delta^{-1}\gamma(a_2^{\delta^{-1}})\delta 
= \delta^{-1}\alpha^{\delta_{22}^{-1}}\delta ,
\end{equation*}
and it coincides with $\gamma(a_2)$ precisely when $\delta^{-1}\alpha^{\delta_{22}^{-1}}\delta=\alpha$. An explicit computation shows that the latter yields $\delta_{11}=\delta_{22}^2$.
Therefore, the stabiliser of $\gamma$ has order $(q-1)p(p-1)$, and there is one orbit of length $p^2-1$.

In the following we exclude the abelian cases just dealt with.

\subsection{\texorpdfstring{$p \mid q - 1$}{p | q-1}}
\label{sssec:Gcirc11}
Here $B \le \ker(\gamma)$, and the only type of groups we can have here is the type 11, beside the type 5 already considered.

Suppose first $\ker(\gamma) = \Span{a_{1}} B$ has order $p q$. Then $\gamma(G)$ has order $p$, and let $a_2$ be such that $\gamma(a_{2}) = \alpha^{i} \eta^{j}$, where $0\le i<p$, $j \ne 0$ (since we are assuming $(G,\circ)$ non abelian).  
The argument in~\ref{sss:GL2_2} shows that $a_1^{\alpha}=a_1$, and Lemma~\ref{lem:kernel-pq-gamma-morphisms} yields that $\gamma$ is a RGF if and only if it is a morphism.
Therefore the GF's are as many as the choices of $(\Span{a_1}, i,j)$, namely $(p+1)p(p-1)=p(p^2-1)$, and each of them corresponds to a group $(G,\circ)$ of type 11.

As to the conjugacy classes, again $\cC_{q-1}$ stabilises every $\gamma$. Moreover, if $\delta \in\GL(2,p)$ stabilises $\gamma$, then $\delta^{-1}$ fixes $\Span{a_1}$, so that $\delta_{12}=0$. This time
\begin{equation*}
\gamma^{\delta}(a_2)
= \delta^{-1}\gamma(a_2^{\delta_{22}^{-1}})\delta 
= \delta^{-1}\alpha^{i\delta_{22}^{-1}}\delta \eta^{j \delta_{22}^{-1}} ,
\end{equation*}
where $j \ne 0$. Therefore $\delta$ stabilises $\gamma$ precisely when $\delta_{12}=0$, $\delta_{22}=1$, and $\delta$ centralises $\alpha^{i}$. If $i=0$, the latter yields no condition, while corresponds to take $\delta_{11}=1$ if $i\ne 0$. So the $\delta$'s in the stabiliser are those of the form 
$$
\delta = \begin{bmatrix}
\delta_{11} & 0 \\
\delta_{21} & 1
\end{bmatrix}  \text{  if  } i=0, \text{  and  } 
\delta = \begin{bmatrix}
1 & 0 \\
\delta_{21} & 1
\end{bmatrix} \text{  if  } i\ne 0.
$$
Therefore, if $i=0$ the stabiliser has order $(q-1)p(p-1)$, and there is one orbit of length $p^2-1$.
If $i \ne 0$, the stabiliser has order $(q-1)p$, and there is one orbit of length $(p^2-1)(p-1)$.
\vskip 0.3cm

Now suppose $\ker(\gamma) = B$ has order $q$. Then $\gamma(G)=\gamma(A)=\Span{\alpha, \eta}.$
Let  $a_{1}, a_{2} \in A$ such that
\begin{equation}
\label{eq:assignment}
  \begin{cases}
    \gamma(a_{1}) = \eta\\
    \gamma(a_{2}) = \alpha \\
  \end{cases},
\end{equation}
Since
$$ a_1^{\alpha}a_2 = a_1 \circ a_2 = a_2 \circ a_1 = a_2 a_1 ,$$
$a_{1}$ is a fixed point of $\alpha$, and since $\alpha$ has determinant equal to $1$, we can suppose
$$\alpha=\begin{bmatrix}
1 & 0 \\
d & 1
\end{bmatrix},
$$
with respect to $\Set{a_1, a_2}$, where $ 1 \leq d \leq p-1$.
By Lemma \ref{lemma:generatori} and Lemma~\ref{remark:gammaGFkerq}, each assignment~\eqref{eq:assignment} defines exactly one GF. Therefore, in this case, we have $p + 1$ choices for $\gamma(G)$, and once $\gamma(G)$ has been chosen, there are $p-1$ ways to choose $a_1$ among the fixed points of $\alpha$, and $p^2-p$ choices for $a_2$, which is any element of $A\setminus\Span{a_1}$. So there are $(p^2-1)p(p-1)$ groups of type 11.

As to the conjugacy classes, every automorphism in $\cC_{q-1}$ stabilises $\gamma$. 
Since $B=\ker(\gamma)$ is characteristic, by Lemma~\ref{lemma:generatori}, we just consider the action of $\GL(2,p)$ on $\gamma$ defined on the generators of $A$.

Let $\delta \in \GL(2,p)$. Then, $\gamma^{\delta}(a_1) = \gamma(a_1)$ if and only if $\gamma(a_1^{\delta^{-1}})=\gamma(a_1)$, as $\delta^{-1}$ centralises $\eta$. The latter yields $\gamma(a_1^{\delta^{-1}})_{|A}=1$, so that $a_1^{\delta^{-1}}\in \Span{a_1}$, namely $\delta_{12}=0$. Moreover, since $\gamma_{|\Span{a_1}}$ is a morphism, $\gamma(a_1^{\delta^{-1}})=\eta$ if and only if $\delta_{11}=1$.
Now, since 
$$\gamma(a_2^k)
=\gamma(a_1)^{-d(\frac{k(k-1)}{2})}\gamma(a_2)^k 
=\eta^{-d(\frac{k(k-1)}{2})}\alpha^k
,$$ 
we have 
\begin{equation*}
\gamma^{\delta}(a_2)
= \delta^{-1} \gamma(a_2^{\delta^{-1}}) \delta \\
=\delta^{-1} \gamma(a_1^{-\delta_{21}\delta_{22}^{-1}}a_2^{\delta_{22}^{-1}}) \delta \\
=\eta^{-\delta_{22}^{-1}(\delta_{21}+\frac{d}{2}(\delta_{22}^{-1}-1))}  \delta^{-1} \alpha^{\delta_{22}^{-1}} \delta ,
\end{equation*} 
and the latter coincides with $\gamma(a_2)$ precisely when
$$
\begin{cases}
\delta^{-1} \alpha^{\delta_{22}^{-1}} \delta= \alpha \\
\delta_{21}=-\frac{d}{2}(\delta_{22}^{-1}-1) .
\end{cases}
$$ 
The first condition yields $\delta_{22}^2=1$, namely $\delta_{22}=\pm 1$, so that the second yields $\delta_{21}= 0, d$ respectively when $\delta_{22}=1, -1$.
Therefore the stabiliser has order $2(q-1)$ and we get $2$ orbits of length $\frac{1}{2}(p^2-1)p(p-1)$.

\subsection{\texorpdfstring{$q \mid p - 1$}{q | p-1}}
\label{subsub:G5-q-divides-p-1}
Here $\gamma(G)\subseteq \GL_2(p)$, so $p\mid \ker(\gamma)$ and $\gamma(G)$ acts trivially on $B$, so that
\begin{equation}
\label{pallino}
 b^{\ominus 1} \circ a \circ b
 =
  b^{-\gamma(b)^{-1} \gamma(a) \gamma(b)} a^{\gamma(b)} b
  =
  b^{-1} b a^{\gamma(b)}
  =
  a^{\gamma(b)}.
  \end{equation}
If $pq\mid \ker(\gamma)$, then equation~\eqref{pallino} becomes
$$ b^{\ominus 1} \circ a \circ b =
  a,$$ 
so $(G,\circ)$ of type 5 and has already been considered. Thus we just deal with the cases of kernel $p^2$ and $p$.

If $\ker(\gamma) = A$ the GF's are exactly the morphisms. Let  $\lambda\in\cC_p^*$ be an element of order $q$. By Subsubsection~\ref{sss:GL2_3}, with respect to a suitable basis $\Set{a_1,a_2}$ of $A$, we have   
$$\beta=[\gamma(b)]=
\begin{bmatrix}
  \lambda^{x_1} & 0\\
    0 & \lambda^{x_2}\\
\end{bmatrix} .
$$ 
Now, since $a^{k}=a^{\circ k}$, equation~\eqref{pallino} yields that the action of $\iota(b)$ on $A$ in $(G, \circ)$ is precisely $\gamma(b)$.
According to the choices of $\gamma(b)$ we easily obtain, besides the abelian cases,
\begin{enumerate}
\item $q - 1$ groups of type $7$, corresponding to the choices $x_1=x_2\ne0$.
\item $\frac{(p+1)p}{2}\cdot2\cdot (q - 1)$ groups of type 6: choose the
  eigenspaces, and then the eigenvalue different from $1$.
\item if $q>2$, we get  $\frac{(p+1)p}{2} \cdot (q-1)$ groups of type 9.
\item if  $q>3$ we get  $\frac{(p+1)p}{2} \cdot (q-1)(q-3)$  groups of
  type 8.  More precisely, denoting by  $Z_\circ$ the action of $b$ on
  $A$ in $(G, \circ)$, since $Z_{\circ}$ is similar to $\diag(\mu^{x_1 x_2^{-1}},
  \mu)$, where $\mu=\lambda^{x_2}$, they split in $p(p+1)(q-1)$ groups
  isomorphic to $G_s$ for every $s \in \mathcal{K}$.
\end{enumerate}

\begin{remark}
  In the following we will write $Z_{1} \sim Z_{2}$ to mean that the
  two square matrices $Z_{1}, Z_{2}$ are similar.
\end{remark}

As to the conjugacy classes, since $A=\ker(\gamma)$ is characteristic,
$\gamma_{|A}$ is stabilised by every automorphism $\phi$ of $G$.

As for $\gamma_{|B}$, let $\mu \in \cC_{q-1}$, so that $b^{\mu^{-1}}=b^{k}$ for some $k$, and let $\delta \in \GL(2,p)$. Then
$$ \gamma^{\mu\delta}(b)=\delta^{-1}\gamma(b)^k\delta .$$
Therefore $\mu\delta$ stabilises $\gamma$ precisely when $T$ and $T^k$ are conjugate, and in that case they need to have the same eigenvalues, namely $kx_1=x_1$ and $kx_2=x_2$ or $kx_1=x_2$ and $kx_2=x_1$.
Note that if $k=1$, then $\delta$ stabilises $\gamma$ if and only if it is in the centraliser of $T$: if $T$ is scalar, then every $\delta \in \GL(2,p)$ stabilises $\gamma$, while for a non-scalar matrix $T$ the condition is equivalent to have $\delta$ a diagonal matrix with no diagonal elements equal to zero. 

Referring to the cases above, we have the following.
\begin{enumerate}
\item $T$ is scalar and $T \sim T^k$ if and only if $k=1$, so that the stabiliser has order $\Size{\GL(2,p)}$, and there is one orbit of length $q-1$.
\item $T$ is non-scalar and $k=1$. In this case the centraliser of $T$ consists of the elements $\delta=\diag(\delta_{11}, \delta_{22})$, with $\delta_{ii} \ne 0$, therefore it has $(p-1)^2$ elements. Thus $\Size{\Stab(\gamma)}=(p-1)^2$, and there is one orbit of length $p(p+1)(q-1)$.
\item $T$ is non-scalar and $k=\pm 1$. If $k=1$ then the elements in the stabiliser are the diagonal matrices as above. If $k=-1$ the stabiliser consists of the elements $\mu\delta$, where $b^{\mu^{-1}}=b^{-1}$, and
$$ 
\delta=\begin{bmatrix}
0 & \delta_{12} \\
\delta_{21} & 0
\end{bmatrix},
$$
where $\delta_{12}\ne 0 \ne \delta_{21}$. Therefore $\Size{\Stab{\gamma}}=2(p-1)^2$, and there is one orbit of length $\frac{1}{2}p(p+1)(q-1)$.
\item $T$ is non-scalar and $k=1$, indeed if $kx_1=x_2$ and $kx_2=x_1$, then $x_2^{-1}x_1=k=x_1^{-1}x_2$, namely $x_1= \pm x_2$ (contradiction).
Therefore $\Size{\Stab{\gamma}}=(p-1)^2$, and for each $G_k$ there is one orbit of length $p(p+1)(q-1)$.
\end{enumerate}

If $\ker(\gamma)=\Span{a_1}$ has order $p$, then $\gamma(G)$ is a subgroup of order $pq$ of $\GL(2,p)$, so $\gamma(G)=\Span{\alpha, \beta}$, where $\alpha$ has order  $p$, $a_1^{\alpha}=a_1$, and $\beta$ is an element of order $q$ in the normaliser of $\alpha$ in $\GL(2,p)$. 
By Subsubsection~\ref{sss:GL2_4}, we can choose $\Span{a_2}\leq A$ such that, with respect to the basis $\{a_1,a_2\}$,
$$
[\alpha]=
\begin{bmatrix}
1 & 0\\
    d &1\\
\end{bmatrix}, \ 
[\beta]=
\begin{bmatrix}
  \lambda^{x_1} & 0\\
    0 & \lambda^{x_2}\\
\end{bmatrix},
$$
where $d\in\cC_p^*$, $\lambda$ has order $q$ in $\cC_p^\ast$, and $(x_1,x_2)\ne (0,0)$.

Let $\gamma$ be a GF such that $\gamma(G)=\Span{\alpha, \beta}$. 
Then $\gamma(a_1)=1$ and $\gamma(a_2)=\alpha^{d}$, where $1 \leq d \leq p-1$. Morever, let $b\in B$ be such that $\gamma(b)=\beta$.

By applying $\gamma$ to~\eqref{pallino}, we get
$$ 
\gamma(b)^{-1}\gamma(a)\gamma(b)=\gamma(a^{\gamma(b)})
$$
which for $a=a_2$, in terms of our notation, can be rewritten as 
\begin{align*}
\beta^{-1}\alpha^d\beta=&\alpha^{d \lambda^{x_2}}\\
\alpha^{d \lambda^{x_1-x_2}}=&\alpha^{d \lambda^{x_2}}
\end{align*}
which correspond to the condition 
\begin{equation}
\label{x1x2}
x_1 \equiv 2x_2\pmod q.
\end{equation}
This condition restrict the choices of $\beta$ to a set of  $(q-1)p$ maps, namely the elements of order $q$ in the normaliser of $\alpha$ with diagonal $\lambda^{2x_2}, \lambda^{x_2}$. Thus for each choice of $\Span{\alpha}$ only one group of order $pq$ can be the image of a GF.

We note that the maps $\beta$ fulfilling equation~\eqref{x1x2} normalise but do not centralise $\Span{\alpha}$, so $\Span{\alpha, \beta}$ is not abelian. 

The condition~\eqref{x1x2} is also sufficient to have that the map $\gamma$, defined as
$$ \gamma(a_1^{e} a_2^{f} b^{g}) = \beta^{g}\alpha^{f},$$
is a gamma function, indeed we have
\begin{align*}
\gamma((a_1^{e} a_2^{f} b^{g})^{\gamma(a_1^{u} a_2^{v} b^{z})}a_1^{u} a_2^{v} b^{z})
&= \gamma((a_1^{e} a_2^{f} b^{g})^{\beta^{z}\alpha^{v}}a_1^{u} a_2^{v} b^{z}) \\
&= \gamma((a_1^{\ast} a_2^{f\lambda^{x_2z}} b^{g}) a_1^{u} a_2^{v} b^{z}) \\
&= \gamma(a_1^{\ast} a_2^{f\lambda^{x_2z}+v} b^{g+z}) \\
&= \beta^{g+z}\alpha^{f\lambda^{x_2z}+v}.
\end{align*}
On the other hand,
\begin{align*}
\gamma(a_1^{e} a_2^{f} b^{g})\gamma(a_1^{u} a_2^{v} b^{z}) 
&= \beta^{g}\alpha^{f}\beta^{z}\alpha^{v} \\
&= \beta^{g+z}\alpha^{f\lambda^{(x_1-x_2)z}+v},
\end{align*}
so that $\gamma$ defined as above is a GF if and only if $x_1 \equiv 2x_2 \pmod q$.

Moreover, since we have $p+1$ choices for $\Span{\alpha}$, $p-1$ for $d$, and $p(q-1)$ for $\beta$, we obtain $p(p^2-1)(q-1)$ groups $(G, \circ)$.

As for the type of $(G,\circ)$, with respect to the basis $\{a_1,a_2\}$ we have
$$
T=[\beta]=
\begin{bmatrix}
  \lambda^{2x_2} & 0\\
    0 & \lambda^{x_2}
\end{bmatrix};
$$
Since $a_1^{k}=a_1^{\circ k}$ and $a_2^{k}=a_2^{\circ k}$ modulo $\Span{a_1}$, denoting by $Z_{\circ}$ the action of $b$ on $A$ in $(G,\circ)$, we have $Z_\circ \sim T$.
Therefore,
\begin{itemize}
\item if $q>3$ all groups $(G,\circ)$ are of type 8, and they are all isomorphic to $G_2$;
\item if $q=3$ all groups  $(G,\circ)$ are of type 9;
\item if $q=2$ we have $x_1=0$, $x_2=1$, so all groups  $(G,\circ)$ are of type 6.
\end{itemize}

As to the conjugacy classes, let $\phi \in \Aut(G)$, and write $\phi=\mu\delta$ as above. 
If $\phi$ is in the stabiliser of $\gamma$ then $\phi$, and hence $\delta$, stabilises $\Span{a_1}$, so $\delta_{12}=0$. Moreover, 
\begin{equation*} 
\gamma^{\phi}(a_2) 
=\phi^{-1}\gamma(a_2^{\delta^{-1}})\phi 
=\phi^{-1}\gamma(a_2^{\delta_{22}^{-1}})\phi
=\delta^{-1}\alpha^{\delta_{22}^{-1}}\delta,
\end{equation*}
and $\gamma^{\phi}(a_2)=\gamma(a_2)$ if and only if $\delta_{11}=\delta_{22}^2$.
Now, 
$$ \gamma^{\phi}(b)
=\phi^{-1} \gamma(b^{\mu^{-1}}) \phi 
=\phi^{-1} \gamma(b^{k}) \phi
=\delta^{-1} T^{k} \delta,
$$
so that, if $\phi$ stabilises $\gamma$, then $T$ and $T^{k}$ are conjugate, and they have the same eigenvalues. This implies that either $k=1$ or $k=2$ and $q=3$. If $k=1$, then every diagonal matrix $\delta$ commutes with $T$. If $q=3$ and $k=2$, then the condition $\delta^{-1} T^{-1} \delta = T$ yields $\lambda^{x_2}=\lambda^{-x_2}$, and since $x_2 \ne 0$ this case does not arise.
Therefore the stabiliser has order $p-1$, and there is one orbit of length $p(p^2-1)(q-1)$.

\subsection{\texorpdfstring{$q \mid p + 1$}{q | p+1}} 
We have to exclude the cases already considered, so we restrict to $q>2$ (otherwise $q$ also divides $p-1$) and $(G,\circ)$ non-abelian. Therefore $(G,\circ)$ can only have type 10.

As above $p\mid\Size{ \ker(\gamma)}$, and the only possibility is $\Size{ \ker(\gamma)}=p^2$ since a group of type 10 has no normal subgroups of order $p$ or $pq$.

Lemma~\ref{Lemma:gamma_morfismi} guarantees that in this case all the GF's are morphisms, so to count them we can just count the possibilities for the image of $b$.

An element $\theta \in \GL(2,p)$ of order $q$ has determinant equal to $1$, as $q \nmid p-1$, and its eigenvalues $\lambda$, $\lambda^{-1}$, belongs to a quadratic extension of $\cC_p$. Therefore, every subgroup of $\GL(2,p)$ of order $q$ is conjugate to $\Span{\theta}$, and in $\GL(2,p)$ there are 
$$\frac{\Size{\GL(2,p)}}{\Size{\Stab(\Span{\theta})}}$$
subgroups of order $q$. Now, if $\theta$ and $\theta^{k}$ are conjugate, they have the same eigenvalues, and this yields $k=\pm 1$. For each of these two choices we obtain $p^2-1$ elements in the stabiliser, therefore there are
\begin{equation*}
  \frac{(p^{2} - 1) (p^{2} - p)}{2 (p^{2} - 1)}
  =
  \dbinom{p}{2}  
\end{equation*}
subgroups of order $q$ in $\GL(2, p)$.

So we can choose the image of $b$ in such a subgroup in $q - 1$ ways, and we get
\begin{enumerate}
\item $\dbinom{p}{2} (q - 1)$ groups of type 10.
\end{enumerate}

As to the conjugacy classes, $A=\ker(\gamma)$ is characteristic, therefore every automorphism $\phi$ of $G$ stabilises $\gamma_{|A}$.

Let $b\in B$ such that $\gamma(b)=\theta$, and let $\phi=\mu\delta \in \Aut(G)$.
Then
$$ \gamma^{\phi}(b)=\delta^{-1}\gamma(b^{k})\delta,$$
so that $\phi$ stabilises $\gamma$ if and only if $\theta$ and $\theta^k$ are conjugate via $\delta$. 
As above, in this case $k=\pm 1$, and for each of these values of $k$ there are $p^2-1$ possibilities for $\delta$.

Therefore, we get one orbit of length $\frac{1}{2}(q-1)p(p-1)$.

We summarise, including the right regular representation.

\begin{prop}
\label{prop:G5}
Let $G$ be a group of order $p^2q$, $p>2$, of type 5. Then in $\Hol(G)$ there are:
\begin{enumerate}
\item $p^2$ groups of type 5, which split in one conjugacy class of length one, and one conjugacy class of length $p^2-1$; 
\item if $p|(q-1)$,
 \begin{enumerate}[(a)]
 \item $p^2(p^2-1)$ groups of type 11, which split in one conjugacy class of length $p^2-1$, one conjugacy class of length $(p-1)(p^2-1)$, and two conjugacy classes of length $\frac{1}{2}p(p-1)(p^2-1)$;
 \end{enumerate}
\item if $q|(p-1)$,
 \begin{enumerate}[(a)]
 \item $p(p+1)(q-1)$ groups of type 6, which form one conjugacy class of length $p(p+1)(q-1)$;
 \item $q-1$ groups of type 7, which form one conjugacy class of length $q-1$;
 \item if $q=2$, further $p(p^2-1)$ groups of type 6, which form one conjugacy class of length $p(p^2-1)$;
 \item if $q>2$, $\frac{1}{2}p(p+1)(q-1)$ groups of type 9, which form one conjugacy class of length $\frac{1}{2}p(p+1)(q-1)$; 
 \item if $q=3$, further $p(p^2-1)(q-1)$ groups of type 9, which form one conjugacy class of length $p(p^2-1)(q-1)$;  
 \item if $q>3$, 
	\begin{itemize}
	\item[$-$] $p^2(p+1)(q-1)$ groups of type 8 isomorphic to $G_2$, which split in one conjugacy class of length $p(p+1)(q-1)$ and one conjugacy class of length $p(p^2-1)(q-1)$;
	\item[$-$] for every $s \ne 2$, $s \in \mathcal{K}$, $p(p+1)(q-1)$ groups of type 8 isomorphic to $G_s$, which form one conjugacy class of length $p(p+1)(q-1)$;
	\end{itemize}
 \end{enumerate} 
\item if $q|(p+1)$ and $q>2$,
 \begin{enumerate}[(a)]
 \item $\frac{p(p-1)}{2}(q-1)$ groups of type 10, which form one conjugacy class of length $\frac{p(p-1)}{2}(q-1)$.
 \end{enumerate}
\end{enumerate}
\end{prop}

\section{Type 6}

In this case $q \mid p-1$, and $G = \cC_{p} \times (\cC_{p} \rtimes \cC_{q})$.
The Sylow $p$-subgroup $A$ is characteristic in $G$. Write $C = \Span{c}$ for the normal subgroup of
order $p$ in $\cC_{p} \rtimes  \cC_{q}$, and $Z = \Span{z}$ for the central factor of order $p$, so that $A = C Z = \Span{c, z}$. 

We have
$$\Aut(G) = \cC_{p-1}  \times   \Hol(\cC_{p}).$$
Write $\Span{\psi}=\cC_{p-1}$ for the central factor in $\Aut(G)$, and let $\Hol(\cC_p) = \iota(C) \rtimes \Span{\mu}$, where, according to~\cite{classp2q},
\begin{equation}
\label{eq:aut-G-6}
\psi:
\begin{cases}
z \mapsto z^{k} \\
c \mapsto c \\
b \mapsto b 
\end{cases},
\ \  
\mu:
\begin{cases}
z \mapsto z \\
c \mapsto c^{h} \\
b \mapsto b 
\end{cases},
\end{equation}
with $1 \leq k, h \leq p-1$. 

By Lemma~\ref{lemma:duality}~and Proposition~\ref{prop:duality} we can assume that $C \le \ker(\gamma)$. 
In the following, it is useful to keep in mind that
\begin{equation*}
  \begin{cases}
    b^{\ominus 1} \circ c \circ b
    =
    c^{\gamma(b) \iota(b)}\\
    b^{\ominus 1} \circ z \circ b
    =
    (b^{-\gamma(b)^{-1} \gamma(z) \gamma(b)} b) z^{\gamma(b)}.
  \end{cases}
\end{equation*}

\subsection{The case \texorpdfstring{$A \leq \ker(\gamma)$}{A in ker}}

Suppose $\ker(\gamma)= A$, as the case $\ker(\gamma)=G$ yields the right regular representation.
So $\gamma(G)$ has order $q$. 

The action of $\gamma(G)$ of order $q$ on the set of the Sylow $q$-subgroups of $G$ has at least one fixed point, namely there is at least one $\gamma(G)$-invariant Sylow $q$-subgroup $B$ of $G$. Therefore, by Proposition~\ref{prop:lifting}, the GF's on $G$ are induced by the RGF's on $B$, and each $\gamma$ is obtained $s$ times, where $s$ is the number of $\gamma(G)$-invariant Sylow $q$-subgroups of $G$.

Note moreover that $[B,\gamma(B)]=1$, as $B$ and $\gamma(B)$ have order $q$, so that by Lemma~\ref{Lemma:gamma_morfismi} the RGF's on $B$ are precisely the morphisms $B\to \Aut(G)$.

Let $\beta$ be the element of order $q$ in the central factor $\cC_{p-1}$ of $\Aut(G)$, such that $z^{\beta} = z^{\lambda}$, where $\lambda$ is the eigenvalue of $C$ under the action of $\iota(b)$, namely
$c^{b} = c^{\lambda}$.
 
Here $c^{\circ k}=c^k$ and $z^{\circ k}=z^k$. Let $Z_\circ$ be the action of $b$ on $A$ in $(G, \circ)$.
We will write $Z_\circ$ with respect to the basis $\Span{c, z}$ of $(A, \circ)$.
\begin{enumerate}
\item If $\gamma(b) = \beta^{i}$, for some $0 < i < q$, then $Z_\circ = \diag(\lambda, \lambda^i)$. Here the choice of $B$ is immaterial, and we get
  \begin{enumerate}
  \item 1 group of type 7 when $i = 1$;
  \item 1 group of type 9 when $i = q - 1$ and $q>2$;
  \item $q - 3$ groups of type 8, when $q > 3$. They split in $2$ groups isomorphic to $G_s$ for every $s\in \mathcal{K}$.
  \end{enumerate}
\item If $\gamma(b) = \iota(b)^{j}$, for some $0 < j < q$, then $Z_\circ = \diag(\lambda^{j+1}, 1)$ and we
  get
  \begin{enumerate}
  \item $p$ groups of type 5 when $j = -1$, for the possible choices of $B$;
  \item $p (q - 2)$ groups of type 6, for the possible choices of $B$.
  \end{enumerate}
\item If $\gamma(b)  = \beta^{i} \iota(b)^{j}$, for some $0  < i, j < q$, then $Z_\circ = \diag(\lambda^{1+j}, \lambda^i)$ and we get
  \begin{enumerate}
  \item $p (q - 1)$ groups of type 6, when $j = -1$;
  \item $p (q - 2)$ of type 7, when $i = j + 1 \ne 0$;
  \item $p (q - 2)$ of type 9, when $-i = j + 1 \ne 0$ and $q>2$;
  \item $p ((q - 1)^{2} - 3 q + 5) = p (q - 2) (q - 3)$ groups of type 8 in the remaining cases; they occur only for $q > 3$. They split in $2p(q-2)$ groups isomorphic to $G_s$, for every $s \in \mathcal{K}$.
  \end{enumerate}
\end{enumerate}

As to the conjugacy classes, since $A = \ker(\gamma)$ is characteristic, to find the automorphisms which stabilise $\gamma$, we can look at the action of $\Aut(G)$ on $\gamma_{|B}$.

The central factor $\Span{\psi}$ of $\Aut(G)$ and $\Span{\mu}$ are in the stabiliser of $\gamma$, as they centralise $b$ and $\gamma(b)$. 
As for $\iota(C)$ we have
\begin{equation*} 
\gamma^{\iota(c^m)}(b) = \iota(c^{-m})\gamma(b)\iota(c^m) = \beta^{i}\iota(b^j c^{m(1-\lambda^{j})}), 
\end{equation*}
so that it stabilises $\gamma$ if and only if $m=0$ or $j=0$. 

Therefore, if $\gamma$ is a GF defined by $\gamma(b)=\beta^i\iota(b)^j$, $j \ne 0$, the stabiliser has order $(p-1)^2$, and the orbits have length $p$. Otherwise $\gamma(b)=\beta^i$ and every automorphism stabilises $\gamma$, so that the orbits have length $1$. 
More precisely we obtain
\begin{enumerate}
\item $p$ groups of type 5 which form one class of length $p$;
\item $p(q-2)+p(q-1)=p(2q-3)$ groups of type 6 which split in $2q-3$ classes of length $p$;
\item $p(q-2)+1$ groups of type 7, which split in $q-2$ classes of length $p$ and one class of length one (the last one is for $j=0$).
\item if $q>3$, $2p(q-2)+ 2$ groups for each isomorphism class $G_s$ of groups of type 8, which split in $2(q-2)$ classes of length $p$, and $1$ classes of length one (these are for $j=0$).
\item if $q>2$, $p(q-2)+1$ groups of type 9, which split in $q-2$ classes of length $p$, and one class of length one (this is for $j=0$). 
\end{enumerate}

\subsection{The case \texorpdfstring{$C \leq \ker(\gamma) \neq A$}{C in ker}}
\label{subsub:G6-p-divides-ker}

Suppose now  $C \le \ker(\gamma) \ne  A$, so that we will have $\gamma(z) = \iota(c)^{s}$, for some $s \ne 0$. 
If $\gamma(b)$ is a (possibly trivial) $q$-element in $\gamma(G)$, then $b$ is a $q$-element in $G$, and we will have
\begin{equation*}
  \gamma(b) \in \Span{ \beta, \iota(b c^{m}) }
\end{equation*}
for some $m$.

If $\gamma(b) = \beta^{t}$, for some $t$, then $\gamma(G)$ is abelian, so that $(G, \circ)$ is of type 5 or 6. However,
\begin{equation*}
  b^{\ominus 1} \circ c \circ b
  =
  c^{\gamma(b) \iota(b)}
  =
  c^{\lambda} = c^{\circ \lambda}
  \ne
  c,
\end{equation*}
so that $(G, \circ)$ is not abelian, and thus of type 6.

We also have
\begin{equation*}
  b^{\ominus 1} \circ z \circ b
  =
  b^{- \gamma(b)^{-1} \gamma(z) \gamma(b)} z^{\gamma(b)} b
  \equiv_{\bmod{C}}
  z^{\lambda^{t}} = z^{\circ \lambda^{t}},
\end{equation*}
so that $t = 0$, as $(G, \circ)$ has to be of type 6. 
Therefore the kernel has order $p q$, $\gamma(G)=\gamma(Z)$ and $[Z,\gamma(Z)]=1$, so that by Proposition~\ref{prop:lifting} and Lemma~\ref{Lemma:gamma_morfismi} the GF's on $G$ are precisely the morphisms $Z\to \Aut(G)$, which are as many as the choices for $s$, namely $p-1$.

If $\gamma(b) = \beta^{t} \iota(b c^{m})^{l}$ for some $l \ne 0$ and $t$, replacing $b$ with $b c^{m}$ we see that we can take $m = 0$.

We have
\begin{equation*}
  b^{\ominus 1} \circ c \circ b
  = c^{\gamma(b) \iota(b)}
  = c^{\lambda^{l + 1}}
  = c^{\circ \lambda^{l + 1}}.
\end{equation*}
Then
\begin{equation*}
  b^{\ominus 1} \circ z \circ b
  =
  b^{- \gamma(b)^{-1} \gamma(z) \gamma(b)} z^{\gamma(b)} b
  \equiv_{\bmod{C}} z^{\lambda^{t}} 
  = z^{\circ \lambda^{t}}.
\end{equation*}
However
\begin{equation*}
  \gamma(b^{\ominus 1} \circ z \circ b)
  =
  \gamma(b)^{-1} \gamma(z) \gamma(b)
  =
  \iota(c)^{s \lambda^{l}}
  =
  \gamma(z)^{\lambda^{l}}.
\end{equation*}
It follows that $t = l$. 

The latter is also a sufficient condition in order to have that the map $\gamma$ defined by
$$\gamma(b^{m}c^{k}z^{n})=\beta^{mt}\iota(b^{ml}c^{ns})$$ 
satisfies the GFE. Indeed
\begin{align*} 
\gamma(b^{m}c^{k}z^{n})\gamma(b^{u}c^{v}z^{w})
&= \beta^{mt}\iota(b^{ml}c^{ns})\beta^{ut}\iota(b^{ul}c^{ws}) \\
&= \beta^{(m+u)t}\iota(b^{(m+u)l}c^{(n\lambda^{ul}+w)s}),
\end{align*}
\begin{align*}
\gamma((b^{m}c^{k}z^{n})^{\gamma(b^{u}c^{v}z^{w})}b^{u}c^{v}z^{w}) 
&= \gamma((b^{m}c^{k}z^{n})^{\beta^{ut}\iota(b^{ul}c^{ws})}b^{u}c^{v}z^{w})  \\
&= \gamma((b^{m}c^{\ast}z^{n\lambda^{ut}})b^{u}c^{v}z^{w})  \\
&= \gamma((b^{m+u}c^{\ast}z^{n\lambda^{ut}+w})  \\
&= \beta^{(m+u)t}\iota(b^{(m+u)l}c^{(n\lambda^{ut}+w)s}) ,
\end{align*}
and they are equal if and only if $l=t$.

As for $(G, \circ)$ we have that $Z_\circ \sim \diag(\lambda^{t+1}, \lambda^{t})$.
\begin{enumerate}
\item For $t = - 1$ we get groups of type 6, with $p$ choices for $B$ and $p-1$ choices for $s$.
\item For $q>2$ and $t = (q-1)/2$ we have $\lambda^{t+1} \lambda^{t} = \lambda^{2 t  + 1} = 1$, so $p (p -1)$ groups of type 9.
\item For $q>3$ for each of the remaining $q - 3$ values of $t$, we get $p (p - 1)$
  groups of type 8, so $(q-3)p(p-1)$ in total. They split in $2p(p-1)$ groups isomorphic to $G_s$, for every $s \in \mathcal{K}$.
\end{enumerate}

As to the conjugacy classes, write $\phi=\psi \iota(c^m) \mu$ for an automorphism of $G$, with $\psi$ and $\mu$ as in~\eqref{eq:aut-G-6}. Here $C=\ker(\gamma)$ is characteristic, so that by Lemma~\ref{lemma:generatori}, we can look at the action of $\phi$ on $\gamma$ defined on the generators $z$, $b$. 

Write $\mu^{-1} \iota(c^{m}) \mu = \iota(c^{m})^{r}$ for the commutation rule in $\Hol(\cC_p)$, where $1 \leq r \leq p-1$. Then
\begin{equation*}
\gamma^{\phi}(z) 
= \phi^{-1} \gamma(z^{\psi^{-1}}) \phi 
= \mu^{-1} \iota(c^{sk^{-1}}) \mu 
= \iota(c^{sk^{-1}})^{r},
\end{equation*}
so that $\gamma^{\phi}(z)=\gamma(z)$ if and only if $k=r$. Moreover
\begin{align*}
\gamma^{\phi}(b) 
&= \phi^{-1} \gamma(c^{m(1-\lambda^{-1})}b) \phi \\
&= \mu^{-1}\iota(c^{-m})  \gamma(b) \iota(c^m)\mu \\
&= \beta^{t} \mu^{-1}\iota(c^{-m}) \iota(b^{t}) \iota(c^m)\mu \\
&= \beta^{t} \mu^{-1}\iota(b^{t} c^{m(1-\lambda^{t})})\mu \\
&= \beta^{t} \iota(b^{t}) \iota(c^{m(1-\lambda^{t})})^{r} ,
\end{align*}
so that $\gamma^{\phi}(b)=\gamma(b)$ if and only if $t=0$ or $m=0$.

Therefore, if $t=0$, namely when $\ker(\gamma)$ has size $pq$, the stabiliser has order $p(p-1)$, and there is one orbit of length $p-1$. Otherwise, if $t \ne 0$, namely $\ker(\gamma)$ has size $p$, then the stabiliser has order $p-1$, and there are $q-1$ orbits of length $p(p-1)$.

We summarise, including the right regular representations, and doubling the numbers just obtained.

\begin{prop}
\label{prop:G6}
Let $G$ be a group of order $p^2q$, $p>2$, of type 6. Then in $\Hol(G)$ there are:
\begin{enumerate}
\item $2p$ groups of type 5, which split in two conjugacy classes of length $p$; 
\item $2p(p+2q-3)$ groups of type 6, which split in two conjugacy classes of length $1$, $2(2q-3)$ conjugacy classes of length $p$, two conjugacy classes of length $p-1$, and two conjugacy classes of length $p(p-1)$; 
\item $2(p(q-2)+1)$ groups of type 7, which split in two conjugacy classes of length $1$, and $2(q-2)$ conjugacy classes of length $p$;
\item if $q>3$, for every $s\in \mathcal{K}$ there are $4(1+p(p+q-3))$ groups of type 8 isomorphic to $G_s$ which split in $4$ conjugacy classes of length $1$, $4(q-2)$ conjugacy classes of length $p$, and $4$ conjugacy classes of length $p(p-1)$;
\item if $q>2$, $2(1+p(p+q-3))$ groups of type 9, which split in two conjugacy classes of length $1$, $2(q-2)$ conjugacy classes of length $p$, and two conjugacy classes of length $p(p-1)$.
\end{enumerate}
\end{prop}

\section{Prologue to Sections \ref{sec: 9}, \ref{sec: 8} and \ref{sec: 7}} 
\label{sec: 7,8,9}

In this section we collect some arguments which are common to the study of the groups of types 7, 8 and 9.

In these groups the Sylow $p$-subgroup $A$ of $G$ is
characteristic. With respect to a suitable basis $a_1, a_2$, the action of a generator $b$ of a Sylow $q$-subgroup $B$ on $A$ can be represented by the matrix \begin{equation*}
 Z =Z_k=
 \begin{bmatrix}
  \lambda & 0\\
  0 & \lambda^{k}\\
 \end{bmatrix},
\end{equation*}
where $\lambda \ne 1$ has order $q$, and $k\ne0$ is an integer modulo
$q$. If $k=1$ the type is 7, if $k=-1$ the type is 9, and if $k\neq 0,\pm 1$ the type is 8.  

\begin{remark}
\label{rem:k-inverse}
Note that if we choose as a basis $a_2, a_1$, then the action of the generator $b^{k^{-1}}$ of $B$ on $A$ is represented by the matrix $Z_{k^{-1}}$.
\end{remark}

\newtheorem{assump}[dummy]{Assumption}

For types 8 and 9, the gamma function $\gamma$
is of the form
\begin{equation}
\label{eq:gamma=iota}
\gamma(a)=\iota(a^{-\sigma})
\end{equation} 
for some $\sigma\in \End(A)$  (see
Sections~\ref{sec: 9},~\ref{sec: 8}). This is not always the case for
groups of type 7 (see Section~\ref{sec: 7}). Therefore, in this
section we will work under the following 
\begin{assump}
\label{assump}
 $G$ is a group of type 8 or 9, or a group of type 7 such that $\gamma(A) \le \Inn(G)$.
\end{assump}
Lemma~\ref{lemma:duality} shows that the condition  $\gamma(A) \le \Inn(G)$ implies~\eqref{eq:gamma=iota}.

 We will prove that every GF of the groups verifying
 Assumption~\ref{assump} admits an invariant Sylow $q$-subgroup B,
 so that it can be obtained as a
 \emph{lifting} of a RGF on $B$ or as a  \emph{gluing} of a RGF on $B$
 and the  map $a\mapsto\iota(a^{-\sigma})$, for some
 $\sigma\in\End(A)$. 
Note that
the assumptions (1), (2), (3) and (4) of Proposition~\ref{prop:more-lifting} are fulfilled for all groups under considerations with $A$ the  Sylow $p$-subgroup and $B$ any Sylow $q$-subgroup of $G$. 

To count each GF exactly once we will also need to determine the exact number of invariant Sylow $q$-subgroups in each case.

A proof of these facts will require a detailed analysis, that will be carried out in
several steps in the next subsections.  We will then complete the classification for the three types in Sections~\ref{sec: 9},~\ref{sec: 8},~\ref{sec: 7}. 
 
\subsection{Invariant Sylow $q$-subgroups of $G$}
\label{ssec:q-sylow-G789}

When $q\mid \Size{\ker(\gamma)}$, there is a Sylow $q$-subgroup contained in $\ker(\gamma)$,  and this is clearly invariant.

Consider thus the case $q\nmid
\Size{\ker(\gamma)}$, so that $\gamma(G)$ contains an element of order
$q$. 
In this subsection we give a characterization of the invariant Sylow $q$-subgroups of $G$, 
that we will then use  to count them
for the various types. 

 For the groups $G$ under consideration, an element of order $q$ of
 $\Aut(G)$ has the form $\iota(a_\ast)\beta$, where $a_{\ast} \in A$,
 and $\beta$ is an element of order $q$ is such that, with respect to the chosen basis, acts as
\begin{equation}
\label{eq:T-matrix}
T=\begin{bmatrix}
\lambda^{x_1}& 0\\
0 &  \lambda^{kx_2}\\
  \end{bmatrix},
\end{equation}
where $\lambda$ and $k$ are as above, and $x_1$ and $x_2$ are not both zero (this will be detailed  in Sections~\ref{sec: 9}, \ref{sec: 8} and \ref{sec: 7}). 

Let  $b \in G$ be an element of order $q$ such that 
 \begin{equation}
 \label{eq:gammab=}
 \gamma(b)=\iota(a_\ast)\beta.
 \end{equation}
A Sylow $q$-subgroup of $G$ is of the form $\Span{b^x}$, for $x \in A$, and it is invariant if and only if $(b^x)^{\gamma(b^x)}\in \Span{b^x}$. 
Now,
\begin{equation*}
\gamma(b^{x}) = \gamma(x^{-1+Z^{-1}}b)
= \iota(x^{(1-Z^{-1})T^{-1}\sigma}a_\ast)\beta,
\end{equation*}
so that
\begin{align*}
(b^{x})^{\gamma(b^{x})} &= (x^{-1+Z^{-1}}b)^{\iota(x^{(1-Z^{-1})T^{-1}\sigma}a_\ast)\beta} \\
&=(x^{-(1-Z^{-1})(1+T^{-1}\sigma(1-Z^{-1}))}a_\ast^{-(1-Z^{-1})}b)^{\beta} \\
&= x^{-(1-Z^{-1})(1+T^{-1}\sigma(1-Z^{-1}))T}a_\ast^{-(1-Z^{-1})T}b.
\end{align*}

Since  $(b^x)^j=x^{-1+Z^{-j}}b^j$
for all $j$, the last expression is in $\Span{b^x}$ if and only if 

$$ x^{-(1-Z^{-1})(1+T^{-1}\sigma(1-Z^{-1}))T}a_\ast^{-(1-Z^{-1})T}b= x^{-(1-Z^{-1})}b,$$
which, writing $M =$
$1-(1+T^{-1}\sigma(1-Z^{-1}))T$, can be rewritten as 
\begin{equation}
\label{eq:inv q-sylow case 1}
x^{(1-Z^{-1})M}= a_\ast^{(1-Z^{-1})T}.
\end{equation}

The number of solutions $x$ of the system~\eqref{eq:inv q-sylow case 1} is the number of invariant Sylow $q$-subgroups: a solution $x$ corresponds to the invariant Sylow $q$-subgroup $B=\Span{b^{x}}$.

When the kernel of $\gamma$  has size $p^2$, the existence of an invariant Sylow $q$-subgroup can be easily shown by noticing that the action of the group  $\gamma(G)$, of size $q$, on the set of Sylow $q$-subgroups, that has cardinality $p^2$,  admits at least one fixed point. This means that the system~\eqref{eq:inv q-sylow case 1} is always solvable.
 
When the kernel of $\gamma$ has size
$p$ or $1$  the system~\eqref{eq:inv q-sylow case 1} can be unsolvable
for some $a_{\ast}.$  
 However, we will show in~\ref{sssec:inv-q-Sylow-kerp-G789} \textcolor{blue}{and in~\ref{ssec:G9-ker-1}} that for all groups fulfilling Assumption~\ref{assump} the following are equivalent:
 \begin{enumerate}
 \item
   equation \eqref{eq:inv q-sylow case 1}  admits  a solution $x$ for a given $a_{\ast}$,
\item
the assignments given in \eqref{eq:gamma=iota} and \eqref{eq:gammab=} can be extended to a GF on $G$. 
\end{enumerate}
In fact, if the assignments in \eqref{eq:gamma=iota} and \eqref{eq:gammab=} can be extended to a GF, then   
$$\gamma(b^{m})=\gamma({(b^{m-1})}^{\gamma(b)^{-1}})\gamma(b),$$ 
and an inductive argument shows that
\begin{equation} 
\label{eq:ordine-q}
\gamma(b^{m})=\iota(a_{\ast}^{-A_m \sigma + 1+T^{-1}+ \cdots + T^{-(m-1)}}) \beta^{m} ,
\end{equation}
where 
$$A_m=\sum_{i=1}^{m-1}(1-Z^{-i})T^{-i} .$$ 
Since $\gamma(b^{q})=1$ and the center of $G$ is trivial,~\eqref{eq:ordine-q} yields
\begin{equation}
\label{eq:condizione-a_ast}
a_{\ast}^{A_q \sigma}=a_{\ast}^{1+T^{-1}+ \cdots + T^{-(q-1)}}.
\end{equation}

In Subsections~\ref{ssec:G789_kerp} and~\ref{ssec:G9-ker-1} we will see that the elements $a_{\ast}$ satisfying~\eqref{eq:condizione-a_ast} are exactly those for which the system~\eqref{eq:inv q-sylow case 1} admits solutions.
The case $\Size{\ker(\gamma)}=1$ is specific to the type 9 and it will be considered in Section~\ref{sec: 9}.

\subsection{The case \texorpdfstring{$\Size{\ker(\gamma)}=p^2$}{ker = p2}}
\label{ssec:G789-kernel p^2} As we already noticed, 
the action on $G$ of the group $\gamma(G)$ of order $q$
 fixes at least one of the $p^2$ Sylow $q$-subgroups of $G$, say $B=\Span{b}$. Therefore, by Proposition~\ref{prop:lifting}, each $\gamma$ on $G$ is the lifting of at least one RGF defined on a such Sylow $q$-subgroup $B$. 
To count each GF exactly once, we need to compute the number of invariant Sylow $q$-subgroups.

In this case, the element $\gamma(b)$ of order $q$ acts
trivially on $B$, so that $a_{\ast}=1$ and  $[B, \gamma(B)] = \Set{1}$; by Lemma~\ref{Lemma:gamma_morfismi}, the RGF's on $B$ are thus precisely the morphisms.  

 Let $\gamma(b)_{|A}=\beta$, where $\beta$ is as in~\eqref{eq:T-matrix}.

Equation~\eqref{eq:inv q-sylow case 1} yields
$$ x^{(1-Z^{-1})M}= 1, $$
where
$$ M = 1-T = \begin{bmatrix}
1-\lambda^{x_1} & 0\\
0 & 1-\lambda^{kx_2}\\
\end{bmatrix} . $$
Since $\det(1-Z^{-1})\ne 0$, we obtain that
\begin{enumerate}
\item there is a unique solution, namely a unique invariant Sylow $q$-subgroup, when both $x_1,x_2 \neq 0$; 
\item there are $p$ solutions, that is, $p$ invariant Sylow $q$-subgroups, when either $x_1=0$ or $x_2=0$.
\end{enumerate}

The action of $b$ on $A$ with respect to the operation $\circ$ is given by
\begin{equation*}
  b^{\ominus 1} \circ a \circ b = b^{-\gamma(b)^{-1} \gamma(a) \gamma(b)} a^{\gamma(b)} b=a^{\gamma(b)\iota(b)},
 \end{equation*}
and denoting by $Z_{\circ}$ its associated matrix, we have
  \begin{equation*}
   Z_{\circ} \sim
  \begin{bmatrix}
  \lambda^{1+x_1} & 0\\
    0 & \lambda^{k+kx_2}\\
  \end{bmatrix}.
\end{equation*}

The enumeration of the groups $(G,\circ)$ depends on the type of the
groups $G$, that is, on the parameter $k$. 
We will treat the various types separately in Subsections~\ref{ssec:G9-kernel p2}, \ref{ssec:G8-kernel p2} and ~\ref{ssec:G7-kernel p2}.

\subsection{The case \texorpdfstring{$\Size{\ker(\gamma)}=p$}{ker = p}}
\label{ssec:G789_kerp}

Here $\Size{\gamma(G)}=pq$. Write $\gamma(b)=\iota(a_{\ast})\beta$, where $b$ has order $q$, and $a_{\ast}$ and $\beta$ are as in Subsection \ref{ssec:q-sylow-G789}. Equation~\eqref{eq:sigma} yields
\begin{equation}
\label{eq:sigmaTZ}
\sigma T(\sigma-1)=(\sigma-1)TZ\sigma.
\end{equation}

The condition $\Size{\ker(\gamma)}=p$ means that $\Size{\ker(\sigma)}=p$, so let $\ker(\sigma)=\Span{v}$. Using~\eqref{eq:sigmaTZ} we obtain $v^{-TZ\sigma}=1$, therefore $v^{-TZ}\in\Span{v}$, namely, $v$ is an eigenvector for $TZ$. 
If $TZ$ is not scalar, then its eigenspaces are $\Span{a_1}$ and
$\Span{a_2}$, so either $v\in \Span{a_1}$ or $v\in
\Span{a_2}$. When  $TZ$ is scalar, $v$ can be
any non-zero element of $A$. 

We proceed by distinguishing three cases, namely when $\ker(\sigma)=\Span{a_1}$, when $\ker(\sigma)=\Span{a_2}$, and lastly when $\ker(\sigma)$ is generated by $v=a_1^{x}a_2^{y}$, where $x,y \ne 0$.

\underline{Case $A$}: $\ker(\sigma)=\Span{a_1}$. Here
$a_1^{\sigma}=1$, and  $a_2^\sigma=a_1^\mu a_2^\nu$ for some $\mu,
\nu$ not both $0$. Evaluating
equation~\eqref{eq:sigmaTZ} in $a_2$ we get 
\begin{equation}
\label{eq:system_case_a}
\begin{cases}
\mu(\nu-\lambda^{x_1-x_2k})=\mu(\nu-1)\lambda^{k}\\
(\nu-1)\nu=(\nu-1)\nu\lambda^{k}.
\end{cases}
\end{equation}
Since $\lambda^{k} \ne 1$, the second equation gives that  either $\nu=0$ or $\nu=1$.
If $\nu=0$, then $\mu\ne0$ and $x_1-x_2k=k$.
If $\nu=1$, then either $\mu=0$ or $\mu\neq 0$ and ${x_1}={x_2k}$.

\underline{Case $A^*$}: $\ker(\sigma)=\Span{a_2}$. 
This case can be reduced to case $A$ according to Remark~\ref{rem:k-inverse}, obtaining a system as in~\eqref{eq:system_case_a} with $\lambda^{k^{-1}}$ replacing $\lambda^{k}$.

When $k=k^{-1}$, that is for the groups $G$ of type 7 and 9, the conditions are the same, so we will double the results we will obtain in case $A$.

For the groups $G$ of type 8, $k^{-1} \ne k$, therefore we will sum the results we will obtain in case $A$ for $k$ with the same results for $k^{-1}$. 

\underline{Case $B$}: $\ker(\sigma)=\Span{a_1^{x}a_2^{y}}$, where $x, y \ne 0$. Here $TZ$ is scalar, namely $x_1=x_2 k+k-1$. 
We can replace the generator $v = a_1^{x}a_2^{y}$ of the kernel
  by $v^{x^{-1}}=a_1a_2^{yx^{-1}}$, and thus assume $v=a_1a_2^{z}$,
  for some $z \neq 0$. Rescaling $a_{2}$, we can also assume $z = 1$,
  keeping in mind that this covers $p-1$ cases here.
Therefore, here $\sigma$ is defined as
$$ a_1^\sigma = a_1^{-\mu}a_2^{-\nu},\ a_2^\sigma = a_1^{\mu}a_2^{\nu} ,$$
for some $\mu, \nu$ not both $0$. 
Evaluating equation~\eqref{eq:sigmaTZ} on $a_2$ we get

\begin{equation}
\label{eq:system_case_b}
\begin{cases}
\mu(-\mu\lambda^{-1}-\lambda^{-1}+\nu\lambda^{-k})=\mu(-\mu +\nu-1)\\
\nu(-\mu\lambda^{-1}-\lambda^{-k}+\nu\lambda^{-k})=\nu(-\mu +\nu-1).
\end{cases}
\end{equation}

If $\nu = 0$ then $\mu \neq 0$ and we get $-(\mu+1)\lambda^{-1}=-(\mu+1)$, that is, $\mu=-1$, so $A^\sigma=\Span{a_1}$.
If $\nu \neq 0$ and $\mu=0$, we get $(\nu-1)\lambda^{-k}=\nu-1$, that is, $\nu=1$ and $A^\sigma=\Span{a_2}$. 
Lastly, if both $\nu, \mu \neq 0$, then if $k=1$ we get $\mu+1-\nu=0$, and if $k\neq 1$ the system has no solution. 

If $k=1$, namely $G$ is of type 7, then the condition $\mu+1-\nu=0$ includes also the cases above in which $\mu=-1$ and $\nu=0$, or $\mu=0$ and $\nu=1$.

If $k\ne 1$, namely $G$ is of type 8 or 9, the last case does not happen. Moreover we reduce the case $(\mu,\nu)=(0,-1)$ to the case $(\mu, \nu)=(1, 0)$ according to Remark~\ref{rem:k-inverse}, obtaining a system as in~\eqref{eq:system_case_b} with $\lambda^{k^{-1}}$ replacing $\lambda^{k}$.

As for case $A^*$, for the groups $G$ of type 9 we will double the results we will obtain in case $(\mu, \nu)=(1, 0)$, and for the groups $G$ of type 8 we will sum the results we will obtain in case $(\mu, \nu)=(1, 0)$ for $k$ with the same results for $k^{-1}$.

For a group $G\simeq G_{k}$ which fulfills
Assumption~\ref{assump}, we sum up our analysis as follows: 
\begin{itemize}
\item Case A: $\ker(\sigma)=\Span{a_1}$.
 \begin{enumerate}
  \item[(A1)] $\nu=0$, $\mu \neq 0$, $x_1-x_2 k= k$;
  \item[(A2)] $\nu=1$, $\mu \neq 0$, $x_1=x_2 k$;
  \item[(A3)] $\nu=1$, $\mu = 0$.
 \end{enumerate}
\item Case A$^\ast$: $\ker(\sigma)=\Span{a_2}$.\\
It is equivalent to considering the case A for $G\simeq G_{k^{-1}}$, as explained above. 
Therefore, replacing $\lambda^k$ with $\lambda^{k^{-1}}$ in (A1), (A2) and (A3), here we obtain the subcases (A1$^\ast$), (A2$^\ast$) and (A3$^\ast$).
\item Case B: $\ker(\sigma)=\Span{a_1a_2}$.
 \begin{enumerate}
  \item If $k\ne 1$ (namely $G$ is of type 8 or 9):
  \begin{enumerate}
  \item[(B1) \phantom{}] $\nu=0$, $\mu =-1$, $x_1=x_2 k +k-1$;
  \item[(B1$^\ast$)]
  $\nu=1$, $\mu = 0$, $x_1=x_2 k +k-1$.
  As explained above, this case is equivalent to considering the case (B1) for $G\simeq G_{k^{-1}}$.
  \end{enumerate}
  \item If $k=1$ (namely $G$ is of type 7):
  \begin{enumerate}
  \item[(B2)] $\mu+1-\nu=0$, $x_1=x_2$.
  \end{enumerate}
 \end{enumerate}
\end{itemize}

\subsubsection{Invariant Sylow $q$-subgroups}
\label{sssec:inv-q-Sylow-kerp-G789}
We are now ready to prove that, when $|\ker(\gamma)|=p$, the groups fulfilling 
Assumption~\ref{assump} always have at least one invariant Sylow $q$-subgroup and to determine their number in terms of our assignments $\gamma(a)=\iota(a^\sigma)$ and $\gamma(b)=\iota(a_{\ast})\beta$.
 
Subsection~\ref{ssec:q-sylow-G789} yields that, for $x\in A$, the Sylow $q$-subgroup $\Span{b^{x}}$ is invariant if and only if $x$ is a solution of~\eqref{eq:inv q-sylow case 1}:
$$ x^{(1-Z^{-1})M}=a_{\ast}^{(1-Z^{-1})T} ,$$
where  $M=1-(1+T^{-1}\sigma(1-Z^{-1}))T$, and $\det(1-Z^{-1})\ne 0$.

If  $\det(M)\ne 0$, then  the system~\eqref{eq:inv q-sylow case 1} admits a unique solution for each $a_\ast$, or equivalently, for each choice of $\gamma(b)$. 

If $\det(M)=0$, then for some  $a_{\ast}$'s  the system~\eqref{eq:inv q-sylow case 1} admits $r=p$ or $p^{2}$ solutions, and for others it has no solution.
However, we show that in the latter case the values for $\gamma(a), \gamma(b)$ do not extend
to a full GF on $G$.
 
In fact, if $\gamma$ is a GF on $G$ such that $\gamma(b)=\iota(a_\ast)\beta$, then $\gamma$
satisfies~\eqref{eq:ordine-q}, and~\eqref{eq:condizione-a_ast}.  
We will show that if the system~\eqref{eq:inv q-sylow case 1} has no solution for some $a_\ast$, then $a_{\ast}$ does not satisfy the condition~\eqref{eq:condizione-a_ast}, and thus the assignments $\gamma(a)=\iota(a^{-\sigma})$ and $\gamma(b)=\iota(a_\ast)\beta$ cannot be extended to a GF.

Write $a_{\ast}=a_1^{x}a_2^{y}$. We distinguish several cases.
\begin{description}
\item[Case A] 
Here $\ker(\sigma)=\Span{a_1}$, and 
$$ M= \begin{bmatrix}
1-\lambda^{x_1} & 0\\
-\mu\lambda^{x_1-x_2 k}(1-\lambda^{-1}) & 1-\lambda^{x_2 k}-\nu(1-\lambda^{-k})
\end{bmatrix}.$$
According to the division into subcases, we have
\begin{description}
\item[A1] $ \det(M)= (1-\lambda^{x_1})(1-\lambda^{x_1-k})$, so that there is a unique invariant Sylow $q$-subgroup when $x_1 \ne 0,k$.

If $x_1=0$, then the system~\eqref{eq:inv q-sylow case 1} admits
(a number $p$ of) solutions if and only if
$a_{\ast}=a_1^{-\mu y} a_2^{y}$. In this case there are $p$ invariant
Sylow $q$-subgroups. 
Moreover, if $a_{\ast}=a_1^{x} a_2^{y}$, the condition~\eqref{eq:condizione-a_ast} yields that $x = -\mu y$. Therefore, by the discussion above, the case in which there are no invariant Sylow $q$-subgroups does not arise.

If $x_1=k$, then~\eqref{eq:inv q-sylow case 1} admits (a number $p$ of) solutions if and only if $a_{\ast}=a_1^{x}$, and in this case there are $p$ invariant Sylow $q$-subgroups.~\eqref{eq:condizione-a_ast} yields that $y = 0$, therefore the case in which there are no invariant Sylow $q$-subgroups does not arise.

\item[A2] $ \det(M)= (1-\lambda^{x_1})(\lambda^{-k}(1-\lambda^{x_1+k}))$, so that there is a unique invariant Sylow $q$-subgroup when $x_1 \ne 0,-k$. Note that here necessarily $x_1\ne 0$, otherwise we would also have $x_2=0$, namely $\beta=1$.

If $x_1=-k$, then the system~\eqref{eq:inv q-sylow case 1} admits (a number $p$ of) solutions if and only if $a_{\ast}=a_1^{x}$, and in this case there are $p$ invariant Sylow $q$-subgroups. Also here the $a_{\ast}$'s for which~\eqref{eq:inv q-sylow case 1} has no solutions are precisely those for which~\eqref{eq:condizione-a_ast} is not satisfied, in fact here~\eqref{eq:condizione-a_ast} yields $y=0$.

\item[A3] $ \det(M)= (1-\lambda^{x_1})(\lambda^{-k}(1-\lambda^{x_2 k+k}))$, so that there is a unique solution when $x_1 \ne 0$ and $x_2 \ne -1$.

If either $x_1=0$ and $x_2 \ne -1$, or $x_1\ne 0$ and $x_2 = -1$, then $(1-Z^{-1})M$ has rank $1$, and
the system~\eqref{eq:inv q-sylow case 1} admits $p$ solutions if and only if $a_{\ast}=a_2^{y}$ in the first case, and $a_{\ast}=a_1^{x}$ in the second case.
Once again, the $a_{\ast}$'s for which~\eqref{eq:inv q-sylow case 1} has no solutions are precisely those for which~\eqref{eq:condizione-a_ast} is not satisfied, in fact here~\eqref{eq:condizione-a_ast} yields $x=0$ in the case $x_1=0$ and $x_2 \ne -1$, and $y=0$ in the case $x_1 \ne 0$ and $x_2 = -1$.

If $x_1=0$ and $x_2=-1$, then $(1-Z^{-1})M$ has rank $0$, and the system~\eqref{eq:inv q-sylow case 1} admits $p^2$ solutions if and only if $a_{\ast}^{(1-Z^{-1})T}=1$, namely when $a_{\ast}=1$. Moreover, the condition~\eqref{eq:condizione-a_ast} yields $x,y=0$, so that the case in which the system has no solution does not arise. 
\end{description}

\item[Case B]
Here $\ker(\sigma)=\Span{a_1a_2}$ and~\eqref{eq:inv q-sylow case 1} yields
$$ M= \begin{bmatrix}
1-\lambda^{x_1}+\mu(1-\lambda^{-1}) & \lambda^{-x_1+k x_2}\nu(1-\lambda^{-k})\\
-\lambda^{x_1-k x_2}\mu(1-\lambda^{-1}) & 1-\lambda^{k x_2}-\nu(1-\lambda^{-k})
\end{bmatrix}.$$
According to the division into subcases, we have

\item[B1] $ \det(M)= \lambda^{-1}(1-\lambda^{x_1+1})(1-\lambda^{x_1+1-k})$, and there exists a unique invariant Sylow $q$-subgroup when $x_1 \ne -1,k-1$.

If $x_1=-1,k-1$, then there are $p$ invariant Sylow $q$-subgroups when $a_{\ast}=a_1^{x}a_2^{x}$ in the case $x_1=-1$, and $a_{\ast}=a_1^{x}$ in the case $x_1=k-1$. 
The other cases, namely those for which there are no invariant Sylow $q$-subgroups, do not arise, in fact the condition~\eqref{eq:condizione-a_ast} yields precisely $x=y$ in the case $x_1=-1$, and $y=0$ in the case $x_1=k-1$. 

\item[B2] $\det(M)=(\lambda^{x_1}-1)\lambda^{-1}(\lambda^{x_1+1}-1)$, and there exists a unique invariant Sylow $q$-subgroup when $x_1\ne 0,-1$. Note that here necessarily $x_1\ne 0$, otherwise we would also have $x_2=0$.

If $x_1=-1$, then there are $p$ invariant Sylow $q$-subgroups when $a_{\ast}=a_1^{x}a_2^{x}$. 
The other cases, namely those for which there are no invariant Sylow $q$-subgroups, do not arise, in fact the condition~\eqref{eq:condizione-a_ast} yields precisely $x=y$. 
\end{description}

By the discussion above we get the following.
\begin{prop}
\label{prop:inv_q-sylow_kerp}
If $G$ is a group satisfying Assumption~\ref{assump},
and $\gamma$ is a GF on $G$ with $\Size{\ker(\gamma)}=p$, then the
number of invariant Sylow $q$-subgroups is 
\begin{enumerate}
\item[\rm{(A1)}] $1$ when $x_1\neq 0, k$ and $p$ otherwise.
\item[\rm{(A2)}] $1$ when $x_1 \neq -k$ and $p$ otherwise.
\item[\rm{(A3)}] $1$ when $x_1\neq 0$ and $x_2\neq -1$, $p^2$ when $x_1= 0$ and $x_2= -1$, and $p$ otherwise.
\item[\rm{(A1$^\ast$)}] $1$ when $x_2\neq 0, k^{-1}$ and $p$ otherwise.
\item[\rm{(A2$^\ast$)}] $1$ when $x_2 \neq -k^{-1}$ and $p$ otherwise.
\item[\rm{(A3$^\ast$)}] $1$ when $x_2\neq 0$ and $x_1\neq -1$, $p^2$ when $x_2= 0$ and $x_1= -1$, and $p$ otherwise.
\item[\rm{(B1) \phantom{}}] $1$ when $x_1 \neq  -1, k-1$ and $p$ otherwise.
\item[\rm{(B1$^\ast$)}] $1$ when $x_2 \neq  -1, -1+k^{-1}$ and $p$ otherwise.
\item[\rm{(B2) \phantom{}}] $1$ when $x_1\ne -1$ and $p$ otherwise.
\end{enumerate}
\end{prop}

\subsubsection{Enumerating the GF's}
\label{sssec:enumerate-GF-789}
We have shown that  there is always  a Sylow
$q$-subgroup $B$ which is invariant under $\gamma(B)$, and
Proposition~\ref{prop:inv_q-sylow_kerp} yields the exact number of
such invariant Sylow $q$-subgroups. 
Moreover, every $a \in A$ satisfies equation~\eqref{eq:gamma=iota},
therefore by Proposition~\ref{prop:more-lifting}, all the GF's here can be obtained as a gluing of a function $\gamma_A$ on $A$ defined as in~\eqref{eq:gamma=iota} and a RGF $\gamma_B$ on $B$, where $B$ is an invariant Sylow $q$-subgroup, whenever $\gamma_A$ and $\gamma_B$ satisfy~\eqref{eq:sigmaTZ}.

To enumerate the GF's we can count the possible couples
$(\gamma_A,\gamma_B)$ with the properties above, taking into account
that every such choice defines a unique $\gamma$\AECh{}{,} and
that a given $\gamma$ built in this way is obtained $s$ times, where $s$ is the number of invariant Sylow $q$-subgroups of $G$.
Thus to obtain the number of distinct GF's on $G$ we count the choices for $(\gamma_A,\gamma_B)$ as above, and then divide this number by $s$.

Let $Z_\circ$ denote the action of $b$ on $A$ in $(G,
\circ)$. We have  
\allowdisplaybreaks
 \begin{align*}
b^{\ominus 1} \circ a \circ b
 &= (b^{\gamma(b)^{-1} \gamma(a) \gamma(b)})^{-1} a^{\gamma(b)} b\\
 &= (b^{\iota(a^{-\sigma} )\beta})^{-1} a^{\gamma(b)} b\\
 &= ((a^{-\sigma(-1+Z^{-1})}b)^{\beta})^{-1} a^{\gamma(b)} b\\
 &= (a^{-\sigma(-1+Z^{-1})T}b)^{-1} a^{T} b\\
 &= b^{-1} a^{\sigma(-1+Z^{-1})T+T} b\\
 &= a^{(\sigma(1-Z)+Z)T},
\end{align*}
and since $a^{\circ k}=a^k$, with respect to the basis $\Span{a_1, a_2}$ of $(A, \circ)$ we have
\begin{equation}
\label{eq:Zcirc}
Z_\circ = (\sigma(1-Z)+Z)T.
\end{equation}

\textbf{Case A.} Here $\ker(\sigma)=\Span{a_1}$ and equality~\eqref{eq:Zcirc} yields 
$$Z_\circ=
\begin{bmatrix}
 \lambda^{x_1 + 1}& 0\\
\mu(1-\lambda)\lambda^{x_1} & \lambda^{x_2 k}(\nu(1-\lambda^k)+\lambda^{k})\\
  \end{bmatrix}.
$$

\begin{enumerate}
\item[(A1)] We have $p-1$ choices for $\sigma$, and
$$
 Z_\circ \sim \begin{bmatrix}
 \lambda^{x_1 +1}& 0\\
0 & \lambda^{x_1}\\
  \end{bmatrix}.
$$

\item[(A2)] We have $p-1$ choices for $\sigma$, and
$$
 Z_\circ \sim \begin{bmatrix}
 \lambda^{x_1 +1}& 0\\
0 & \lambda^{x_1}\\
  \end{bmatrix}.
$$

\item[(A3)]
We have $1$ choice for $\sigma$, and
$$  Z_\circ \sim
\begin{bmatrix}
\lambda^{x_1+1}& 0\\
0 &\lambda^{x_2 k}\\
\end{bmatrix}.
$$
\end{enumerate}

\textbf{Case B.} Here $\ker(\sigma)=\Span{a_1a_2}$, $x_1=x_2 k +k-1$, and equality~\eqref{eq:Zcirc} yields 
$$Z_\circ=
\begin{bmatrix}
 \lambda^{x_1 + 1}-\lambda^{x_1}\mu(1-\lambda)& -\lambda^{k x_2} \nu (1-\lambda^k)\\
 \lambda^{x_1}\mu(1-\lambda) & \lambda^{k x_2 + k }+\lambda^{k x_2}\nu(1-\lambda^k)\\
  \end{bmatrix}.
$$
\begin{enumerate}
\item[(B1)] Here $k\ne 1$. We have $p-1$ choices for $\sigma$, and
$$Z_\circ=
\begin{bmatrix}
 \lambda^{x_1}& 0\\
-\lambda^{x_1}(1-\lambda) & \lambda^{x_2 k+k}\\
  \end{bmatrix}
  \sim
\begin{bmatrix}
\lambda^{x_1} & 0 \\
0 & \lambda^{x_1+1}\\
\end{bmatrix}.
$$

\item[(B2)] Here $k=1$, $x_1=x_2$, $\mu+1=\nu$, and we have $p(p-1)$ choices for 
$$\sigma=
\begin{bmatrix}
 -\mu & -\mu-1\\
  \mu & \mu+1\\
  \end{bmatrix}.$$
We have
$$Z_\circ=
\begin{bmatrix}
 \lambda^{x_1 + 1}-\lambda^{x_1}\mu(1-\lambda)& -\lambda^{x_1} (\mu+1) (1-\lambda)\\
 \lambda^{x_1}\mu(1-\lambda) & \lambda^{x_1 + 1}+\lambda^{x_1}(\mu+1)(1-\lambda)\\
  \end{bmatrix},
$$
and 
$$ Z_\circ \sim \begin{bmatrix}
\lambda^{x_1+1} & 0 \\
0 & \lambda^{x_1} 
\end{bmatrix} .$$
\end{enumerate}

\subsubsection{Conjugacy classes}
\label{sssec:789-conj-classes}

Here we exhibit a general scheme to compute the
conjugacy classes for 
groups which satisfy the assumptions of this section in the case
$\Size{\ker(\gamma)}=p$, when the automorphisms are of the form
$\phi=\iota(x)\delta$, where $x \in A$ and $\delta_{|A} \in
\GL(2,p)$. We will show in
Subsections~\ref{ssec:G9_kerp}, \ref{ssec:G8_kerp}
and~\ref{ssec:G7_kerp-inn} that this scheme can be applied
to the groups of types 9, 8 and to the groups $G$ of
  type 7, when $\gamma(A)\leq \Inn(G)$. 

Suppose thus that an automorphism of $G$ has the form $\phi=\iota(x)\delta$, where $x \in A$ and $\delta \in \GL(2,p)$. We have
\begin{align*} 
\gamma^{\phi}(a)
&=\phi^{-1}\gamma(a^{\delta^{-1}})\phi \\
&=\delta^{-1} \iota(a^{-\delta^{-1}\sigma}) \delta  \\
&= \iota(a^{-\delta^{-1}\sigma\delta}),
\end{align*}
and
\begin{align*} 
\gamma^{\phi}(b)
&=\phi^{-1}\gamma(x^{1-Z^{-1}}b)\phi \\
&=\phi^{-1} \iota(x^{-(1-Z^{-1})T^{-1}\sigma})\beta \phi \\
&=\delta^{-1} \iota(x^{-1+T^{-1}-(1-Z^{-1})T^{-1}\sigma})\beta \delta \\
&=\iota(x^{(-1+T^{-1}-(1-Z^{-1})T^{-1}\sigma)\delta}) \delta^{-1} \beta \delta.
\end{align*}
Write $H =
  -1+T^{-1}-(1-Z^{-1})T^{-1}\sigma$. Setting
$\gamma^{\phi}(a)=\iota(a^{-\sigma})$ and $\gamma^{\phi}(b)=\beta$, we
obtain that $\phi$ stabilises $\gamma$ if and only if the conditions 
\begin{align}
& [\sigma, \delta] = 1, \label{eq:G-7-ker-p-conj-clas-1} \\
& [\beta, \delta] = 1, \label{eq:G-7-ker-p-conj-clas-2} \\
& x^{H\delta}  = 1, \label{eq:G-7-ker-p-conj-clas-3}
\end{align}
hold.

We now distinguish two cases, namely when
$\delta$ is diagonal  
and when $\delta$ is not necessarily diagonal. 

\subsubsection{$\delta$ is a diagonal matrix}
Suppose first that $\delta$ is diagonal, namely $\delta=\diag(\delta_{11}, \delta_{22})$; as we will see in Subsections~\ref{ssec:G9_kerp} and~\ref{ssec:G8_kerp}, this will be the case for the groups $G$ of type 9 and 8, therefore we do not consider here the case (B2).

In this case equation~\eqref{eq:G-7-ker-p-conj-clas-2} is satisfied for every $\delta$.
In  both cases (A) and (B1), the
condition~\eqref{eq:G-7-ker-p-conj-clas-1} yields
$\mu\delta_{22}^{-1}\delta_{11}=\mu$, so that $\mu=0$ or $\delta$ is
scalar. 
Note that $\mu=0$ only in the case (A3), thus we consider any $\delta$ in this case, and $\delta$ scalar in the cases (A1), (A2) and (B1).

In the case (A1) we have
$$ H=\begin{bmatrix}
-1+\lambda^{-x_1} & 0 \\
-\mu(1-\lambda^{-k})\lambda^{k+x_1} & -1+\lambda^{k+x_1}
\end{bmatrix} ,$$
so that~\eqref{eq:G-7-ker-p-conj-clas-1} has one solution if $x_1 \ne 0, k$, and $p$ solutions if $x_1=0, k$. 

In the case (A2),
$$ H=\begin{bmatrix}
-1+\lambda^{-x_1} & 0 \\
-\mu(1-\lambda^{-k})\lambda^{-x_1} & -1+\lambda^{-k-x_1}
\end{bmatrix} ,$$
and~\eqref{eq:G-7-ker-p-conj-clas-1} has one solution if $x_1 \ne -k$, and $p$ solutions if $x_1= -k$. 

In the case (A3),
$$ H=\begin{bmatrix}
-1+\lambda^{-x_1} & 0 \\
0 & -1+\lambda^{-k-kx_2}
\end{bmatrix} ,$$
and~\eqref{eq:G-7-ker-p-conj-clas-1} has one solution if $x_1 \ne 0$ and $x_2 \ne -1 $, $p^2$ solutions if $x_1= 0$ and $x_2=-1$, and $p$ solutions otherwise. 

In the case (B1),
$$ H=\begin{bmatrix}
-1+\lambda^{-x_1-1} & 0 \\
(1-\lambda^{-k})\lambda^{k-1-x_1} & -1+\lambda^{k-1-x_1}
\end{bmatrix} ,$$
and~\eqref{eq:G-7-ker-p-conj-clas-1} has one solution if $x_1 \ne -1, k-1$, and $p$ solutions otherwise.

\subsubsection{$\delta$ is a (not necessarily diagonal) matrix}
Suppose now that $\delta$ is an arbitrary matrix. As we
will see in Subsection~\ref{ssec:G7_kerp-inn} this will be the case
for the groups $G$ of type 7, therefore we do not consider here the
case (B1). 

If we are in  case (A1), then~\eqref{eq:G-7-ker-p-conj-clas-1} yields $\delta_{12}=0$ and $\delta_{11}=\delta_{22}$, and~\eqref{eq:G-7-ker-p-conj-clas-2} yields $\delta_{21}=0$. Therefore $\delta$ is scalar, and with the same computations above (taking $k=1$) we obtain that $H$ has rank $2$ if $x_1 \ne 0, 1$, and $1$ otherwise. 

If we are in  case (A2),
then~\eqref{eq:G-7-ker-p-conj-clas-1} yields $\delta_{12}=0$ and
$\delta_{21}=\mu(\delta_{22}-\delta_{11})$. Since in the case (A2) $T$
is scalar, equation~\eqref{eq:G-7-ker-p-conj-clas-2} is satisfied for
every $\delta$. Moreover $H$ has rank $2$ if $x_1 \ne -1$, and $1$
otherwise.

If we are in  case (A3),
then~\eqref{eq:G-7-ker-p-conj-clas-1} yields $\delta_{12},
\delta_{21}=0$, namely $\delta$ is diagonal. In particular
~\eqref{eq:G-7-ker-p-conj-clas-2} is satisfied. $H$ has rank $2$ if
$x_1 \ne 0$ and $x_2 \ne -1$, $0$ if $x_1=0$ and $x_2=-1$, and $1$
otherwise.

If we are in case (B2), then~\eqref{eq:G-7-ker-p-conj-clas-1} yields
$$
\begin{cases}
\mu \delta_{12} = -(\mu+1)\delta_{21} \\
\mu(\delta_{12}-\delta_{11})=\mu(\delta_{21}-\delta_{22})\\
(\mu+1)(\delta_{12}-\delta_{11})=(\mu+1)(\delta_{21}-\delta_{22}) .
\end{cases}
$$
If $\mu=0$ then $\delta_{21}=0$ and
$\delta_{12}=\delta_{11}-\delta_{22}$; if $\mu=-1$, then
$\delta_{12}=0$ and $\delta_{21}=\delta_{22}-\delta_{21}$; if $\mu \ne
0, -1$, then $\delta_{12}=-\frac{\mu+1}{\mu}\delta_{21}$ and
$\frac{2\mu+1}{\mu}\delta_{21}=\delta_{22}-\delta_{11}$. Therefore, in
all  cases we have one choice for the elements
$\delta_{12}, \delta_{21}$, and $(p-1)^{2}$ choices for
$\delta_{11},\delta_{22}$. 
Since $T$ is scalar here, ~\eqref{eq:G-7-ker-p-conj-clas-2} is always satisfied. Moreover
$$
H=\begin{bmatrix}
-1+\lambda^{-x_1}+\mu(1-\lambda^{-1})\lambda^{-x_1} & (\mu+1)(1-\lambda^{-1})\lambda^{-x_1} \\
-\mu(1-\lambda^{-1})\lambda^{-x_1} & -1+\lambda^{-x_1}+(\mu+1)(1-\lambda^{-1})\lambda^{-x_1}
\end{bmatrix}
$$
has determinant $(1-\lambda^{-x_1})(1-\lambda^{-x_1-1})$. Since $x_1 \ne 0$, we obtain that $H$ has rank $2$ if $x_1 \ne -1$, and $1$ otherwise.

\section{Type 9} 
\label{sec: 9}

Here $q\mid p-1$, where $q>2$, and $G=(\cC_{p} \times \cC_{p}) \rtimes_{D_1} \cC_{q}$. The Sylow $p$-subgroup $A=\Span{a_1, a_2}$ of $G$ is characteristic, and if $a_1, a_2\in A$ are in the eigenspaces of the action of a generator $b$ of a Sylow $q$-subgroup $B$ on $A$, then this action can be represented by a non-scalar diagonal matrix $Z$, with no eigenvalues $1$ and $\det(Z)=1$.

For all the section, we consider $A=\Span{a_1, a_2}$, where $a_1, a_2$ are eigenvectors for $\iota(b)$. With respect to that basis, we have
\begin{equation*}
 Z =
 \begin{bmatrix}
  \lambda & 0\\
  0 & \lambda^{-1}\\
 \end{bmatrix},
\end{equation*}
where $\lambda \ne 1$ has order $q$.

The divisibility condition on $p$ and $q$ implies that $(G,\circ)$ can  be of type 5, 6, 7, 8 and 9.

According to Subsections 4.1 and 4.3 of~\cite{classp2q}, we have 
\begin{equation*}
\Aut(G) = (\Hol(\cC_{p}) \times \Hol(\cC_{p})) \rtimes \cC_{2} .
\end{equation*}

The Sylow $p$-subgroup of $\Aut(G)$ has order $p^2$ and is characteristic, so, since $G$ has trivial center, all its elements are conjugation by elements of $A$.

If $\gamma$ is a GF on $G$, then $\gamma_{|A}: A\to\Inn(G)\leq \Aut(G)$ is a RGF, as $A$ is characteristic in $G$. Moreover, Lemma~\ref{Lemma:gamma_morfismi} yields that $\gamma_{|A}$ is a morphism, as $\iota(A)$ acts trivially on the abelian group $A$.
Therefore, for each gamma function $\gamma$ there exists $\sigma\in \End(A)$ such that
\begin{equation}
\label{eq:gamma=iota_9}
\gamma(a)=\iota(a^{-\sigma})
\end{equation} 
for each $a\in A.$

\subsection{Duality}
\label{ssec:G9_duality}
Since every $\gamma$ on $G$ satisfies equation~\eqref{eq:gamma=iota_9}, we can apply Lemma~\ref{lemma:duality} with $C=A$, and this yields equation~\eqref{eq:sigma}. 

Now, by the discussion in Subsections~\ref{subsec:sigma} and~\ref{subsec:sigma_1-sigma_inv}, if $\sigma$ and $1-\sigma$ are not both invertible, then $p\mid \Size{\ker(\gamma)}$ or $p\mid \Size{\ker(\gammatilde)}$, namely $\sigma$ has $0$ or $1$ as an eigenvalue. Otherwise $\sigma$ and $1-\sigma$ are both invertible, and there are actually $\sigma$ with no eigenvalues $0$ and $1$, and this corresponds to the existence of $\gamma$ such that $p \nmid \Size{\ker(\gamma)}, \Size{\ker(\gammatilde)}$.

Except for the case when both $\gamma$ and $\gammatilde$ have kernel of size not divisible by $p$, we will use duality to swich to a more convenient kernel.

\subsection{Outline}
\label{ssec:strategyG789}

We will use Proposition~\ref{prop:lifting} to deal with the kernels of size $q, pq,$ and $p^2$. As for the kernels of size $p$ and $1$, we will appeal to Proposition~\ref{prop:more-lifting}.
To do this, we will show that each $\gamma$ on $G$ with kernel of size $p$ or $1$ always admits at least one invariant Sylow $q$-subgroup $B$.

In order to do that, in Subsubsection~\ref{sssec:order-q-automorphisms} we describe the elements of $\Aut(G)$ of order $q$.

\subsubsection{Description of the elements of order $q$ of $\Aut(G)$}
\label{sssec:order-q-automorphisms}

The Sylow $q$-subgroups of $\Aut(G)$ are of the form $\cC_{q^e}\times\cC_{q^e}$, for $q^e||p-1$, and they
can be described as the Sylow $q$-subgroups of $C_{\Aut(G)}( \Span{\iota(b)} )$, where $\Span{b}$ varies among the Sylow $q$-subgroups of $G$. 
Since they are abelian, each of them contains exactly one subgroup of type $\iota(\Span{b})$, and 
this establishes a one-to-one correspondence between the Sylow $q$-subgroups of $G$ and the Sylow $q$-subgroups of $\Aut(G)$.

We  note also that for $a \in A$ one has
\begin{equation*}
   C_{\Aut(G)}( \Span{\iota(b)} )^{\iota(a)}
 = C_{\Aut(G)}( \Span{\iota(b^{a})} ).
\end{equation*}

For $b \in G \setminus A$, recalling that $\iota(b)$ acts on $A$ as $\diag(\lambda, \lambda^{-1})$, we write
\begin{equation}
  \label{eq:b-and-beta}
  \begin{aligned}
    \beta_{1} \colon &a_{1} \mapsto a_{1}^{\lambda}\\
    &a_{2} \mapsto a_{2}\\
    &b \mapsto b\\
  \end{aligned}
  \qquad
  \begin{aligned}
    \beta_{2} \colon &a_{1} \mapsto a_{1}\\
    &a_{2} \mapsto a_{2}^{\lambda^{-1}}\\
    &b \mapsto b\\
  \end{aligned}
\end{equation}
so that $\iota(b) = \beta_{1} \beta_{2}$, and $\Span{\beta_{1}, \beta_{2}}$ is the $q$-part of a Sylow $q$-subgroup of $C_{\Aut(G)}( \Span{\iota(b)} )$.

Now, if $\beta \in \Aut(G)$ is an element of order $q$, then it belongs to the centraliser of $\Span{\iota(b)}$, where $\Span{b}$ is a Sylow $q$-subgroup of $G$. Therefore, if $\beta_{1}, \beta_{2}$ are as above, then $\beta \in \Span{\beta_{1}, \beta_{2}}$, namely $\beta=\beta_{1}^{x_1}\beta_{2}^{x_2}$, where $0\le x_1, x_2<q$ not both zero.
\vskip 0.3cm

Let us start with the enumeration of the GF's on $G$. We proceed case by case, according to the size of the kernel.

As usual, if $\Size{\ker(\gamma)}=p^2q$, then $\gamma$ corresponds to the right regular representation, so that we will assume $\gamma \neq 1$.

\subsection{The case \texorpdfstring{$\Size{\ker(\gamma)}=q$}{ker = q}}
\label{ssec:G9-kernel q}

Let $B=\ker(\gamma)$. Here $(G,\circ)$ is necessarily of type 5, as it is the only type having a normal subgroup of order $q$. 

By Proposition~\ref{prop:lifting}, since $A$ is characteristic, each GF on $G$ is the lifting of a RGF on $A$, and, conversely, a RGF on $A$ lifts to $G$ if and only if $B$ is invariant under $\{\gamma(a)\iota(a)\mid a\in A\}$.

For each $a\in A$, $\gamma(a)=\iota(a^{-\sigma})$, where $\sigma\in \GL(2,p)$, so that $\gamma(a)\iota(a)=\iota(a^{1-\sigma})$. Taking into account that each Sylow $q$-subgroup of $G$ is self-normalising, we obtain that $\gamma$ lifts to $G$ if and only if $\sigma=1$, namely when
$$\gamma(a)=\iota(a^{-1}). $$
Since this map is a morphism and $[A, \gamma(A)]=\Set{1}$, by Lemma~\ref{Lemma:gamma_morfismi} $\gamma$ is actually a RGF. Therefore, for each of the $p^2$ choices for a Sylow $q$-subgroup, there is a unique RGF on $A$ which lifts to $G$, and we obtain $p^2$ groups.

Note that for all the $\gamma$'s in this case $p\mid\Size{\ker(\gammatilde)}.$

As to the conjugacy classes, if $\gamma$ has kernel $B$, then, for $x\in A$, $\gamma^{\iota(x)}$ has kernel $B^{\iota(x)}$, as for $b\in \ker(\gamma)$,
$$\gamma^{\iota(x)}(b^{\iota(x)})=\iota(x^{-1})\gamma(b)\iota(x)=1 .$$
Since $\iota(A)$ conjugates transitively the $p^2$ Sylow $q$-subgroups of $G$, the orbits contain at least $p^2$ elements. Since there are $p^2$ GF, there is a unique orbit of length $p^2$.

\subsection{The case \texorpdfstring{$\Size{\ker(\gamma)}=pq$}{ker = pq}}
\label{ssec:G9-kernel pq}
Here $K=\ker(\gamma)$ is a subgroup of $G$ isomorphic to $\cC_p\rtimes \cC_q$, therefore we will obtain $(G,\circ)$ of type 6, as it is the only type having a non abelian normal subgroup of order $pq$. 

We can choose $K$ in $2p$ ways, indeed for each of the $p^2$ choices for a Sylow $q$-subgroup $B$, the subgroups of order $p$ that are $B$-invariant are the $1$-dimensional invariant subspaces of the action of $B$. Therefore, there are $2$ of such subgroups. Moreover, since $\cC_p\rtimes \cC_q$ has $p$ subgroups of order $q$, exactly $p$ choices for $B$ give the same group.

Let $K=\Span{a_1, b}$, and let $a_2\in A$ be such that $A=\Span{a_1, a_2}$. The cyclic complement $\Span{a_2}$ of $K$ in $G$ can be chosen in $p$ ways, and since $\gamma(G)\leq \iota(A)$, each of these choices yields a $\gamma(G)$-invariant subgroup. 

Therefore, by Proposition~\ref{prop:lifting}, each $\gamma$ is the lifting of a RGF defined on any of the complements of order $p$. So, we fix $\Span{a_2}$ and we consider the RGF's $\gamma':\Span{a_2} \to \Aut(G)$, taking into account that the choice of the complement is immaterial. Again appealing to Proposition~\ref{prop:lifting}, the RGF's $\gamma'$ which can be lifted to $G$ are those for which $K$ is invariant under $\{\gamma'(x)\iota(x) : x\in \Span{a_2} \}$, namely the maps defined as
$$\gamma'(a_2)=\iota(a_1^{j}a_2^{-1}),$$
for some $j$, $0\le j \le p-1$. Moreover, since $[\Span{a_2}, \gamma(\Span{a_2})]=\Set{1}$, by Lemma~\ref{Lemma:gamma_morfismi} the RGF's correspond to the morphisms. Therefore, since there are $p$ choices for $j$ and $2p$ for $K$, the number of distinct gamma functions is $2p^2$.

Notice that, for every $\gamma$ as above, $p\mid\Size{\ker(\gammatilde)}.$

As to the conjugacy classes, let $\phi\in \Aut(G)$. According to~\cite{classp2q}, $\phi$ has the form $\iota(x)\delta\psi$, where $x\in A$, $\delta_{|B}=1$ and, with respect to the fixed basis, $\delta_{|A}=(\delta_{ij}) \in \GL(2,p)$ is diagonal. $\psi$ is defined as $b^{\psi}=b^{r}$ and $a^{\psi}=a^{S}$, where either $r=1$ and $S=1$, or $r=-1$ and
$$ S= \begin{bmatrix}
0 & 1 \\
1 & 0
\end{bmatrix}.
$$

We have that $\gamma(a_1^{\phi^{-1}})=\gamma(a_1^{\psi\delta^{-1}})$, and
 $\gamma^{\phi}(a_1)=1$ if and only if $a_1^{\phi^{-1}} \in \ker(\gamma)\cap A=\Span{a_1}$, therefore $\psi=1$.
Moreover,
\begin{equation}
\label{eq:conj-class-G9-pq-1}
\gamma^{\phi}(b)=\phi^{-1}\gamma(b^{\iota(x^{-1})})\phi=\phi^{-1}\gamma(x^{1-Z^{-1}})\phi,
\end{equation}
so it is equal to $\gamma(b)=1$ when $x \in \Span{a_1}$. 
Now, writing $a=a_1^{j}a_2^{-1}$, we have
\begin{equation}
\label{eq:conj-class-G9-pq-2}
\gamma^{\phi}(a_2) 
=\phi^{-1}\gamma(a_2^{\delta^{-1}})\phi 
=\phi^{-1}\gamma(a_2^{\delta_{22}^{-1}})\phi 
=\iota(a^{\delta_{22}^{-1}})^{\delta},
\end{equation}
so that $\phi$ stabilises $\gamma$ if and only if $\iota(a^{\delta_{22}^{-1}})^{\delta}=\iota(a)$, and this yields the condition $ j(\delta_{11}-\delta_{22})=0$.

So, if $j=0$ the last condition is always satisfied, and if $j \ne 0$ the $\delta$'s in the stabiliser are the scalar matrices.
Therefore we get one orbit of length $2p$ and one orbit of length $2p(p-1)$.

\subsection{The case \texorpdfstring{$\Size{\ker(\gamma)}=p^2$}{ker = p2}}
\label{ssec:G9-kernel p2}

The action of $\gamma(G)$ of order $q$ on $G$ fixes at least one of the $p^2$ Sylow $q$-subgroups of $G$, say $B=\Span{b}$. 

Now, as $B$ is $\gamma(G)$-invariant, $\gamma(b)_{|B}=1$, and let $\gamma(b)_{|A}=\beta$.
The discussion in Subsubsection~\ref{sssec:order-q-automorphisms} yields that $\beta=\beta_1^{x_1}\beta_2^{x_2}$, where $x_1, x_2$ are not both zero, so that, with respect to the basis $\Span{a_1, a_2}$, we can represent $\beta$ as the matrix
$$T=\begin{bmatrix}
\lambda^{x_1}& 0\\
0 &  \lambda^{-x_2}\\
  \end{bmatrix},
$$
where $\lambda\ne 1$ of order $q$.

Therefore, since we are under the assumptions of Subsection~\ref{ssec:G789-kernel p^2}, we obtain that

\begin{enumerate}
\item there is a unique invariant Sylow $q$-subgroup when both $x_1,x_2 \neq 0$; 
\item there are $p$ invariant Sylow $q$-subgroups when either $x_1=0$ or $x_2=0$.
\end{enumerate}

Moreover, taking $k=-1$ in Subsection~\ref{ssec:G789-kernel p^2}, we find that the action of $b$ on $A$ with respect to the operation $\circ$ has associated matrix
  \begin{equation*}
   Z_{\circ} \sim
  \begin{bmatrix}
  \lambda^{1+x_1} & 0\\
    0 & \lambda^{-1-x_2}\\
  \end{bmatrix}.
\end{equation*}

Therefore we obtain the following groups $(G,\circ)$. 
\begin{description}
\item[Type 5] if $x_1=x_2=-1$, and there are $p^2$ groups.
\item[Type 6] if either $x_1=-1$ and $x_2\ne -1$, or $x_1\ne -1$ and $x_2=-1$. In both the cases there is a unique invariant Sylow $q$-subgroup, except if either $x_2=0$ or $x_1=0$, when there are $p$ invariant Sylow $q$-subgroups. Therefore there are $2p^2(q-2)$ groups when we are in the case (1), plus other $2p$ groups for the case (2).
\item[Type 7] if $x_1+1=-(x_2+1) \ne 0$. There are $p^2(q-3)$ groups for the case (1), plus $2p$ groups for the case (2).
\item[Type 8] if $Z_{\circ} $ is a non scalar matrix with no eigenvalues $1$, and determinant different from $1$.

In case (1) this corresponds to the conditions $x_2 \neq 0,-1$ and the four conditions $x_1\neq 0,-1,-x_2,-x_2-2$, which are independent if and only if in addition $x_2\neq -2$. 
Therefore, for $x_2\ne 0, -1, -2$ we obtain $p^2(q-4)(q-3)$, groups. For $x_2=-2$ the four conditions on $x_1$ reduce to three conditions, and we obtain further $p^2(q-3)$, groups. 

In case (2), suppose $x_1=0$. Then there are three independent conditions on $x_2$. Doubling for the case $x_2= 0$, we obtain $2p(q-3)$.

Summing up, we have just obtained $p^2(q-3)^2+2p(q-3)$ groups of type 8;
looking at the eigenvalues of $Z_\circ$, we easily obtain that they are $2p^2(q-3)+4p$ groups isomorphic to $G_s$, for every $s \in \mathcal{K}$.

\item[Type 9] if $ Z_{\circ} $ is a non-scalar matrix with no eigenvalue $1$ and determinant $1$, namely $x_1\ne -1, -x_2 -2$, $x_2\ne -1$, and $x_1-x_2=0$.
The case (2) can not happen, otherwise $\beta=1$. 
In case (1) we have $x_2\neq 0,-1$, therefore there are $p^2(q-2)$ groups.
\end{description}

As to the conjugacy classes, since the kernel $A$ is characteristic, we have that $\gamma^{\phi}(a)=\gamma(a)$, for every $\phi \in \Aut(G)$.

In the notation of Subsection~\ref{ssec:G9-kernel pq} , write $\phi=\iota(x)\delta\psi$. Since
$b^{\phi^{-1}}=b^{\psi\iota(x^{-1})}\equiv b^{r} \bmod \ker(\gamma)$, we have
\begin{equation*}
\gamma^{\phi}(b) 
= \phi^{-1}\gamma(b^r)\phi 
= \psi \delta^{-1} T^{r} \iota(x^{1-T^{r}}) \delta \psi 
= \psi \delta^{-1} T^{r} \delta \iota(x^{(1-T^{r})\delta}) \psi.
\end{equation*}
Therefore, $\phi$ stabilises $\gamma$ if and only if
\begin{equation*}
\begin{cases}
x^{(1-T^{r})\delta}=1 \\
\delta^{-1} T^{r} \delta = \psi T \psi .
\end{cases}
\end{equation*}
The first condition yields $x=1$ or, if $x=a_1^{u}a_2^{v}$, either $x_1=0$ and $v=0$, or $x_2=0$ and $u=0$. If $\psi=1$ the second condition is always satisfied.
If $\psi \ne 1$, since $\delta^{-1}T^{-1}\delta=T^{-1}$ and $\psi$ acts on $T$ by conjugation exchanging the eigenvalues, the second condition yields $x_1=x_2$.

We obtain the following.
\begin{enumerate}
\item For $(G,\circ)$ of type 5 the stabiliser has order $2(p-1)^2$, so that there is one orbit of length $p^2$.

\item For $(G,\circ)$ of type 6 the stabiliser has order $p(p-1)^2$
when either $x_1=0$ or $x_2=0$, and $(p-1)^2$ when $x_1, x_2 \ne 0$. Therefore, there is one orbit of length $2p$ together with $q-2$ orbits of length $2p^2$.

\item For $(G,\circ)$ of type 7 the stabiliser has order $(p-1)^2$ when $x_1, x_2 \ne 0$, and $p(p-1)^2$ otherwise. Therefore there are $\frac{q-3}{2}$ orbits of length $2p^2$, and one orbit of length $2p$.

\item For $(G,\circ)$ of type 8, if $x_1, x_2 \ne 0$ then the stabiliser has order $(p-1)^2$; otherwise either $x_1=0$ or $x_2=0$, and the stabiliser has order $p(p-1)^2$. Therefore, 
if $(G,\circ)\simeq G_s$, for every $s \in \mathcal{K}$ we obtain $q-3$ orbits of length $2p^2$ and two orbits of length $2p$.

\item For $(G,\circ)$ of type 9, $x_1, x_2 \ne 0$ so that the stabiliser has order $(p-1)^2$, and there are $q-2$ orbits of length $p^2$.
\end{enumerate}

\subsection{The case \texorpdfstring{$\Size{\ker(\gamma)}=p$}{ker = p}}
\label{ssec:G9_kerp}

To count the GF's of this case we will use Proposition~\ref{prop:more-lifting}.

Here $\Size{\gamma(G)}=pq$.
The discussion in Subsubsection~\ref{sssec:order-q-automorphisms} yields that $\gamma(G) = \Span{\iota(a_0), \beta}$, for some $1 \ne a_0 \in A$ with $A^{\sigma} = \Span{\iota(a_0)}$, and $\beta \neq 1$. We can assume $\gamma(b)=\iota(a_0^j)\beta$ for some $j$, where $\beta=\beta_1^{x_1}\beta_2^{x_2}$. With respect to the basis $\Span{a_1, a_2}$, where $\Span{a_1}$ and $\Span{a_2}$ are the eigenspaces of $\iota(b)$, the matrix associated to $\beta$ is $\diag(\lambda^{x_1}, \lambda^{-x_2})$, where $x_1$ and $x_2$ are not both zero.

We write $\gamma(b)=\iota(a_{\ast})\beta$, and with respect to $\Span{a_1, a_2}$, we can represent $\beta_{|A}$ as
$$ T=\begin{bmatrix}
\lambda^{x_1} & 0\\
0 & \lambda^{-x_2}
\end{bmatrix} ,$$
where $x_1, x_2$ are not both zero. 

Following Subsection~\ref{ssec:G789_kerp}, and recalling that for $G$ of type 9 $k=-1$, here we find the following cases:

\begin{itemize}
\item Case A: $\ker(\sigma)=\Span{a_1}$.
 \begin{enumerate}
  \item[(A1)] $\nu=0$, $\mu \neq 0$, $x_1+x_2= -1$;
  \item[(A2)] $\nu=1$, $\mu \neq 0$, $x_1=-x_2 $;
  \item[(A3)] $\nu=1$, $\mu = 0$.
 \end{enumerate}
\item Case B: $\ker(\sigma)=\Span{a_1a_2}$.
  \begin{enumerate}
  \item[(B1) \phantom{}] $\nu=0$, $\mu =-1$, $x_1=-x_2 -2$;
  \end{enumerate}
\end{itemize}
As explained in Subsection~\ref{ssec:G789_kerp}, the results in the cases (A1$^\ast$), (A2$^\ast$), (A3$^\ast$) and (B1$^\ast$) can be obtained doubling the results we will obtain in the cases (A1), (A2), (A3) and (B1).

Notice that $p$ divides both $\Size{\ker(\gamma)}$ and $\Size{\ker(\tilde\gamma)}$ if and only if $\sigma $ has both $0$ and $1$ as eigenvalues, that is, in all the cases above except A1, where, since $\sigma$ has only $0$ as eigenvalue, $p\mid \Size{\ker(\gamma)}$ but $p \nmid\Size{\ker(\tilde\gamma)}$.

\subsubsection{Invariant Sylow $q$-subgroups}
\label{sssec:inv-q-Sylow-G9}

Taking $k=-1$ in Subsubsection~\ref{sssec:inv-q-Sylow-kerp-G789}, we find the following.

\begin{prop}
\label{prop:inv_q-sylow_kerp_G9}
If $G$ is of type 9 and $\gamma$ is a GF on $G$ with $\Size{\ker(\gamma)}=p$, the number of invariant Sylow $q$-subgroups is
\begin{enumerate}
\item[\rm{(A1)}] $1$ when $x_1\neq 0, -1$ and $p$ otherwise.
\item[\rm{(A2)}] $1$ when $x_1 \neq 1$ and $p$ otherwise.
\item[\rm{(A3)}] $1$ when $x_1\neq 0$ and $x_2\neq -1$, $p^2$ when $x_1= 0$ and $x_2= -1$, and $p$ otherwise.
\item[\rm{(B1) \phantom{}}] $1$ when $x_1 \neq  -1, -2$ and $p$ otherwise.
\end{enumerate}
\end{prop}

\subsubsection{Computations}
\label{sssec:computation-G9}

By Subsubsection~\ref{sssec:enumerate-GF-789} the action $Z_\circ$ of $b$ on $A$ in $(G, \circ)$ is given by
\begin{equation*}
Z_\circ = (\sigma(1-Z)+Z)T,
\end{equation*}
and we obtain the following.    
      
\textbf{Case A.} Here $\ker(\sigma)=\Span{a_1}$ and equality~\eqref{eq:Zcirc} yields 
$$Z_\circ=
\begin{bmatrix}
 \lambda^{x_1 + 1}& 0\\
\mu(1-\lambda)\lambda^{x_1} & \lambda^{-x_2}(\nu(1-\lambda^{-1})+\lambda^{-1})\\
  \end{bmatrix}.
$$

\begin{enumerate}
\item[(A1)] We have $p-1$ choices for $\sigma$, and
$$
 Z_\circ \sim \begin{bmatrix}
 \lambda^{x_1 +1}& 0\\
0 & \lambda^{x_1}\\
  \end{bmatrix}.
$$
We obtain the following groups $(G, \circ)$.
\begin{description}
\item[Type 5] does not arise.
\item[Type 6] if $x_1=0$ or $x_1=-1$. If $x_1=0$ for each of the $(p-1)$ choices for $\sigma$ we have $p^2/p$ choices for $B$ giving different GF's, so $p(p-1)$ groups. 
If $x_1=-1$, then there are $p(p-1)$ groups. 
\item[Type 7] does not arise.
\item[Type 8] if $x_1 \ne 0$, ${-1}$, $(q-1)/2$, and these are always three independent conditions. 
Since $x_1 \neq -1$, we get $p^2(p-1)(q-3)$ groups. They split in $2p^2(p-1)$ groups isomorphic to $G_s$ for every $s \in \mathcal{K}$.

\item[Type 9] if $x_1=(q-1)/2$. 
Since $x_1 \neq -1$, we get $p^2(p-1)$ groups.
\end{description}

\item[(A2)] We have $p-1$ choices for $\sigma$, and
$$
 Z_\circ \sim \begin{bmatrix}
 \lambda^{x_1 +1}& 0\\
0 & \lambda^{x_1}\\
  \end{bmatrix}.
$$
We obtain the following groups $(G, \circ)$.
\begin{description}
\item[Type 5] does not arise.
\item[Type 6] when $x_1=-1$. Since $x_1 \neq 1$, there are $p^2(p-1)$ groups. 
\item[Type 7] does not arise.
\item[Type 8] if $x_1\ne0$, ${-1}$, $(q-1)/2$, and these are always three independent conditions. 
Here $x_1$ can be equal to $1$ and so there are $p(p-1)+p^2(p-1)(q-4)$ groups. They split in $p(p-1)+p^2(p-1)$ groups isomorphic to $G_2$ and $2p^2(p-1)$ groups isomorphic to $G_s$ for every $s \neq 2, s \in \mathcal{K}$.
\item[Type 9] if $x_1=(q-1)/2$. Here $x_1=1$ if and only if $q=3$, so that there are $p^2(p-1)$ groups if $q>3$ and $p(p-1)$ if $q=3$.
\end{description}

\item[(A3)]
We have $1$ choice for $\sigma$, and
$$  Z_\circ \sim
\begin{bmatrix}
\lambda^{x_1+1}& 0\\
0 &\lambda^{-x_2}\\
\end{bmatrix}.
$$
We obtain the following groups $(G, \circ)$.
\begin{description}
\item[Type 5] if $1+x_1=-x_2=0$. Since $x_1 \neq 0$ and $x_2 \neq -1$, there are $p^2$ groups.
\item[Type 6] if either $x_1={-1}$ and $x_2\ne 0$ or $x_1 \ne {-1}$ and $x_2 = 0$. In the first case, there are $p$ groups when $x_2=-1$, otherwise, for $x_2 \neq -1$, there are $p^2(q-2)$ groups. 
In the second case, since $x_2=0$, we have to take $x_1\ne 0$ and there are $p^2(q-2)$ groups.
\item[Type 7] when $-x_2 =1+x_1\ne 0$. If $x_1=0$ and $x_2=-1$, then there is one group. The cases $x_1 \neq 0$, $x_2=-1$, and $x_1=0$, $x_2\neq -1$ can not happen, while
if $x_1 \neq 0$ and $x_2 \neq -1$ there are $p^2(q-2)$ groups.
\item[Type 8] when $x_1 \ne -1, -x_2-1, x_2-1$, $x_2\ne 0$.
The case $x_1= 0$ and $x_2=-1$ can not happen. 
If $x_1 \neq 0$ and $x_2=-1$, the four conditions on $x_1$ are actually three conditions, and there are $p(q-3)$ groups. 
If $x_1= 0$ and $x_2\neq -1$ we get further $p(q-3)$ groups.
 
Suppose now $x_1\neq 0$, $x_2 \neq -1$.

There are always four independent conditions on $x_1$ except when $x_2=1$, where the conditions become three. There are $p^2((q-4)(q-3)+(q-3))=p^2(q-3)^2 $ groups. 

Therefore we have just obtained $2p(q-3)+p^2(q-3)^2$ groups, which split in $4p+2p^2(q-3) $ groups isomorphic to $G_s$ for every $s \in \mathcal{K}$.

\item[Type 9] if $x_1 \ne -1$, $x_2\ne 0$ and $1+x_1-x_2 =0$.
For $x_2=-1$ and $x_1=-2 \neq 0$ there are $p$ groups. Similarly, for $x_2 \neq -1$ and $x_1 = 0$ and  there are $p$ groups.
For $x_2 \neq -1$ and $x_1=x_2 -1\neq 0$, namely $x_2 \neq 0,\pm 1$, we get $p^2(q-3)$ groups. 
\end{description}
\end{enumerate}

\textbf{Case B.} Here $\ker(\sigma)=\Span{a_1a_2}$, $x_1=-x_2 -2$. 

\begin{enumerate}
\item[(B1)] We have $p-1$ choices for $\sigma$, and equality~\eqref{eq:Zcirc} yields
$$Z_\circ=
\begin{bmatrix}
 \lambda^{x_1}& 0\\
-\lambda^{x_1}(1-\lambda) & \lambda^{-x_2-1}\\
  \end{bmatrix}
  \sim
\begin{bmatrix}
\lambda^{x_1} & 0 \\
0 & \lambda^{x_1+1}\\
\end{bmatrix}.
$$
We obtain the following groups $(G, \circ)$.
\begin{description}
 \item[Type 5] does not arise.
 \item[Type 6] when $x_1=0$ or $x_1=-1$. In the first case $x_1 \neq -2, -1$ and there are $p^2(p-1)$ groups, while in the second case there are $p(p-1)$ groups. 
 \item[Type 7] does not arise.
\item[Type 8] when $x_1 \neq 0,  -1, (q-1)/2$. If $x_1=-2$ there are $p(p-1)$ groups. Suppose now $x_1\neq -2$; the four conditions on $x_1$ are independent and there are $(p-1)(q-4)p^2$ groups. 
Therefore there are $p^2(p-1)(q-4)+p(p-1)$ groups, which split in $p(p-1)+p^2(p-1)$ groups isomorphic to $G_{2}$, and $2p^2(p-1)$ groups isomorphic to $G_s$ for every $s\ne 2, s \in \mathcal{K}$.
\item[Type 9] if $x_1 = (q-1)/2$ ($x_1 \neq 0$ and $x_1 \neq -1$). There are $(p-1)p^2$ groups. 
 \end{description}
\end{enumerate}

\subsubsection{Conjugacy classes}

As to the conjugacy classes, let $\phi=\iota(x)\delta\psi$, where $x \in A$, $\delta\in \GL(2,p)$ is diagonal, and $\psi \in \cC_2$. Suppose $\psi \ne 1$. Then
$$ \gamma^{\phi}(a_1)=\phi^{-1}\gamma(a_1^{S\delta^{-1}})\phi
=\phi^{-1}\iota(a_2^{-\delta_{22}^{-1}\sigma})\phi
=\phi^{-1}\iota(a_1^{-\mu\delta_{22}^{-1}}a_2^{-\nu\delta_{22}^{-1}})\phi.
$$
In the case (A) we have that $\gamma(a_1)=1$, and since $\iota(a_1^{-\mu\delta_{22}^{-1}}a_2^{-\nu\delta_{22}^{-1}})\ne 1$, then $ \gamma^{\phi}(a_1)\ne \gamma(a_1)$. In the case (B1), $\gamma(a_1)=\iota(a_1^{-\sigma})=\iota(a_1)$, and since here $\nu=0$ and $\mu=-1$,
$$ \gamma^{\phi}(a_1)=\psi^{-1}\delta^{-1} \iota(a_1^{\delta_{22}^{-1}}) \delta \psi 
=\psi^{-1} \iota(a_1^{\delta_{22}^{-1}\delta_{11}})  \psi = \iota(a_2^{\delta_{22}^{-1}\delta_{11}})
\ne \gamma(a_1).
$$

Therefore $\psi=1$, and consider $\phi=\iota(x)\delta$. Now taking $k=-1$ and $\delta$ diagonal in Subsubsection~\ref{sssec:789-conj-classes} we obtain the following.

In the case (A1), equation~\eqref{eq:G-7-ker-p-conj-clas-1} has one solution if $x_1 \ne 0, -1$, and $p$ solutions if $x_1=0, -1$, therefore the orbits have length $2p(p-1)$ if $x_1=0, -1$, and $2p^2(p-1)$ when $x_1 \ne 0, -1$.

In the case (A2), equation~\eqref{eq:G-7-ker-p-conj-clas-1} has one solution if $x_1 \ne 1$, and $p$ solutions if $x_1= 1$, therefore the orbits have length $2p(p-1)$ when $x_1= 1$, and $2p^2(p-1)$ when $x_1 \ne 1$.

In the case (A3), equation~\eqref{eq:G-7-ker-p-conj-clas-1} has one solution if $x_1 \ne 0$ and $x_2 \ne -1 $, $p^2$ solutions if $x_1= 0$ and $x_2=-1$, and $p$ solutions otherwise. 
Therefore the orbits have length $2p^{2}$ if $x_1\ne 0$ and $x_2 \ne -1$, $1$ if $x_1 =0$ and $x_2= -1$, and $2p$ otherwise.

In the case (B1), equation~\eqref{eq:G-7-ker-p-conj-clas-1} has one solution if $x_1 \ne -1, -2$, and $p$ solutions otherwise. 
Therefore the orbits have length $2p^2(p-1)$ if $x_1\ne -1, -2$, and $2p(p-1)$ otherwise.

Therefore here the orbits have length:
\begin{enumerate}
\item in the case (A1), $2p(p-1)$ if $x_1=0, -1$, and $2p^2(p-1)$ otherwise;
\item in the case (A2), $2p(p-1)$ if $x_1=-1$, and $2p^2(p-1)$ otherwise;
\item in the case (A3), $2p^2$ if $x_1\ne 0$ and $x_2 \ne -1$, $2$ if $x_1=0$ and $x_2=-1$,  and $2p$ otherwise;
\item in the case (B1), $2p(p-1)$ if $x_1=-1, -2$, and $2p^2(p-1)$ otherwise.
\end{enumerate}

\begin{recap}
For $G$ of type 9 and $\gamma$ a GF on $G$ with kernel of size $p$, we obtain the following.
\begin{enumerate}
\item the $2p^2$ groups $(G,\circ)$ of type 5 form one class of length $2p^2$.
\item the groups $(G,\circ)$ of type 6 split in this way:
	 \begin{itemize}
	 \item[$-$] $2p$ groups form one class of length $2p$;
	 \item[$-$] $4p^2(q-2)$ groups split in $2(q-2)$ classes of length $2p^2$;
	 \item[$-$] $6p(p-1)$ groups split in $3$ classes of length $2p(p-1)$;
	 \item[$-$] $4p^2(p-1)$ groups split in $2$ classes of length $2p^2(p-1)$;
	 \end{itemize}
In total there are $2(q+1)$ classes.
\item the $2+2p^2(q-2)$ groups $(G,\circ)$ of type 7 split in 
	 \begin{itemize}
	 \item[$-$] $1$ class of length $2$;
	 \item[$-$] $q-2$ classes of length $2p^2$;
	 \end{itemize}
In total there are $q-1$ classes.
\item the groups $(G,\circ)$ of type 8 split in this way:
	 \begin{itemize}
	\item[$-$] $8p$ groups isomorphic to $G_k$, which split in $4$ classes of length $2p$, for every $k \in \mathcal{K}$.
	\item[$-$] $4p^2(q-3)$ groups isomorphic to $G_k$, which split in $2(q-3)$ classes of length $2p^2$, for every $k \in \mathcal{K}$.
	\item[$-$] $4p(p-1)$ groups isomorphic to $G_2$, which split in two classes of length $2p(p-1)$.
	\item[$-$] $12p^2(p-1)$ groups isomorphic to $G_k$, which split in $6$ classes of length $2p^2(p-1)$, for every $k\ne 2$, $k\in \mathcal{K}$, and $8p^2(p-1)$ groups isomorphic to $G_2$ which split in $4$ classes of length $2p^2(p-1)$.
 	 \end{itemize}
In total there are $2(q+2)$ classes for every $G_k$.
 \item the groups $(G,\circ)$ of type 9 split in this way:
  	 \begin{itemize}
 	 \item[$-$] $4p$ groups split in two classes of length $2p$;
	 \item[$-$] $2p^2(q-3)$ groups split in $q-3$ classes of length $2p^2$;
	 \item[$-$] if $q=3$ there are further $2p(p-1)+4p^2(p-1)$ groups, which split in one class of length $2p(p-1)$ and two classes of length $2p^2(p-1)$;
	 \item[$-$] if $q>3$, there are further $6p^2(p-1)$ groups, which split in $3$ classes of length $2p^2(p-1)$.
	 \end{itemize}
In total there are $q+2$ classes.
\end{enumerate}
\end{recap}

\subsection{The case \texorpdfstring{$\Size{\ker(\gamma)}=1$}{ker = 1}}
\label{ssec:G9-ker-1}

The GF's of this case can be divided into subclasses according to the size of  $\ker(\gammatilde)$.
Those for which $\Size{\ker(\gammatilde)}\ne1$ can be recovered via duality from the previous computations applied to $\gammatilde$. For the others, for which $\Size{\ker(\gammatilde)}=1$, we will use Proposition~\ref{prop:more-lifting}.

We recall that $\gammatilde(x)=\gamma(x^{-1})\iota(x^{-1})$ for all $x\in G$, so $\Size{\ker(\gammatilde)}\ne1$ means that there exists $x_0\in G$, $x_0 \ne 1$, such that 
\begin{equation}
\label{eq:gamma=iota_9x0}
\gamma(x_0)=\iota(x_0^{-1})
\end{equation}
whereas the condition $\Size{\ker(\gamma)}=1$ corresponds to 
\begin{equation}
\label{eq:notiota}
\gammatilde(x)\ne\iota(x^{-1}), \text{  for each }x\in G, x\ne 1.
\end{equation}
Clearly, when $\Size{\ker(\gammatilde)}=p^2q$, $\gamma=\tilde{\gammatilde}$ corresponds to the left regular representation, and this gives one group of the same type of $G$.

In the remaining cases for which $q\mid \Size{\ker(\gammatilde)}$, 
the condition~\eqref{eq:notiota} is not fulfilled, so none of the corresponding $\gamma$'s has trivial kernel. 

Consider now the GF's $\gamma$ for which $p\mid \Size{\ker(\gammatilde)}$ (and $q \nmid \Size{\ker(\gammatilde)}$). Here $\gamma(a)=\iota(a^{-\sigma})$, where $\sigma$ has $1$, but not $0$, as eigenvalue (because $p \mid \Size{\ker(\gammatilde)}$ and $\gamma$ is injective).
Therefore, for each $a \in A$, we have that $\gammatilde(a)=\gamma(a^{-1})\iota(a^{-1})=\iota(a^{\sigma-1})$.

Suppose that $\sigma=1$. Then $\Size{\ker(\gammatilde)}=p^2$ and $\gamma(a)=\iota(a^{-1})$. Therefore $p^2 \mid \Size{\gamma(G)}$, and $\ker(\gamma)$ can have size $1$ or $q$. We have $\Size{\ker(\gamma)}=1$ if and only if~\eqref{eq:notiota} is satisfied, and by Subsection~\ref{ssec:G9-kernel p2}, $\gammatilde(b)=\iota(b^{-1})$ if and only if $x_1=x_2= -1$. Therefore, the $p^2$ GF's $\gammatilde$ corresponding to $(G,\circ)$ of type 5 are such that the corresponding $\gamma$ have kernel of size $q$, and all the others $\gammatilde$ correspond to $\gamma$ with kernel of size $1$. 

Suppose now that $\sigma \ne 1$. Since $\sigma$ has $1$ as eigenvalue, then $\Size{\ker(\gammatilde)}=p$, and if $a_0 \in A$ generates $\ker(\gammatilde)$, then $\gamma(a_0)=\iota(a_0^{-1})$, namely $p\mid \Size{\gamma(G)}$. Moreover, since $0$ is not an eigenvalue for $\sigma$, $p \nmid \Size{\ker(\gamma)}$. 
Therefore, again, $\ker(\gamma)$ can have size $1$ or $q$.
By Subsection~\ref{ssec:G9_kerp}, the $\gammatilde$'s such that $p\mid \Size{\ker(\gammatilde)}$ and $p \nmid \Size{\ker(\gamma)}$ are those of the cases (A1) and (A1$^{\ast}$). Moreover, for every $\gammatilde$ belonging to these cases the condition~\eqref{eq:notiota} is satisfied, namely the corresponding $\gamma$ are injective.

We are left with the case when $\Size{\ker(\gamma)}=\Size{\ker(\gammatilde)}  = 1$. In the following we suppose that $\sigma$ has no eigenvalues $0$ or $1$.

Here $\gamma(G) = \Span{\iota(a_1), \iota(a_2), \beta}$, where $\beta \neq 1$. As in Subsubsection~\ref{ssec:q-sylow-G789}, if $b$ is an element of order $q$ fixed by $\beta$, we can assume $\gamma(b)=\iota(a_\ast)\beta$ for some $a_\ast \in A^{\sigma}$, and $\beta=\beta_1^{x_1}\beta_2^{x_2}$.
As usual, denote by $T$ the matrix of $\gamma(b)_{|A}$ with respect to the basis $\Span{a_1, a_2}$. 
The discussion in Subsection~\ref{ssec:strategyG789} yields equation~\eqref{eq:conjugation}, which in our notation here is
\begin{equation}
(\sigma^{-1}-1)^{-1}T(\sigma^{-1}-1)=TZ.
\end{equation}
Now $T$ and $TZ$, being conjugate, have the same eigenvalues, so that $\lambda^{-x_2}=\lambda^{x_1+1}$. 
Therefore $T=\diag(\lambda^{x_1},\lambda^{1+x_1})$, and $\sigma^{-1}-1$ exchanges the two eigenspa\-ces, so
\begin{equation}
\label{eq:sigmamatrix}
\sigma^{-1}-1=
\begin{bmatrix}
0 & s_1\\
    s_2& 0\\
  \end{bmatrix}
\end{equation}
with the conditions $s_1,s_2\ne0$ (due to our assumptions on the eigenvalues of $\sigma$) and $s_1s_2\ne1$. 

\subsubsection{Invariant Sylow $q$-subgroups}
\label{sub:7.3.1}
\label{subsub:invariant_ker1}
We will show that also in this case there always exists at least one invariant Sylow $q$-subgroup.
By the discussion above $\gamma(b)=\iota(a_\ast)\beta_1^{x_1}\beta_2^{-(x_1+1)}$, for some $a_\ast \in A^{\sigma}$ and $0\le x_1 <q$.

By Subsection~\ref{ssec:q-sylow-G789} there exists an invariant Sylow $q$-subgroup, $\Span{b^x}$ where $x \in A$, if and only if the equation~\eqref{eq:inv q-sylow case 1}, namely 
$$ x^{(1-Z^{-1})M}=a_\ast^{(1-Z^{-1})T} $$
where $M=1-(1+T^{-1}\sigma(1-Z^{-1}))T$, has a solution in $x$.

Here
$$
M = 
\begin{bmatrix}
1-\lambda^{x_1}-\frac{(1-\lambda^{-1})}{1-s_1s_2}  & \frac{s_1\lambda(1-\lambda)}{1-s_1s_2}\\
\frac{s_2\lambda^{-1}(1-\lambda^{-1})}{1-s_1s_2}   & 1-\lambda^{x_1+1}-\frac{(1-\lambda)}{1-s_1s_2}
\end{bmatrix},
$$
and since $\det(1-Z^{-1}) \ne 0$ and $\det(M)=(1-\lambda^{x_1})(1-\lambda^{x_1+1})$, we have the following.
\begin{enumerate}
\item If $x_1 \neq 0, -1$, then $M$ has rank $2$ and the system~\eqref{eq:inv q-sylow case 1} admits a unique solution.
\item If $x_1=0$, then, writing $a_{\ast}=a_1^{x}a_2^{y}$, the system~\eqref{eq:inv q-sylow case 1} admits solutions if and only if $y = - s_1 x$. Moreover, in that case there are $p$ solutions. 
If $y \neq - s_1 x$, then there are no GF on $G$ extending the assignment $\gamma(b)=\iota(a_\ast)\beta$, as the condition~\eqref{eq:condizione-a_ast} is not satisfied.

\item If $x_1=-1$, then~\eqref{eq:inv q-sylow case 1} admits $p$ solutions if and only if $a_{\ast}=a_1^{- s_2 y}a_2^{y}$. 
Reasoning as above, if $x \ne -s_2 y$ there are no GF on $G$ extending the assignment $\gamma(b)=\iota(a_\ast)\beta$.
\end{enumerate}

\subsubsection{Computations}
Using Proposition~\ref{prop:more-lifting} we can count the GF's as follows.
\begin{itemize}
\item[$-$] Choose $\sigma\in\GL(2, p)$ without eigenvalues $1$, and a RGF $\gamma: B\to\Aut(G)$ such that $\sigma$ and $\gamma$ satisfy~\eqref{eq:sigma} ($q$ choices for $\gamma$ corresponding to $\gamma(b)=\beta_1^{x_1}\beta_2^{-(x_1+1)}$ and $(p-1)(p-2)$ choices for $\sigma$ as in equation~\eqref{eq:sigmamatrix}).
\item[$-$] By Proposition~\ref{prop:more-lifting} each such assignment defines a unique function $\gamma$, and the GF's obtained in this way are distinct for $x_1 \ne 0,-1$ (namely when $\gamma(G)$ is centerless) and each of them is obtained $p$ times when $x_1=0$ or $-1$  (namely when $Z(\gamma(G))$ is non trivial). 
\end{itemize}

Now, let $\gamma$ be a GF obtained for a choice of $\sigma, B, x_1$. 
Since $\gamma$ is injective, $(G,\circ)$ is isomorphic to $\gamma(G)$. We have
$$
\gamma(b)^{-1}\gamma(a)\gamma(b)=\gamma(b)^{-1}\iota(a^{-\sigma})\gamma(b)=\iota(a^{-\sigma T})=\gamma(a^{\sigma T\sigma^{-1}}),
$$ 
from which we obtain,
\begin{equation*}
  b^{\ominus 1} \circ a \circ b = a^{\sigma T \sigma^{-1}}.
\end{equation*}
Since $a^{\circ k}=a^k$, the action of $\iota(b)$ on $A$ in $(G, \circ)$ is
$$
Z_\circ
\sim \begin{bmatrix}
\lambda^{x_1} & 0\\
    0& \lambda^{x_1+1}\\
  \end{bmatrix}.
$$

We obtain the following groups $(G, \circ)$.
\begin{description}
\item[Type 5] does not arise.
\item[Type 6] when $x_1=0$ or $x_1={-1}$ and there are $2p(p-1)(p-2)$ groups.
\item[Type 7] does not arise.
\item[Type 8] when $x_1 \ne 0,-1,(q-1)/2$ and there are $p^2(p-1)(p-2)(q-3)$ groups. They are $2p^2(p-1)(p-2)$ groups isomorphic to $G_s$ for every $s \in \mathcal{K}$. 
\item[Type 9] when $x_1={(q-1)/2}$ and there are $p^2(p-1)(p-2)$ groups.
\end{description}

\subsubsection{Conjugacy classes}
As to the conjugacy classes, in the notation of Subsection~\ref{ssec:G9-kernel pq} , let $\phi=\iota(x)\delta\psi \in \Aut(G)$. We have
\begin{equation*}
\gamma^{\phi}(a)
= \phi^{-1}\gamma(a^{S\delta^{-1}})\phi 
= \psi^{-1}\delta^{-1}\iota(a^{-S\delta^{-1}\sigma})\delta\psi 
= \iota(a^{-S\delta^{-1}\sigma\delta S}) ,
\end{equation*}
so that $\gamma^{\phi}(a)=\gamma(a)$ if and only if $ \sigma^{-1} \delta S \sigma =  \delta S$. The last condition yields $\delta_{22}=\delta_{11}$ if $S=1$, and $\delta_{22}=\frac{s_2}{s_1}\delta_{11} $ if $S\ne 1$. 

Now,
\begin{equation*}
\gamma^{\phi}(b)
=\phi^{-1}\gamma(b^{\phi^{-1}})\phi
=\phi^{-1}\gamma(x^{1-Z^{-r}}b^{r})\phi ;
\end{equation*}
suppose first $\psi=1$. Using Proposition~\ref{prop:more-lifting}, we obtain
\begin{align*}
\gamma^{\phi}(b) 
&= \phi^{-1}\gamma(x^{1-Z^{-1}}b)\phi \\
&= \phi^{-1}\iota(x^{-(1-Z^{-1})T^{-1}\sigma} ) \beta \phi \\
&= \delta^{-1}\iota(x^{-1+T^{-1}-(1-Z^{-1})T^{-1}\sigma} ) \beta \delta \\
&= \iota(x^{(-1+T^{-1}-(1-Z^{-1})T^{-1}\sigma)\delta} ) \beta ,
\end{align*}
so that $\gamma^{\phi}(b)=\gamma(b)$ if and only if the system
$x^{H_1} =1$, where $H_1:=-1+T^{-1}-(1-Z^{-1})T^{-1}\sigma$, admits a solution. Since $\det(H_1)=(1-\lambda^{-x_1})(1-\lambda^{-x_1-1})$, there is one solution if $x_1 \ne 0, -1$, and $p$ solutions otherwise.

Suppose now that $\psi \ne 1$. Since $\gamma(b^{-1})=\beta^{-1}$, we have
\begin{align*}
\gamma^{\phi}(b) 
&= \phi^{-1}\gamma(x^{1-Z}b^{-1})\phi \\
&= \phi^{-1}\iota(x^{-(1-Z)T\sigma}) \gamma(b^{-1}) \phi \\
&= \phi^{-1}\iota(x^{-(1-Z)T\sigma}) \beta^{-1} \phi \\
&= \psi\delta^{-1}\iota(x^{-1+T-(1-Z)T\sigma}) \beta^{-1} \delta\psi \\
&= \psi\iota(x^{(-1+T-(1-Z)T\sigma)\delta}) \beta^{-1} \psi .
\end{align*}
If $H_2:=-1+T-(1-Z)T\sigma$, we have that $\gamma^{\phi}(b)=\gamma(b)$ if and only if 
$$ \iota(x^{H_2\delta}) \beta^{-1} =  \psi \beta \psi ,$$
namely if and only if
\begin{equation*}
\begin{cases}
x^{H_2}=1 \\
T^{-1}=STS .
\end{cases}
\end{equation*}
Since $\det(H_2)=(1-\lambda^{x_1})(1-\lambda^{x_1+1})$, the system $x^{H_2} =1$ has one solution if $x_1 \ne 0, -1$ and $p$ solutions otherwise, while the condition $T^{-1}=STS$ is satisfied if and only if $x_1=\frac{q-1}{2}$.

We obtain the following.
\begin{enumerate}
\item if $x_1=0, -1$, then the stabiliser has order $p(p-1)$. Here there are $2p(p-1)(p-2)$ groups $(G, \circ)$ of type 6, so that there are $p-2$ orbits of length $2p(p-1)$.
\item if $x_1=\frac{q-1}{2}$, then $(G, \circ)$ is of type 9, and the stabiliser has order $2(p-1)$. Since there are $p^2(p-1)(p-2)$ groups, they split in $p-2$ orbits of length $p^2(p-1)$.
\item if $x_1 \ne 0, -1, \frac{q-1}{2}$, then $(G, \circ)$ is of type 8, and the stabiliser has order $p-1$. 
Since for every $s \in \mathcal{K}$ there are $2p^2(p-1)(p-2)$ groups isomorphic to $G_s$, they split in $p-2$ classes for every $s \in \mathcal{K}$.
\end{enumerate}

\subsection{Results}

\begin{prop}
\label{prop:G9}
Let $G$ be a group of order $p^2q$, $p>2$, of type 9.
Then in $\Hol(G)$ there are:
\begin{enumerate}
\item $4p^2$ groups of type 5, which split in two conjugacy classes of length $p^2$, and one conjugacy class of length $2p^2$;
\item $2p^2(4q+3p-7)$ groups of type 6, which split in 
$4$ conjugacy classes of length $2p$, $4(q-2)$ conjugacy classes of length $2p^2$, 
$p+4$ conjugacy classes of length $2p(p-1)$, and two conjugacy classes of length $2p^2(p-1)$;

in total there are $4q+p+2$ conjugacy classes;
\item $2+4p+2p^2(2q-5)$ groups of type 7, which split in one conjugacy class of length $2$, two conjugacy classes of length $2p$, and $2q-5$ conjugacy classes of length $2p^2$;

in total there are $2(q-1)$ conjugacy classes;
\item 
\begin{itemize}
\item[$-$] $2p(p^3+3p^2-14p+4pq-6)$ groups of type 8 isomorphic to $G_2$, which split in 
$8$ conjugacy classes of length $2p$, 
$4(q-3)$ conjugacy classes of length $2p^2$, 
two conjugacy classes of length $2p(p-1)$, 
and $p+4$ conjugacy classes of length $2p^2(p-1)$;
\item[$-$] for every $s\ne 2$, $s \in \mathcal{K}$, $2p(p^3+5p^2-18p+4pq+8)$ groups of type 8 isomorphic to $G_s$, which split in 
$8$ conjugacy classes of length $2p$, 
$4(q-3)$ conjugacy classes of length $2p^2$, 
and $p+6$ conjugacy classes of length $2p^2(p-1)$;
\end{itemize}

in both the cases in total there are $4q+p+2$ conjugacy classes for every isomorphism class $G_s$;
\item if $q>3$, $2+4p+p^2(p^2+5p+4q-16)$ groups of type 9, which split in 
two conjugacy classes of length $1$,
$2(q-2)$ conjugacy classes of length $p^2$,
two conjugacy classes of length $2p$, 
$q-3$ conjugacy classes of length $2p^2$, 
$p-2$ conjugacy classes of length $p^2(p-1)$, 
and $4$ conjugacy classes of length $2p^2(p-1)$;

in total there are $3q+p-1$ conjugacy classes;
\item if $q=3$, $2+2p+p^3(p+3)$ groups of type 9, which split in 
two conjugacy classes of length $1$,
two conjugacy classes of length $p^2$,
two conjugacy classes of length $2p$, 
$p-2$ conjugacy classes of length $p^2(p-1)$, 
one conjugacy class of length $2p(p-1)$, 
and $3$ conjugacy classes of length $2p^2(p-1)$;

in total there are $8+p$ conjugacy classes.
\end{enumerate}
\end{prop}

\section{Type 8} 
\label{sec: 8}

This case can be handled in a very similar way to the case in which G is type 9, so in the following we will often refer to the previous Section, highlighting only the points that require a different treatment.
\vskip 0.3 cm

Here $q\mid p-1$, $q>3$, and $G$ is isomorphic to one of the groups $(\cC_{p} \times \cC_{p}) \rtimes_{D_0} \cC_{q}$. 
The Sylow $p$-subgroup $A=\Span{a_1, a_2}$ of $G$ is characteristic, and if $a_1, a_2\in A$ are in the eigenspaces of the action of a generator $b$ of a Sylow $q$-subgroup $B$ on $A$, then this action can be represented by a non-scalar diagonal matrix $Z$, with no eigenvalues $1$ and $\det(Z) \ne 1$.

For all the section, we consider $A=\Span{a_1, a_2}$, where $a_1, a_2$ are eigenvectors for $\iota(b)$. With respect to that basis, we have
\begin{equation*}
 Z =
 \begin{bmatrix}
  \lambda & 0\\
  0 & \lambda^{k}\\
 \end{bmatrix},
\end{equation*}
where $\lambda \ne 1$ has order $q$, and $k\neq 0,\pm 1$. 

Recall that the type 8 includes $\frac{q-3}{2}$ different isomorphism classes of groups, and that $\Set{G_k : k \in \mathcal{K}}$ denotes a set of representatives of the isomorphism classes (see Section~\ref{sec:the-groups}).

According to Subsections 4.1 and 4.3 of~\cite{classp2q}, we have 
\begin{equation*}
\Aut(G) = \Hol(\cC_{p}) \times \Hol(\cC_{p});
\end{equation*}
as in the case of the groups of type 9, all the elements of the Sylow $p$-subgroup of $\Aut(G)$ are conjugation by elements of $A$, and for each gamma function $\gamma$ there exists $\sigma\in \End(A)$ such that~\eqref{eq:gamma=iota}, namely
$
\gamma(a)=\iota(a^{-\sigma})
$,
is satisfied for each $a\in A$ (see Section~\ref{sec: 9}).

\subsection{Duality}
\label{ssec:G8_duality}
For $G$ of type 8 the discussion in Subsections~\ref{subsec:sigma} and~\ref{subsec:sigma_1-sigma_inv} yields that $\sigma$ has $0$ or $1$ as an eigenvalue, and this corresponds to have $p \mid \Size{\ker(\gamma)}$ or $p \mid \Size{\ker(\gammatilde)}$.

Therefore we can assume that $p\mid \Size{\ker(\gamma)}$ (equivalently $\sigma$ has $0$ as an eigenvalue), and once we have counted the gamma functions with this property, we will double the number of those for which moreover $p \nmid \Size{\ker(\gammatilde)}$ (we will double only those GF for which $1$ is not an eigenvalue of $\sigma$).

\subsection{Description of the elements of order $q$ of \texorpdfstring{$\Aut(G)$}{Aut(G)}}
\label{ssec:order-q-automorphisms-G8}

The discussion in Subsubsection~\ref{sssec:order-q-automorphisms} here yields that, if $b \in G \setminus A$, recalling that $\iota(b)$ acts on $A$ as $\diag(\lambda, \lambda^{k})$, we can write
\begin{equation}
  \label{eq:b-and-beta-G8}
  \begin{aligned}
    \beta_{1} \colon &a_{1} \mapsto a_{1}^{\lambda}\\
    &a_{2} \mapsto a_{2}\\
    &b \mapsto b\\
  \end{aligned}
  \qquad
  \begin{aligned}
    \beta_{2} \colon &a_{1} \mapsto a_{1}\\
    &a_{2} \mapsto a_{2}^{\lambda^{k}}\\
    &b \mapsto b\\
  \end{aligned}
\end{equation}
so that $\iota(b) = \beta_{1} \beta_{2}$, and if $\beta \in \Aut(G)$ is an element of order $q$, then $\beta \in \Span{\beta_{1}, \beta_{2}}$, namely $\beta=\beta_{1}^{x_1}\beta_{2}^{x_2}$, where $0\le x_1, x_2<q$ are not both zero.

\vskip 0.3 cm

Let us start with the enumeration of the GF's on $G$. As usual, if $\Size{\ker(\gamma)}=p^2q$, then $\gamma$ corresponds to the right regular representation, so that we will assume $\gamma \neq 1$.

Suppose moreover that $G\simeq G_k$, for a certain $k\in \mathcal{K}$.

\subsection{The case \texorpdfstring{$\Size{\ker(\gamma)}=pq$}{ker = pq}}
\label{ssec:G8-kernel pq}
Reasoning as in Subsection~\ref{ssec:G9-kernel pq} we obtain $2p^2$ gamma functions, corresponding to groups $(G,\circ)$ of type 6.

Moreover, for every $\gamma$ here, $p\mid\Size{\ker(\gammatilde)}$.

As to the conjugacy classes, this time an automorphism $\phi$ of $G$ has the form $\phi=\iota(x)\delta$, where $x \in A$ and $\delta$ is such that $\delta_{|B}=1$, $\delta_{|A}=(\delta_{ij})\in \GL(2,p)$ diagonal with respect to the fixed basis.

With the same computations of Subsection~\ref{ssec:G9-kernel pq} we obtain two orbits of length $p$ and two orbits of length $p(p-1)$.

\subsection{The case \texorpdfstring{$\Size{\ker(\gamma)}=p^2$}{ker = p2}}
\label{ssec:G8-kernel p2}

Reasoning as in Subsection~\ref{ssec:G9-kernel p2} we obtain that each $\gamma$ on $G$ is the lifting of at least one RGF defined on an invariant Sylow $q$-subgroup $B$, 
and the RGF's on $B$ are precisely the morphisms. 
We have $\gamma(b)_{|B}=1$, and let $\gamma(b)_{|A}=\beta$; then the discussion in Subsubsection~\ref{ssec:order-q-automorphisms-G8} yields that $\beta=\beta_1^{x_1}\beta_2^{x_2}$, where $x_1, x_2$ not both zero, so that, with respect to the basis $\Span{a_1, a_2}$, we can represent $\beta$ as the matrix
$$T=\begin{bmatrix}
\lambda^{x_1}& 0\\
0 &  \lambda^{kx_2}\\
  \end{bmatrix},
$$
where $\lambda \ne 1$, $x_1, x_2$ are not both zero, and $k \in \mathcal{K}$ is such that $G\simeq G_k$.

To know the exact number of the invariant Sylow $q$-subgroups we appeal to the discussion in Subsubsection~\ref{ssec:q-sylow-G789}; here equation~\eqref{eq:inv q-sylow case 1} yields
$x^{(1-Z^{-1})M}= 1$, where
$$ M = 1-T = \begin{bmatrix}
1-\lambda^{x_1} & 0\\
0 & 1-\lambda^{kx_2}\\
\end{bmatrix} ,$$
and we obtain that
\begin{enumerate}
\item if both $x_1,x_2 \neq 0$, there is a unique invariant Sylow $q$-subgroup; 
\item if either $x_1=0$ or $x_2=0$, there are $p$ invariant Sylow $q$-subgroups.
\end{enumerate}

Denoting as usual by $Z_{\circ}$ the associated matrix of the action of $b$ on $A$ with respect to the operation $\circ$, here we have
  \begin{equation*}
   Z_{\circ} \sim
  \begin{bmatrix}
  \lambda^{1+x_1} & 0\\
    0 & \lambda^{k+kx_2}\\
  \end{bmatrix},
\end{equation*}
and we obtain precisely the same number of groups $(G,\circ)$ of type 5, 6, and 7 as in Subsection~\ref{ssec:G9-kernel p2}. As for the type 8 and 9, we have the following.

\textbf{Type 8.} 
In case (1) the type 8 corresponds to the conditions $x_2 \neq 0,-1$ and $x_1\neq 0,-1,-kx_2-k-1,kx_2+k-1$. The conditions on $x_1$ are independent if and only if in addition $x_2\neq k^{-1}-1, -k^{-1}-1$. When these four conditions on $x_1$ are dependent, they reduce to three. 

We obtain $p^2(q-4)^2$ groups if $x_2\neq -k^{-1}-1$, $k^{-1}-1$, plus
further $p^2(q-3)$ groups if $x_2= -k^{-1}-1, k^{-1}-1$.

In case (2), suppose $x_1=0$. Then, there are four independent conditions on $x_2$. Doubling for the case $x_2=0$, we obtain $2p(q-4)$ groups.

Now, looking at the eigenvalues of $Z_{\circ}$ as $x_1$ and $x_2$ vary, and taking into account the conditions on $x_1$ and $x_2$, one can see that 
the $2p^2(q-3)$ groups split in $4p^2$ groups isomorphic to $G_s$ for every $s\in\mathcal{K}$, and
the $p^2(q-4)^2+2p(q-4)$ groups split in $2p^2(q-5)+4p$ groups isomorphic to $G_s$ for every $s\ne k$, and $2p^2(q-5)+p^2+2p$ groups isomorphic to $G_k$. Therefore in total we have
	\begin{itemize}
	\item[$-$] $2p^2(q-3)+4p$ groups isomorphic to $G_s$ for every $s\neq k$;
	\item[$-$] $p^2(2q-5)+2p$ groups isomorphic to $G_k$. 
	\end{itemize}

\textbf{Type 9.}
$(G, \circ)$ is of type 9 when $x_1\ne -1, kx_2 + k -1$, $x_2\ne -1$, and $x_1+1+kx_2+k=0$.
In case (1) $x_2\neq 0,-1$ and also $x_2\neq -k^{-1}-1$ (otherwise we would have $x_1=0$). Since the last one is a further condition, we obtain $p^2(q-3)$ groups.
The case (2) yields $2p$ groups.

Summing up, there are $2p+p^2(q-3)$ groups.

As to the conjugacy classes, with the same computations as in Subsection~\ref{ssec:G9-kernel p2} (imposing $\psi=1$), we obtain the following.
\begin{enumerate}
\item For $(G,\circ)$ of type 5 there is one orbit of length $p^2$;
\item For $(G,\circ)$ of type 6 we obtain $2$ orbits of length $p$ and $2(q-2)$ orbits of length $p^2$;
\item For $(G,\circ)$ of type 7 there are $q-3$ orbits of length $p^2$ and $2$ orbits of length $p$.
\item For $(G,\circ)$ of type 8, $(G,\circ)\simeq G_s$, then for every $s \neq k$, $s \in \mathcal{K}$, we obtain $2(q-3)$ orbits of length $p^2$ and $4$ orbits of length $p$; otherwise $s=k$, and we get $2q-5$ orbits of length $p^2$ and $2$ orbits of length $p$.
\item For $(G,\circ)$ of type 9 there are $q-3$ orbits of length $p^2$ and $2$ orbits of length $p$.
\end{enumerate}

\subsection{The case \texorpdfstring{$\Size{\ker(\gamma)}=p$}{ker = p}}
\label{ssec:G8_kerp}

As in Subsection~\ref{ssec:G9_kerp}, we have $\gamma(G) = \Span{\iota(a_0), \beta}$, for some $1 \ne a_0 \in A$ with $A^{\sigma} = \Span{\iota(a_0)}$, and $\beta \neq 1$. We can assume $\gamma(b)=\iota(a_0^j)\beta$ for some $j$, where $\beta=\beta_1^{x_1}\beta_2^{x_2}$ (by Subsection~\ref{ssec:order-q-automorphisms-G8}). With respect to the fixed basis we can represent $\beta_{|A}$ as
$$ T=\begin{bmatrix}
\lambda^{x_1} & 0\\
0 & \lambda^{k x_2}
\end{bmatrix} .$$

Also here $\ker(\gamma)=\ker(\sigma)=\Span{v}$ and from equation~\eqref{eq:sigma} we obtain~\eqref{eq:sigmaTZ}, and we distinguish in the three cases $v\in \Span{a_1}$, $v\in \Span{a_2}$, and $v=a_1^{x}a_2^{y}$. 

Following Subsection~\ref{ssec:G789_kerp}, and recalling that $k\neq \pm 1, 0$, here we find the following cases.

\begin{itemize}
\item Case A: $\ker(\sigma)=\Span{a_1}$.
 \begin{enumerate}
  \item[(A1)] $\nu=0$, $\mu \neq 0$, $x_1-x_2 k= k$;
  \item[(A2)] $\nu=1$, $\mu \neq 0$, $x_1=x_2 k$;
  \item[(A3)] $\nu=1$, $\mu = 0$.
 \end{enumerate}
\item Case B: $\ker(\sigma)=\Span{a_1a_2}$.
  \begin{enumerate}
  \item[(B1) \phantom{}] $\nu=0$, $\mu =-1$, $x_1=x_2 k +k-1$;
  \end{enumerate}
\end{itemize}
As explained in Subsection~\ref{ssec:G789_kerp}, the cases (A1$^\ast$), (A2$^\ast$), (A3$^\ast$) and (B1$^\ast$) can be recovered by the cases (A1), (A2), (A3) and (B1) considering $k^{-1}$ instead of $k$.

Notice that $p$ divides both $\Size{\ker(\gamma)}$ and $\Size{\ker(\tilde\gamma)}$ if and only if $\sigma $ has both $0$ and $1$ as eigenvalues, that is, in all the cases above except A1 and A1$^\ast$, where, since $\sigma$ has only $0$ as eigenvalue, $p\mid \Size{\ker(\gamma)}$ but $p \nmid\Size{\ker(\tilde\gamma)}$.

\subsubsection{Invariant Sylow $q$-subgroups}

Taking $k\ne \pm 1, 0$ in Subsubsection~\ref{sssec:inv-q-Sylow-kerp-G789}, we find the following.
\begin{prop}
\label{prop:inv_q-sylow_kerp-G8}
If $G$ is of type 8 and $\Size{\ker(\gamma)}=p$, then the number of invariant Sylow $q$-subgroups is
\begin{enumerate}
\item[\rm{(A1)}] $1$ when $x_1\neq 0, k$ and $p$ otherwise.
\item[\rm{(A2)}] $1$ when $x_1 \neq -k$ and $p$ otherwise.
\item[\rm{(A3)}] $1$ when $x_1\neq 0$ and $x_2\neq -1$, $p^2$ when $x_1= 0$ and $x_2= -1$, and $p$ otherwise.
\item[\rm{(B1) \phantom{}}] $1$ when $x_1 \neq  -1, k-1$ and $p$ otherwise.
\item[\rm{(B1$^\ast$)}] $1$ when $x_2 \neq  -1, -1+k^{-1}$ and $p$ otherwise.
\end{enumerate}
\end{prop}

\subsubsection{Computations}

By Subsubsection~\ref{sssec:enumerate-GF-789} the action $Z_\circ$ of $b$ on $A$ in $(G, \circ)$ is given by
\begin{equation*}
Z_\circ = (\sigma(1-Z)+Z)T,
\end{equation*}
and we obtain the following.

\textbf{Case A.} Here
$$
Z_\circ=
\begin{bmatrix}
 \lambda^{x_1 + 1}& 0\\
\mu(1-\lambda)\lambda^{x_1} & \lambda^{x_2 k}(\nu(1-\lambda^k)+\lambda^{k})\\
  \end{bmatrix}.
$$
\begin{enumerate}
\item[(A1)] We have $p-1$ choices for $\sigma$, and
$$
 Z_\circ \sim \begin{bmatrix}
 \lambda^{x_1 +1}& 0\\
0 & \lambda^{x_1}\\
  \end{bmatrix}.
$$
We obtain the following groups $(G, \circ)$.
\begin{description}
\item[Type 5] does not arise.
\item[Type 6] if $x_1=0$ or $x_1=-1$. If $x_1=0$ we obtain $p(p-1)$ groups; 
if $x_1=-1$ there are $p^2(p-1)$ groups. 
\item[Type 7] does not arise.
\item[Type 8] if $x_1 \ne 0$, ${-1}$, $(q-1)/2$. Suppose first that $k=(q-1)/2$; then $x_1 \neq k$ and there are $p^2(p-1)(q-3)$ groups. Otherwise $k\neq (q-1)/2$ and there are $p(p-1)+p^2(p-1)(q-4)$ groups. Therefore, if $k=\frac{q-1}{2}$, there are $2p^2(p-1)$ groups isomorphic to $G_s$ for every $s \in \mathcal{K}$, and if $k\ne \frac{q-1}{2}$, then there are $p(p-1)+p^2(p-1)$ groups isomorphic to $G_{1+k^{-1}}$ (obtained for $x_1=k, -(k+1)$) and $2p^2(p-1)$ groups isomorphic to $G_s$ for every $s\ne 1+k^{-1}, s \in \mathcal{K}$.

\item[Type 9] if $x_1=(q-1)/2$. When $k=(q-1)/2$ then $x_1=k$ and there are $p(p-1)$ groups. Otherwise $k\neq (q-1)/2$ so that $x_1 \neq k$ and there are $p^2(p-1)$ groups.
\end{description}

\item[(A2)] We have $p-1$ choices for $\sigma$, and
$$
 Z_\circ \sim \begin{bmatrix}
 \lambda^{x_1 +1}& 0\\
0 & \lambda^{x_1}\\
  \end{bmatrix}.
$$
We obtain the following groups $(G, \circ)$.
\begin{description}
\item[Type 5] does not arise.
\item[Type 6] when $x_1=-1$. Here $x_1 \neq -k$ and there are $p^2(p-1)$ groups. 
\item[Type 7] does not arise.
\item[Type 8] if $x_1\ne 0$, ${-1}$, $(q-1)/2$.
	When $k=1/2$ then $x_1 \neq -k$ and there are $(p-1)(q-3)p^2$ groups, which split $2p^2(p-1)$ groups isomorphic to $G_s$ for every $s \in\mathcal{K}$.
	If $k\neq 1/2$ then there are $(p-1)p+p^2(p-1)(q-4)$ groups, which split in $p^2(p-1)+p(p-1)$ groups isomorphic to $G_{1-k^{-1}}$, and $2p^2(p-1)$ groups isomorphic to $G_s$ for every $s \ne 1-k^{-1}, s \in \mathcal{K}$.
\item[Type 9] if $x_1=(q-1)/2$. If $k=1/2$ then $x_1=-k$ and there are $p(p-1)$ groups. Otherwise $k \neq 1/2$, so $x_1 \neq -k$ and there are $p^2(p-1)$ groups.
\end{description}

\item[(A3)]
We have $1$ choice for $\sigma$, and
$$  Z_\circ \sim
\begin{bmatrix}
\lambda^{x_1+1}& 0\\
0 &\lambda^{x_2 k}\\
\end{bmatrix}.
$$
We obtain the following groups $(G, \circ)$.
\begin{description}
\item[Type 5] if $1+x_1=x_2 k=0$. There are $p^2$ groups.
\item[Type 6] if either $x_1={-1}$ and $x_2\ne 0$ or $x_1 \ne {-1}$ and $x_2 = 0$. In the first case, there are $p$ groups when $x_2=-1$, and $p^2(q-2)$ otherwise. 
In the second case, since $x_2=0$, we have to take $x_1\ne 0$ and there there $p^2(q-2)$ groups.
\item[Type 7] when $x_2 k=1+x_1\ne 0$. In both the cases $x_1 \neq 0$, $x_2=-1$, and $x_1=0$, $x_2\neq -1$ there are $2p$ groups. If $x_1 \neq 0$ and $x_2 \neq -1$ there are $p^2(q-3)$ groups. 
\item[Type 8] when $x_1 \ne -1, x_2 k-1, -x_2 k-1$, $x_2\ne 0$.
If $x_1= 0$ and $x_2=-1$, then there is one group. 
If $x_1 \neq 0$ and $x_2=-1$, the four conditions on $x_1$ are independent, and so there are $p(q-4)$ groups. 
If $x_1= 0$ and $x_2\neq -1$ we get further $p(q-4)$ groups. 
Suppose now $x_1\neq 0$, $x_2 \neq -1$; there are always four independent conditions on $x_1$, except when $x_2= \pm k^{-1}$, where the conditions become three. Thus there are $p^2(q-4)^2+2p^2(q-3)$ groups. 

The $2p^2(q-3)$ groups split in $4p^2$ groups isomorphic to $G_s$ for every $s\in\mathcal{K}$, and the $1+2p(q-4)+p^2(q-4)^2$ groups split in $4p+2p^2(q-5)$ groups isomorphic to $G_s$ for every $s\ne -k$, and $1+2p+p^2+2p^2(q-5)$ groups isomorphic to $G_{-k}$. Therefore in total there are $1+ 2p+p^2(2q-5)$ groups isomorphic to $G_{-k}$, and $ 4p+2p^2(q-3)$ groups isomorphic to $G_s$ for every $s \ne -k, s \in \mathcal{K}$.
	
\item[Type 9] if $x_1 \ne -1$, $x_2\ne 0$ and $1+x_1+x_2 k =0$.
For $x_2=-1$ and $x_1=k-1 \neq 0$ there are $p$ groups. Similarly, for $x_2 \neq -1$ and $x_1 = 0$ and  there are $p$ groups.
For $x_2 \neq -1$ and $x_1=-x_2 k -1\neq 0$, namely $x_2 \neq 0,-1, k^{-1}$, we get $p^2(q-3)$ groups. 
\end{description}
\end{enumerate}

\textbf{Case B.} Here $x_1=x_2 k +k-1$.
\begin{enumerate}
\item[(B1)] We have $p-1$ choices for $\sigma$, and
$$
Z_\circ=
\begin{bmatrix}
 \lambda^{x_1}& 0\\
-\lambda^{x_1}(1-\lambda) & \lambda^{x_2 k+k}\\
  \end{bmatrix}
  \sim
\begin{bmatrix}
\lambda^{x_1} & 0 \\
0 & \lambda^{x_1+1}\\
\end{bmatrix}.
$$
We obtain the following groups $(G, \circ)$.
\begin{description}
 \item[Type 5] does not arise.
 \item[Type 6] when $x_1=0$ or $x_1=-1$. In the first case $x_1 \neq k-1, -1$ and there are $p^2(p-1)$ groups, while in the second case there are $p(p-1)$ groups. 
 \item[Type 7] does not arise.
\item[Type 8] when $x_1 \neq 0,  -1, -1/2$. If $x_1=k-1$ (and thus $k \neq 1/2$) there are $p(p-1)$ groups. Suppose now $x_1\neq k-1$; when $k\neq 1/2$ (in particular when $G$ is of type 9) the four conditions on $x_1$ are independent and there are $(p-1)(q-4)p^2$ groups. Otherwise $k=1/2$ and the four conditions are actually three, thus there are $p^2(p-1)(q-3)$ groups.
Therefore there are $p^2(p-1)(q-3)$ groups if $k=1/2$, and they split in $2p^2(p-1)$ groups isomorphic to $G_s$ for every $s \in \mathcal{K}$. If $k\neq 1/2$ there are $p^2(p-1)(q-4)+p(p-1)$ groups, which split in $p(p-1)+p^2(p-1)$ groups isomorphic to $G_{1-k^{-1}}$, 
and $2p^2(p-1)$ groups isomorphic to $G_s$ for every $s\ne 1-k^{-1}, s \in \mathcal{K}$.
\item[Type 9] if $x_1 = (q-1)/2$. We have that $x_1=k-1$ when $k=1/2$ and in this case there are $(p-1)p$ groups. Otherwise $k\neq 1/2$ and there are $(p-1)p^2$ groups. 
 \end{description}
\end{enumerate}

As for the conjugacy classes, here an automorphism of $G$ has the form $\phi=\iota(x)\delta$, where $x\in A$ and $\delta$ is a diagonal matrix. Therefore we can refer to Subsubsection~\ref{sssec:789-conj-classes} for the computation of the conjugacy classes.

Summing up all the results obtained for the kernel of size $p$, we have the following.

\begin{recap}
For $G$ of type 8, $G \simeq G_k$, and $\gamma$ a GF on $G$ with kernel of size $p$, we obtain the following.
\begin{enumerate}
\item $2p^2$ groups $(G,\circ)$ of type 5 which split in two classes of length $p^2$.
\item The groups $(G,\circ)$ of type 6 split in this way:
	 \begin{itemize}
	 \item[$-$] $2p$ groups split in two classes of length $p$;
	 \item[$-$] $4p^2(q-2)$ groups split in $4(q-2)$ classes of length $p^2$;
	 \item[$-$] $4p(p-1)$ groups split in $4$ classes of length $p(p-1)$;
	 \item[$-$] $6p^2(p-1)$ groups split in $6$ classes of length $p^2(p-1)$.
	 \end{itemize}
Therefore in total $4(q+1)$ classes.
\item The $4p+2p^2(q-3)$ groups $(G,\circ)$ of type 7 split in 
	 \begin{itemize}
	 \item[$-$] $4$ classes of length $p$;
	 \item[$-$] $2(q-3)$ classes of length $p^2$.
	 \end{itemize}
In total $2(q-1)$ classes.	 
\item The groups $(G,\circ)$ of type 8 split in this way:
	 \begin{itemize}
 	 \item[$-$] the cases (A3),(A3$^{\ast}$) yield $1+2p+p^2(2q-5)$ groups isomorphic to $G_{-k}$, $4p+2p^2(q-3)$ groups isomorphic to $G_s$ for every $s \ne -k, s \in \mathcal{K}$, $1+2p+p^2(2q-5)$ groups isomorphic to $G_{-k^{-1}}$ and $4p+2p^2(q-3)$ groups isomorphic to $G_s$ for every $s \ne -k^{-1}, s \in \mathcal{K}$.       
\\
      Therefore there are 
      $2+4p+2p^2(2q-5)$ groups isomorphic to $G_{-k}$,
      which split in two classes of length $1$, $4$ classes of length $p$, and $2(2q-5)$ classes of length $p^2$, and 
      $8p+4p^2(q-3)$ groups isomorphic to $G_s$, for every $s \ne -k, s \in \mathcal{K}$, which split in $8$ classes of length $p$, and $4(q-3)$ classes of length $p^2$.
\\      
      In total there are $4(q-1)$ conjugacy classes of groups isomorphic to $G_s$ for every $s\in\mathcal{K}$.
	 \item[$-$] The cases (A1),(A1$^{\ast}$) yield the following. If one among $k, k^{-1}$ is equal to $q-2$, then there are $p(p-1)+ 3p^2(p-1)$ groups isomorphic to $G_2$ and $4p^2(p-1)$ groups isomorphic to $G_s$, for every $s \ne q-2, s \in \mathcal{K}$. Otherwise both $k, k^{-1} \ne q-2$, and there are 
     $2p(p-1)+ 2p^2(p-1)$ groups isomorphic to $G_{1+k}$ if $G_{1+k}\simeq G_{1+k^{-1}}$,
	 $p(p-1)+ 3p^2(p-1)$ groups isomorphic to $G_s$ for $s=1+k, 1+k^{-1}$ if $G_{1+k}\not\simeq G_{1+k^{-1}}$, and $4p^2(p-1)$ groups isomorphic to $G_s$, for every $s \ne 1+k$, $1+k^{-1}$, $s \in \mathcal{K}$. 
\\	 
	 We obtain two classes of length $p(p-1)$ and two classes of length $p^2(p-1)$ in the first case, 
	 one class of length $p(p-1)$ and $3$ classes of length $p^2(p-1)$ in the second case, and $4$ classes of length $p^2(p-1)$ in the last case.
\\
      In total there are $4$ conjugacy classes of groups isomorphic to $G_s$ for every $s\in\mathcal{K}$.	
     \item[$-$] The cases (A2),(A2$^{\ast}$),(B1),(B1$^{\ast}$) yield the following. 
     If one among $k, k^{-1}$ is equal to $2$, then there are $2p(p-1)+6p^2(p-1)$ groups isomorphic to $G_{2}$, and $8p^2(p-1)$ groups isomorphic to $G_s$ for every $s \ne 2, s \in\mathcal{K}$.
     Otherwise both $k, k^{-1} \ne 2$, and there are 
     $4p(p-1)+ 4p^2(p-1)$ groups isomorphic to $G_{1-k}$ if $G_{1-k}\simeq G_{1-k^{-1}}$,
	 $2p(p-1)+ 6p^2(p-1)$ groups isomorphic to $G_s$ for $s=1-k, 1-k^{-1}$ if $G_{1-k}\not\simeq G_{1-k^{-1}}$, and $8p^2(p-1)$ groups isomorphic to $G_s$, for every $s \ne 1-k, 1-k^{-1}, s \in \mathcal{K}$ . 
\\	 
	 We obtain $4$ classes of length $p(p-1)$ and $4$ classes of length $p^2(p-1)$ in the first case, 
	 two classes of length $p(p-1)$ and $6$ classes of length $p^2(p-1)$ in the second case, and $8$ classes of length $p^2(p-1)$ in the last case.
\\
      In total there are $8$ conjugacy classes of groups isomorphic to $G_s$ for every $s\in\mathcal{K}$.
 	\end{itemize}
In total there are $4(q+2)$ classes for every $G_s$. 	 
 \item the groups $(G,\circ)$ of type 9 split in this way:
  	 \begin{itemize}
  	 \item[$-$] the cases (A3),(A3$^{\ast}$) yield $4p+2p^2(q-3)$ groups, which split in $4$ classes of length $p$, and $2(q-3)$ classes of length $p^2$;
  	 \item[$-$] the cases (A1),(A1$^{\ast}$) yield $p(p^2-1)+p^2(p^2-1)$ groups if one among $k, k^{-1}$ is equal to $q-2$ (namely if $G\simeq G_{-2}$) and $2p^2(p-1)$ groups if both $k, k^{-1} \ne q-2$ (so $G\not\simeq G_{-2}$). We obtain one class of length $p(p-1)$ plus one class of length $p^2(p-1)$ in the first case, and two classes of length $p^2(p-1)$ in the second case.
   	 \item[$-$] the cases (A2),(A2$^{\ast}$),(B1),(B1$^{\ast}$) yield 
 $2p(p^2-1)+2p^2(p^2-1)$ groups if $G \simeq G_2$, and $4p^2(p-1)$ groups if $G \not\simeq G_{2}$. We obtain two classes of length $p(p-1)$ plus two classes of length $p^2(p-1)$ in the first case, and $4$ classes of length $p^2(p-1)$ in the second case.
     \end{itemize}
In all the cases in total there are $2(q+2)$ classes.     
\end{enumerate}
\end{recap}

\subsection{Results}

\begin{prop}
\label{prop:G8}
Let $G$ be a group of order $p^2q$, $p>2$, of type 8, so that $G$ is isomorphic to $G_{k}$, where $k\ne 0, 1, -1$ determines the isomorphism class of $G$.
Then in $\Hol(G)$ there are:
\begin{enumerate}
\item $4p^2$ groups of type 5, which split in $4$ conjugacy classes of length $p^2$;
\item $8p^2(q+p-2)$ groups of type 6, which split in $8$ conjugacy classes of length $p$, $8(q-2)$ conjugacy classes of length $p^2$, $8$ conjugacy classes of length $p(p-1)$, and $8$ conjugacy classes of length $p^2(p-1)$; 

in total there are $8(q+1)$ conjugacy classes;
\item $8p+4p^2(q-3)$ groups of type 7, which split in  $8$ conjugacy classes of length $p$, and $4(q-3)$ conjugacy classes of length $p^2$;

in total there are $4(q-1)$ conjugacy classes;
\item if $G$ is not isomorphic to $G_{\pm 2}$, then $q>5$ and there are further 
	\begin{enumerate}
	\item if either $k$ or $k^{-1}$ is a solution of $x^2-x-1=0$,
	\begin{enumerate}
	\item $2(1+5p+4p^2q-17p^2+7p^3)$ groups of type 8 isomorphic to $G_s$, for $s=k, 1-k$, which split in two conjugacy classes of length $1$, $12$ conjugacy classes of length $p$, $2(4q-11)$ conjugacy classes of length $p^2$, two conjugacy classes of length $p(p-1)$, and $14$ conjugacy classes of length $p^2(p-1)$;
	\item $4(3p+2p^2q-8p^2+3p^3)$ groups of type 8 isomorphic to $G_{1+k}$, which split in $16$ conjugacy classes of length $p$, $8(q-3)$ conjugacy classes of length $p^2$, $4$ conjugacy classes of length $p(p-1)$, and $12$ conjugacy classes of length $p^2(p-1)$;	
	\item $8(2p+p^2q-5p^2+2p^3)$ groups of type 8 isomorphic to $G_s$ for every $s \in \mathcal{K}$, $s \ne k, 1+k, 1-k$, which split in $16$ conjugacy classes of length $p$, $8(q-3)$ conjugacy classes of length $p^2$, and $16$ conjugacy classes of length $p^2(p-1)$;
	\end{enumerate} 
in total there are $8(q+1)$ conjugacy classes for every isomorphism class $G_s$;
	\item if $k$ and $k^{-1}$ are the solutions of $x^2+x+1=0$,
	\begin{enumerate}
	\item $2(1+6p+4p^2q-19p^2+8p^3)$ groups of type 8 isomorphic to $G_k$, which split in two conjugacy classes of length $1$, $12$ conjugacy classes of length $p$, $2(4q-11)$ conjugacy classes of length $p^2$, and $16$ conjugacy classes of length $p^2(p-1)$;
	\item $2(1+4p+4p^2q-15p^2+6p^3)$ groups of type 8 isomorphic to $G_{1+k}$, which split in two conjugacy classes of length $1$, $12$ conjugacy classes of length $p$, $2(4q-11)$ conjugacy classes of length $p^2$, $4$ conjugacy classes of length $p(p-1)$, and $16$ conjugacy classes of length $p^2(p-1)$;
	\item $2(7p+4p^2q-18p^2+7p^3)$ groups of type 8 isomorphic to $G_{s}$ for $s=1-k, 1-k^{-1}$, which split in $16$ conjugacy classes of length $p$, $8(q-3)$ conjugacy classes of length $p^2$, two conjugacy classes of length $p(p-1)$, and $14$ conjugacy classes of length $p^2(p-1)$;	
	\item $8(2p+p^2q-5p^2+2p^3)$ groups of type 8 isomorphic to $G_s$ for every $s \in \mathcal{K}$, $s \ne k, 1+k, 1-k, 1-k^{-1}$, which split in $16$ conjugacy classes of length $p$, $8(q-3)$ conjugacy classes of length $p^2$, and $16$ conjugacy classes of length $p^2(p-1)$;
	\end{enumerate} 	
	in total there are $8(q+1)$ conjugacy classes for every isomorphism class $G_s$;
	\item if $k$ and $k^{-1}$ are the solutions of $x^2-x+1=0$, we obtain as many groups as in the previous case, but the isomorphism class depends on $-k$ instead of $k$.
	\item if $k$ and $k^{-1}$ are the solutions of $x^2+1=0$,
	\begin{enumerate}
	\item $4(1+2p+2p^2q-9p^2+4p^3)$ groups of type 8 isomorphic to $G_k$, which split in $4$ conjugacy classes of length $1$, $8$ conjugacy classes of length $p$, $4(2q-5)$ conjugacy classes of length $p^2$, and $16$ conjugacy classes of length $p^2(p-1)$;
	\item $4(3p+2p^2q-8p^2+3p^3)$ groups of type 8 isomorphic to $G_{s}$, for $s=1+k, 1-k$, which split in $16$ conjugacy classes of length $p$, $8(q-3)$ conjugacy classes of length $p^2$, $4$ conjugacy classes of length $p(p-1)$, and $12$ conjugacy classes of length $p^2(p-1)$;	
	\item $8(2p+p^2q-5p^2+2p^3)$ groups of type 8 isomorphic to $G_s$ for every $s \in \mathcal{K}$, $s \ne k, 1+k, 1-k$, which split in $16$ conjugacy classes of length $p$, $8(q-3)$ conjugacy classes of length $p^2$, and $16$ conjugacy classes of length $p^2(p-1)$;
	\end{enumerate} 
in total there are $8(q+1)$ conjugacy classes for every isomorphism class $G_s$;	
	
	\item if $k$ is not a solution of $x^2-x-1=0$, $x^2+x-1=0$, $x^2+x+1=0$, $x^2-x+1=0$, $x^2+1=0$,
	\begin{enumerate}
	\item $2(1+6p+4p^2q-19p^2+8p^3)$ groups of type 8 isomorphic to $G_s$ for $s=k, -k$, which split in two conjugacy classes of length $1$, $12$ conjugacy classes of length $p$, $2(4q-11)$ conjugacy classes of length $p^2$, and $16$ conjugacy classes of length $p^2(p-1)$;
	\item $2(7p+4p^2q-18p^2+7p^3)$ groups of type 8 isomorphic to $G_{s}$ for $s=1+k, 1+k^{-1}, 1-k, 1-k^{-1}$, which split in $16$ conjugacy classes of length $p$, $8(q-3)$ conjugacy classes of length $p^2$, two conjugacy classes of length $p(p-1)$, and $14$ conjugacy classes of length $p^2(p-1)$;	
	\item $8(2p+p^2q-5p^2+2p^3)$ groups of type 8 isomorphic to $G_s$ for every $s \in \mathcal{K}$, $s \ne k, 1+k, 1-k, 1-k^{-1}, 1+k^{-1}, -k$, which split in $16$ conjugacy classes of length $p$, $8(q-3)$ conjugacy classes of length $p^2$, and $16$ conjugacy classes of length $p^2(p-1)$;
	\end{enumerate} 	
in total there are $8(q+1)$ conjugacy classes for every isomorphism class $G_s$;
	\item  $4p(2+p(q+2p-5))$ groups of type 9, which split in $8$ conjugacy classes of length $p$, $4(q-3)$ conjugacy classes of length $p^2$, and $8$ conjugacy classes of length $p^2(p-1)$;

in total there are $4(q+1)$ conjugacy classes;
	\end{enumerate}
\item if $G$ is isomorphic to $G_k$ for $k=\pm 2$ and $q>5$, then there are further
	\begin{enumerate}
	\item $2(1+5p+4p^2q-17p^2+7p^3)$ groups of type 8 isomorphic to $G_2$, which split in two conjugacy classes of length $1$, $12$ conjugacy classes of length $p$, $2(4q-11)$ conjugacy classes of length $p^2$, two conjugacy classes of length $p(p-1)$, and $14$ conjugacy classes of length $p^2(p-1)$;

	in total there are $8(q+1)$ conjugacy classes for every isomorphism class $G_s$;
	\item if $q=7$, 
	\begin{enumerate}
	\item $2(1+4p+13p^2+6p^3)$ groups of type 8 isomorphic to $G_3$, which split in two conjugacy classes of length $1$, $12$ conjugacy classes of length $p$, $2(4q-11)$ conjugacy classes of length $p^2$, $4$ conjugacy classes of length $p(p-1)$, and $12$ conjugacy classes of length $p^2(p-1)$;

in total there are $64$ conjugacy classes for every isomorphism class $G_s$;	
	\end{enumerate}
	\item if $q>7$, 
	\begin{enumerate}
	\item $2(1+6p+4p^2q-19p^2+8p^3)$ groups of type 8 isomorphic to $G_{-2}$, which split in two conjugacy classes of length $1$, $12$ conjugacy classes of length $p$, $2(4q-11)$ conjugacy classes of length $p^2$, and $16$ conjugacy classes of length $p^2(p-1)$;
	\item $2(7p+4p^2q-18p^2+7p^3)$ groups of type 8 isomorphic to $G_s$, for $s=3, \frac{3}{2}$, which split in $16$ conjugacy classes of length $p$, $8(q-3)$ conjugacy classes of length $p^2$, two conjugacy classes of length $p(p-1)$ and $14$ conjugacy classes of length $p^2(p-1)$;
	\item $8(2p+p^2q-5p^2+2p^3)$ groups of type 8 isomorphic to $G_s$, for every $s \in \mathcal{K}, s\ne 3, \frac{q+3}{2}, 2, q-2$, which split in $16$ conjugacy classes of length $p$, $8(q-3)$ conjugacy classes of length $p^2$, and $16$ conjugacy classes of length $p^2(p-1)$;
	\end{enumerate}

in total there are $8(q+1)$ conjugacy classes for every isomorphism class $G_s$;
	\item  $2p(3+p(2q+3p-8))$ groups of type 9, which split in $8$ conjugacy classes of length $p$, $4(q-3)$ conjugacy classes of length $p^2$, two conjugacy classes of length $p(p-1)$, 
	and $6$ conjugacy classes of length $p^2(p-1)$;

in total there are $4(q+1)$ conjugacy classes;
	\end{enumerate}
\item If $q=5$, then $G$ is isomorphic to $G_2$ and there are further
	\begin{enumerate}
	\item $4(1+p+3p^{2}(p+1))$ groups of type 8 isomorphic to $G_2$, which split in $4$ conjugacy classes of length $1$, $8$ conjugacy classes of length $p$, $20$ conjugacy classes of length $p^2$, $4$ conjugacy classes of length $p(p-1)$, and $12$ conjugacy classes of length $p^2(p-1)$;

in total there are $48$ conjugacy classes for every isomorphism class $G_s$;
	\item  $8p(1+p+2p(p^2-1))$ groups of type 9, which split in $8$ conjugacy classes of length $p$, $8$ conjugacy classes of length $p^2$, $4$ conjugacy classes of length $p(p-1)$, and $4$ conjugacy classes of length $p^2(p-1)$.

in total there are $24$ conjugacy classes;
	\end{enumerate}
\end{enumerate}
\end{prop}

\begin{proof}
For the types 5, 6, 7 and 9, the number of $(G,\circ)$ is obtained just summing up the results in 
Subsections~\ref{ssec:G8-kernel pq}, \ref{ssec:G8-kernel p2}, 
and~\ref{ssec:G8_kerp},
and doubling those such that $p\nmid \Size{\ker(\gammatilde)}$ (see the discussion in Subsection~\ref{ssec:G8_duality}), namely when $\Size{\ker(\gamma)}=p^2$ and the cases (A1), (A1$^{\ast}$) when $\Size{\ker(\gamma)}=p$.

If $(G, \circ)$ is of type 8 and $q=5$, then there is only one isomorphism class of groups of type 8, so that also in this case we obtain the number of $(G, \circ)\simeq G_2$ simply summing up the results in the previous sections, and doubling for the cases $\Size{\ker(\gamma)}=p^2q$, $\Size{\ker(\gamma)}=p^2$ and the cases (A1), (A1$^{\ast}$) when $\Size{\ker(\gamma)}=p$.

Suppose now $(G, \circ)$ of type 8 and $q>5$. To obtain the total number of $(G, \circ)$ for every isomorpism class of groups of type 8, we have to distinguish some cases.

Suppose first $k=\pm 2$, then by Subsections~\ref{ssec:G8-kernel p2} and~\ref{ssec:G8_kerp} the number of $(G, \circ)$ of type 8 depends on the isomorphis classes of the groups $G_{2}, G_{3}, G_{\frac{3}{2}}$ and $G_{-2}$. Since two groups of type 8, say $G_{k_1}$ and $G_{k_2}$, are isomorphic if and only if $k_1=k_2$ or $k_1k_2=1$, in this case we have that the $G_{3}\simeq G_{\frac{3}{2}} \simeq G_{-2} \not\simeq G_2$ if $q=7$, and that $G_{2}, G_{3}, G_{\frac{3}{2}}$ and $G_{-2}$ represent different isomorphism classes if $q>7$. 

Suppose now $k\ne\pm 2$. By 
Subsections~\ref{ssec:G8-kernel p2} 
and~\ref{ssec:G8_kerp} the number of $(G, \circ)$ of type 8 depends on the isomorphis classes of the groups $G_{k}$, $G_{1+k^{-1}}$, $G_{1+k}$, $G_{1-k^{-1}}$, $G_{1-k}$ and $G_{-k}$. 

Suppose $5$ is a quadratic residue modulo $q$; then $G_{k}\simeq G_{1+k^{-1}}$, $ G_{1+k} \simeq G_{1-k^{-1}}$ and $G_{1-k} \simeq G_{-k}$ if and only if $k$ is a solution of $x^2-x-1=0$. Moreover $G_{k}\simeq G_{1+k}$, $ G_{1+k^{-1}} \simeq G_{1-k}$ and $G_{1-k^{-1}} \simeq G_{-k}$ if and only if $k$ is a solution of $x^2+x-1=0$. Note also that if $k$ is a solution of $x^2-x-1=0$, then $k^{-1}$ is a solution of $x^2+x-1=0$. Therefore if $G\simeq G_k$ and either $k$ or $k^{-1}$ are solutions of $x^2-x-1=0$ then the groups above are in three different isomorphism classes, namely $G_{k}\simeq G_{1+k^{-1}}$, $ G_{1+k} \simeq G_{1-k^{-1}}$, and $G_{1-k} \simeq G_{-k}$.

Suppose $q-3$ is a quadratic residue modulo $q$; then $ G_{1+k^{-1}} \simeq G_{1+k} \simeq G_{-k}$ if and only if $k$ is a solution of $x^2+x+1=0$. In that case $k^{-1}$ is the other solution, and the groups above are in four different isomorphism classes. Similarly, if $k$ is a solution of $x^2-x+1=0$ there are four different isomorphism classes. Note moreover that if $\alpha_1$, $\alpha_2$ are the solutions of $x^2+x+1=0$, then the solutions of $x^2-x+1=0$ are $-\alpha_1, -\alpha_2$. Therefore the last case can be obtained by the previous one changing $k$ in $-k$.

Lastly suppose $q-4$ is a quadratic residue modulo $q$; then $G_{k}\simeq G_{-k}$, $ G_{1+k} \simeq G_{1-k^{-1}}$ and $G_{1-k} \simeq G_{1+k^{-1}}$ if and only if $k$ is a solution of $x^2+1=0$. Also here the other solution is $k^{-1}$.

Now, note that either $k$ is a solution of exactly one of the above equations, or $k$ is no a solution for any of them. In the last case the groups above form $6$ different isomorphism classes.

In compliance with these facts, summing up the results of the previous subsections 
and doubling for the cases $\Size{\ker(\gamma)}=p^2q$, $\Size{\ker(\gamma)}=p^2$ and the cases (A1), (A1$^{\ast}$) when $\Size{\ker(\gamma)}=p$, 
we obtain \textit{(a)}-\textit{(e)} in \textit{4} and \textit{(a)}-\textit{(c)} in \textit{5}.
\end{proof}

\section{Type 7}

\label{sec: 7}

Here $q\mid p-1$ and $G$ is isomorphic to a group $(\cC_{p} \times \cC_{p}) \rtimes_{S} \cC_{q}$.
The Sylow $p$-subgroup $A=\Span{a_1, a_2}$ of $G$ is characteristic, and if $a_1$, $a_2$ are in the eigenspaces of the action of a generator $b$ of a Sylow $q$-subgroup $B$ on $A$, then this action can be represented by a scalar matrix with no eigenvalues $1$. 
Therefore, if $a_1, a_2$ are eigenvectors for $\iota(b)$, then with respect to that basis, we have
\begin{equation*}
 Z =
 \begin{bmatrix}
  \lambda & 0\\
  0 & \lambda\\
 \end{bmatrix}.
\end{equation*}

The divisibility condition on $p$ and $q$ implies that $(G,\circ)$ can  be of type 5, 6, 7, 8 and 9.

According to Subsections 4.1 and 4.2 of~\cite{classp2q}, we have 
\begin{equation*}
\Aut(G)= \Hol(\cC_{p}\times \cC_{p}).
\end{equation*}

Differently from the types 8 and 9, in this case if $\gamma$ is a GF on a group $G$ of type 7, then $\gamma(A)$ is not necessarily contained in $\Inn(G)$, as here a Sylow $p$-subgroup of $\Aut(G)$ is of the form $\iota(A)\rtimes \mathcal{P}$, where $\mathcal{P}$ is a Sylow $p$-subgroup of $\GL(2,p)$. 

In the following we will distinguish two cases, namely when $\gamma(A)\le \Inn(G)$ and when $\gamma(A) \not\le \Inn(G)$. In the first case, if $\gamma$ is a GF on $G$, then $\gamma_{|A}: A\to\Inn(G)\leq \Aut(G)$ is a RGF, as $A$ is characteristic in $G$. Moreover, Lemma~\ref{Lemma:gamma_morfismi} yields that $\gamma_{|A}$ is a morphism, as $\iota(A)$ acts trivially on the abelian group $A$.
Therefore, for each gamma function $\gamma$ there exists $\sigma\in \End(A)$ satisfying equation \eqref{eq:gamma=iota}, namely
\begin{equation*}
\gamma(a)=\iota(a^{-\sigma})
\end{equation*} 
for each $a\in A.$ 

The case $\gamma(A)\le \Inn(G)$ can be handled in a very similar way to the cases in which $G$
is type 8 or 9, therefore in the following we will often refer to Sections~\ref{sec: 7,8,9}, \ref{sec: 9} and \ref{sec: 8}. The case $\gamma(A) \not\le \Inn(G)$ instead will require a separate treatment.

\subsection{Duality}
\label{ssec: duality-G-7}

Suppose first that $\gamma(A)\leq \Inn(G)$, so that every $\gamma$ on $G$ satisfies equation~\eqref{eq:gamma=iota}. We can apply Lemma~\ref{lemma:duality} with $C=A$, and this yields equation~\eqref{eq:sigma}. 
By the discussion in Subsections~\ref{subsec:sigma} and~\ref{subsec:sigma_1-sigma_inv}, if $\sigma$ and $1-\sigma$ are not both invertible, then $p\mid \Size{\ker(\gamma)}$ or $p\mid \Size{\ker(\gammatilde)}$, namely $\sigma$ has $0$ or $1$ as an eigenvalue. Otherwise $\sigma$ and $1-\sigma$ are both invertible, but this happens only when $q=2$.

Suppose now that $\gamma(A)\not\le \Inn(G)$. We show that, appealing to duality, we can always suppose that $p \mid \Size{\ker(\gamma)}$.

If $\gamma(A)$ has order $p$, then $p\mid \Size{\ker(\gamma)}$. Moreover $\gamma(A)=\Span{\iota(c)\alpha}$, for some $c \in A$ and $\alpha$ in $\GL(2,p)$ of order $p$, therefore, by the discussion in Subsection~\ref{sss:GL2_2}, the kernel is the fixed point space of $\alpha$. 

Now suppose $\Size{\gamma(A)} = p^{2}$. 
We show that there exists a subgroup $C$ of order $p$ which satisfies the hypotheses of Proposition~\ref{prop:duality}, in the slightly more general version described in Remark~\ref{remark:duality-G-7}.
 
Let $\gamma(A) = \Span{\iota(c), \iota(d)\alpha}$, for some $c, d \in A$, and $\alpha \in \GL(2, p)$ of order $p$. Since $1=[\iota(c),  \iota(d)\alpha] = \iota( [c, \alpha ])$, we have that $\alpha$ fixes $c$.
Let $x_1, x_2 \in A$ be such that $\gamma(x_{1}) = \iota(c)$, and $\gamma(x_{2}) = \iota(d)\alpha $. Then
$$x_{1}^{\alpha} x_{2} = x_{1} \circ x_{2} = x_{2} \circ x_{1} = x_{2} x_{1},$$ 
so that $x_1\in\Span{c}$. It follows that $\gamma(c) = \iota(c^{-k})$ for
some $k$.
The subgroup $C=\Span{c}$ is $\gamma(G)$-invariant, as if $b \in G$ has order $q$, then $\gamma(A) \cap \iota(A)$ is normalised by $\Span{\gamma(b)}$, so that $\gamma(b)$ leaves $C$ invariant.
Since $C$ is also normal in $G$, 
Proposition~\ref{prop:duality} and Remark~\ref{remark:duality-G-7}
yield that $\gamma(c) = \iota(c^{-k})$ with $k = 0, 1$, namely either
$C\le \ker(\gamma)$ or $C\le \ker(\gammatilde)$. Now by
\cite[Corollary 2.25]{p2qciclico}
we can assume $C\le \ker(\gamma)$.

Therefore, when $\gamma(A)\leq \Inn(G)$ and $q>2$ we can assume that $p\mid \Size{\ker(\gamma)}$ (equivalentely $\sigma$ has $0$ as an eigenvalue), and once we have counted the gamma functions with this property, we will double the number of those for which moreover $p \nmid \Size{\ker(\gammatilde)}$ (we will double only those GF for which $1$ is not an eigenvalue of $\sigma$). If $\gamma(A)\leq \Inn(G)$ and $q=2$ then there are actually $\sigma$ with no eigenvalues $0$ and $1$, and this corresponds to the existence of $\gamma$ such that $p \nmid \Size{\ker(\gamma)}, \Size{\ker(\gammatilde)}$. Here, except for the case when both $\gamma$ and $\gammatilde$ have kernel of size not divisible by $p$, we will use duality to swich to a more convenient kernel. Otherwise, when $\gamma(A)\not\leq \Inn(G)$, we can assume that $p \mid \Size{\ker(\gamma)}$, and then we will double the numbers we will obtain.

\subsection{Description of the elements of order $q$ of $\Aut(G)$}
\label{ssec:order-q-automorphisms-G7}
An element of order $q$ in $\Aut(G)$ is of the form $\iota(a_\ast)\beta$, where $a_{\ast} \in A$, and $\beta \in \GL(2,p)$ of order $q$.

As we now show, the Sylow $q$-subgroups of $\Aut(G)$ are as many as the Sylow $q$-subgroups of $\GL(2,p)$ (see Subsection~\ref{sss:GL2_3}) multiplied by $p^2$.
In fact, the number of the Sylow $q$-subgroups is equal to the index of $\Norm_{\Aut(G)}(Q)$ in $\Aut(G)$, where $Q$ is any Sylow $q$-subgroup of $\Aut(G)$. Since $\Aut(G)= A\rtimes \GL(2,p)$, necessarily a Sylow $q$-subgroup of $\GL(2,p)$ is a Sylow $q$-subgroup of $\Aut(G)$ as well, therefore we can suppose that $Q$ is contained in $\GL(2,p)$. 

We use the following lemma to show that the normaliser of $Q$ in
$\Aut(G)$ is equal to the normaliser of $Q$ in $\GL(2,p)$, obtaining
the claim above.
\begin{lemma}
  Let $H$ be a permutation group containing a regular subgroup. Let $Q
  \le H$ be such that $Q$ is contained in a unique stabiliser $T_{a}$.
  Then $\Norm_{H}(Q) \le T_{a}$.
\end{lemma}
\begin{proof}
Let $R \leq H$ be a regular subgroup, and let $h \in H$. Then, given $a$ and $a^{h}$, there exists $r \in R$ such that $a^{hr}=a$. Therefore $hr \in T_{a}$ and $h=hr \cdot r^{-1} \in T_{a}R$.
  
Now, let $h=s r \in \Norm_{H}(Q)$, with $s \in T_{a}$, and $r \in R$. Then $Q =
  Q^{s r} = Q^{r} \le T_{a}^{r} = T_{a^{r}}$. It follows that $r = 1$.
\end{proof}
Now we identify $\Aut(G)$ with $H:=\Hol(A)=\rho(A)\Aut(A)$, in particular $\iota(A)$ with $R:=\rho(A)$ and $H$ is acting on $A$. If the Sylow $q$-subgroup $Q$ is such that $Q\subseteq \Stab(a)$ for a certain $a \in A$, then $a$ is fixed by every matrix with eigenvalues of order a power of $q$, and this yields $a=1$.
Therefore $Q$ is contained in a unique stabiliser, $T_{1}\simeq \GL(2,p)$, and by the lemma above $\Norm_{\Aut(G)}(Q)\le \GL(2,p)$, namely $\Norm_{\Aut(G)}(Q) = \Norm_{\GL(2,p)}(Q)$.

\vskip 0.3cm

Let us start with the enumeration of the GF's on $G$. We proceed case by case, according to the size of the kernel.

As usual, if $\Size{\ker(\gamma)}=p^2q$, then $\gamma$ corresponds to the right regular representation, so that we will assume $\gamma \neq 1$.

\subsection{The case \texorpdfstring{$\Size{\ker(\gamma)}=pq$}{ker = pq}}
\label{ssec:G7-kernel-pq} 

Here $K=\ker(\gamma)$ is a subgroup of $G$ isomorphic to $\cC_p\rtimes \cC_q$, therefore we will obtain $(G,\circ)$ of type 6, as it is the only type having a non abelian normal subgroup of order $pq$.

We can choose $K$ in $p(p+1)$ ways, indeed for each of the $p^2$ choices for a Sylow $q$-subgroup $B$, the subgroups of order $p$ that are $B$-invariant are the $1$-dimensional invariant subspaces of the action of $B$. Since the action of a Sylow $q$-subgroup on $A$ here is scalar, all the $p+1$ subgroups of $G$ of order $p$ are $B$-invariant. Moreover, since $\cC_p\rtimes \cC_q$ has $p$ subgroups of order $q$, exactly $p$ choices for $B$ give the same group.

Let $K=\Span{a_1, b}$, and let $a_2\in A$ be such that $A=\Span{a_1, a_2}$.

Suppose first that $\gamma(A)\leq \Inn(G)$. Reasoning as in Subsection~\ref{ssec:G9-kernel pq}, 
here we obtain $p^2(p+1)$ gamma functions, corresponding to groups $(G,\circ)$ of type 6.
Moreover, for every $\gamma$ here, $p\mid\Size{\ker(\gammatilde)}$.

As to the conjugacy classes, this time an automorphism $\phi$ of $G$ has the form $\phi=\iota(x)\delta$, where $x \in A$ and $\delta$ is such that $\delta_{|B}=1$, $\delta_{|A}=(\delta_{ij})\in \GL(2,p)$.
With computations similar to those of Subsection~\ref{ssec:G9-kernel pq} we obtain 
\begin{itemize}
\item[$-$] $\gamma^{\phi}(a_1)=1$ if and only if $a_1^{\delta^{-1}} \in \ker(\gamma)\cap A=\Span{a_1}$, namely $\delta_{12}=0$;
\item[$-$] taking into account \eqref{eq:conj-class-G9-pq-1}, we have that $\gamma^{\phi}(b)=1$ if and only if $x \in \Span{a_1}$; 
\item[$-$] taking into account \eqref{eq:conj-class-G9-pq-2} and writing $a=a_1^{j}a_2^{-1}$, we have that $\gamma^{\phi}(a_2)=\gamma(a_2)$ if and only if $\iota(a^{\delta_{22}^{-1}})^{\delta}=\iota(a)$, namely $ j(\delta_{11}-\delta_{22})=\delta_{21}.$
\end{itemize}
The latter condition yields $\delta_{21}$ as a function of the diagonal elements, so that the stabiliser has order $p(p-1)^2$, and we obtain one orbit of length $p^2(p+1)$.

Suppose now that $\gamma(A)\not\le \Inn(G)$.
In this case there are no $\gamma(G)$-invariant complements of $K$, therefore let us consider $G=KA$. 
Proposition~\ref{prop:lifting} yields that every GF on $G$ is the lifting of a RGF $\gamma':A\to \Aut(G)$ with $\gamma(G)=\gamma'(A)$, and such that $K$ is invariant under $\Set{\gamma'(x)\iota(x) : x\in A}$. 
Conversely, every RGF $\gamma'$ such that $\gamma'(\Span{a_1})=1$, and which makes $K$ invariant under $\Set{\gamma'(x)\iota(x) : x\in A}$, can be lifted to $G$.
Now we show that a such map is a morphism, and it is defined by 
$$\gamma'(a_2)=\alpha\iota(a_1^j a_2^{-1}),$$
for some $0 \leq j \leq p-1$, and $\alpha \in \GL(2,p)$ of order $p$.

Indeed, since $\gamma'(A)=\gamma'(\Span{a_2})$ has order $p$, $\gamma'(a_2)=\alpha \iota(a)$ for some $a\in A$, and $\alpha \in \GL(2,p)$ of order $p$. By Subsubsection~\ref{sss:GL2_2}, $\Span{a_1}$ is the space of the fixed points of $\alpha$, so that $a_1^\alpha=a_1$. Moreover, we can write $a_2^\alpha=a_1^{d}a_2$, for some $1 \leq d \leq p-1$, and by Lemma~\ref{lem:kernel-pq-gamma-morphisms} in the Appendix, the RGF's are morphisms.

Now, if $K$ is invariant under $\gamma'(x)\iota(x)$ for every $x \in A$, then
$\gamma'(a_2)\iota(a_2)=\alpha\iota(aa_2) $
leaves $K$ invariant, so that $a a_2 \in \Span{a_1}$, namely $a=a_1^{j}a_2^{-1}$ for some $j$, $0 \leq j \leq p-1$.
Conversely, choosing $a=a_1^{j}a_2^{-1}$ then $\gamma'(a_2)\iota(a_2)=\alpha\iota(a_1^{j})$, and since $\gamma'$ is a morphism, $K$ is invariant under $\gamma'(x)\iota(x)$ for every $x\in \Span{a_2}$, and so for every $x\in A$. 

Since there are $p(p+1)$ choices for $K$, $p-1$ choices for $\alpha$ and $p$ choices for $\iota(a_1^{j}a_2^{-1})$, we obtain $p^2(p^{2}-1)$ groups.

As to the conjugacy classes, let $\phi=\iota(x)\delta \in \Aut(G)$. As above, the conditions $\gamma^{\phi}(a_1)=\gamma(a_1)$ and $\gamma^{\phi}(b)=\gamma(b)$ yield $\delta_{12}=0$ and $x \in \Span{a_1}$. Moreover here
\begin{align*}
\gamma^{\phi}(a_2)
&=\phi^{-1}\gamma(a_2^{\delta^{-1}})\phi \\
&=\delta^{-1}\alpha^{\delta_{22}^{-1}}\iota({a}^{\alpha^{\delta_{22}^{-1}-1}+\cdots +\alpha +1})\delta \\
&=\delta^{-1}\alpha^{\delta_{22}^{-1}}\delta \iota(a^{\alpha^{\delta_{22}^{-1}-1}+\cdots +\alpha +1})^{\delta},
\end{align*}
so that $\phi$ stabilises $\gamma$ if and only if
both $\delta^{-1}\alpha^{\delta_{22}^{-1}}\delta= \alpha$ (namely $\delta_{22}^2=\delta_{11}$) and 
$$\iota(a^{\alpha^{\delta_{22}^{-1}-1}+\cdots +\alpha +1})^{\delta}=\iota(a),$$ 
namely $\delta_{21}=(j+\frac{d}{2})\delta_{22}(\delta_{22}-1)$.
Therefore the stabiliser has order $p(p-1)$, and there is one orbit of length $p^2(p^2-1)$.

\subsection{The case \texorpdfstring{$\Size{\ker(\gamma)}=p^2$}{ker = p2}}
\label{ssec:G7-kernel p2}

Reasoning as in Subsection~\ref{ssec:G9-kernel p2} we obtain that each $\gamma$ on $G$ is the lifting of at least one RGF defined on an invariant Sylow $q$-subgroup $B$, 
and the RGF's on $B$ are precisely the morphisms. 
We have $\gamma(b)_{|B}=1$, and let $\gamma(b)_{|A}=\beta$; then the discussion in Subsubsection~\ref{ssec:order-q-automorphisms-G7} yields that 
$\beta$ is a matrix of order $q$ in $\GL(2,p)$, and, with respect to a suitable basis of $A$, we have can represent $\beta$ as the diagonal matrix
$$T=\begin{bmatrix}
\lambda^{x_1}& 0\\
0 &  \lambda^{x_2}\\
  \end{bmatrix},
$$
where $\lambda \ne 1$, $x_1, x_2$ are not both zero (see Subsubsection~\ref{sss:GL2_3}).
We assume that $\beta$ is diagonal with respect to $\Set{a_1, a_2}$, taking into account that if $\beta$ is non-scalar, then there are $\frac{1}{2}p(p+1)$ choices for a pair $\Set{A_1, A_2}$ of distinct one-dimensional subspaces of $A$.

To know the exact number of the invariant Sylow $q$-subgroups we appeal to the discussion in Subsubsection~\ref{ssec:q-sylow-G789}; here equation~\eqref{eq:inv q-sylow case 1} yields
$x^{(1-Z^{-1})M}= 1$, where
$$ M = 1-T = \begin{bmatrix}
1-\lambda^{x_1} & 0\\
0 & 1-\lambda^{x_2}\\
\end{bmatrix} ,$$
and we obtain that
\begin{enumerate}
\item if both $x_1,x_2 \neq 0$, there is a unique invariant Sylow $q$-subgroup; 
\item if either $x_1=0$ or $x_2=0$, there are $p$ invariant Sylow $q$-subgroups.
\end{enumerate}

Denoting as usual by $Z_{\circ}$ the associated matrix of the action of $b$ on $A$ with respect to the operation $\circ$, here we have
  \begin{equation*}
   Z_{\circ} \sim
  \begin{bmatrix}
  \lambda^{1+x_1} & 0\\
    0 & \lambda^{1+x_2}\\
  \end{bmatrix}.
\end{equation*}
We obtain the followings groups $(G,\circ)$. 
\begin{description}
\item[Type 5] if $x_1=x_2=-1$, therefore $p^2$ groups.
\item[Type 6] if either $x_1=-1$ and $x_2\ne -1$, or $x_1\ne -1$ and $x_2=-1$. In both cases there is a unique invariant Sylow $q$-subgroup, except if either $x_2=0$ or $x_1=0$, when there are $p$ invariant Sylow $q$-subgroups. Therefore there are $p^3(p+1)(q-2) + p^2 (p+1)$ groups.
\item[Type 7] if $x_1+1=x_2+1 \ne 0$, namely $x_1=x_2 \ne -1$. Since we are in case (1), we obtain $p^2(q-2)$ groups.
\item[Type 8] if $Z_{\circ} $ is a non scalar matrix with no eigenvalues $1$, and determinant different from $1$.

In case (1) this corresponds to the conditions $x_2 \neq 0,-1$ and the four conditions $x_1\neq 0,-1,-x_2-2, x_2$, which are independent if and only if in addition $x_2\neq 0, -2$. 
When these four conditions are dependent, they reduce to three independent condition on $x_1$. 
If $x_2\ne 0, -1, -2$ we have four independent conditions on $x_1$, and therefore we obtain $\frac{1}{2}p^3(p+1)(q-4)(q-3)$ groups. For $x_2=-2$ the conditions are three, and we obtain further $\frac{1}{2}p^3(p+1)(q-3)$ groups. 

In case (2), if $x_1=0$ there are three independent conditions on $x_2$. Doubling for the case $x_2=0$, we obtain $p^2(p+1)(q-3)$ groups.

Summing up, we have just obtained $\frac{1}{2}p^3(p+1)(q-3)^2+p^2(p+1)(q-3)$ groups of type 8;
looking at the eigenvalues of $Z_\circ$, we easily obtain that they are $2p^2(p+1)+p^3(p+1)(q-3)$ groups isomorphic to $G_s$, for every $s \in \mathcal{K}$;

\item[Type 9] if $ Z_{\circ} $ is a non-scalar matrix with no eigenvalue $1$ and determinant $1$, namely $x_1\ne -1, x_2$, $x_2\ne -1$, and $x_1+x_2+2=0$.
In case (1) $x_2\neq 0,-1$ and also $x_2\neq -2$, otherwise we would have $x_1=0$;
since the latter is a new condition there are $\frac{1}{2}p^3(p+1)(q-3)$ groups.
The case (2) yields $p^2(p+1)$.

Summing up, there are $p^2(p+1) + \frac{1}{2} p^3(p+1)(q-3)$ groups.
\end{description}

As to the conjugacy classes, in the notation of Subsection \ref{ssec:G7-kernel-pq}, let $\phi=\iota(x)\delta \in \Aut(G)$. With computations similar to those in Subsection~\ref{ssec:G9-kernel p2} (taking $\psi=1$), here we obtain that $\phi$ stabilises $\gamma$ if and only if
\begin{equation*}
\begin{cases}
x^{(1-T)\delta}=1 \\
\delta^{-1} T \delta = T.
\end{cases}
\end{equation*}

The first condition yields $x=1$ or, if $x=a_1^{u}a_2^{v}$, either $x_1=0$ and $v=0$, or $x_2=0$ and $u=0$. From the second condition we obtain that $\delta$ is any matrix when $T$ is scalar, and $\delta$ is diagonal when $T$ is non-scalar.

We obtain the following.
\begin{enumerate}
\item For $(G,\circ)$ of type 5 , $x_1 = x_2 =-1$, so that the stabiliser has order $\Size{\GL(2,p)}$, and there is one orbit of length $p^2$.
\item For $(G,\circ)$ of type 6, $x_1 \ne x_2$. The stabiliser has order $p(p-1)^2$
when either $x_1=0$ or $x_2=0$, and $(p-1)^2$ when $x_1, x_2\ne 0$. Therefore we obtain one orbit of length $p^2(p+1)$ and $q-2$ orbits of length $p^3(p+1)$.
\item For $(G,\circ)$ of type 7, $x_1=x_2\ne -1$. The stabiliser has order $\Size{\GL(2,p)}$ and there are $q-2$ orbits of length $p^2$.

\item For $(G,\circ)$ of type 8 $x_1 \ne x_2$, so that if $x_1, x_2 \ne 0$ the stabiliser has order $(p-1)^2$; otherwise either $x_1=0$ or $x_2=0$, and the stabiliser has order $p(p-1)^2$. Therefore, 
if $(G,\circ)\simeq G_s$, for every $s\in \mathcal{K}$ we obtain $q-3$ orbits of length $p^3(p+1)$ and two orbits of length $p^2(p+1)$;
\item For $(G,\circ)$ of type 9, $x_1\ne x_2$. When $x_1, x_2 \ne 0$ the stabiliser has order $(p-1)^2$, otherwise the stabiliser has order $p(p-1)^2$. Therefore there are $\frac{q-3}{2}$ orbits of length $p^3(p+1)$ and one of length $p^2(p+1)$;
\end{enumerate}

\subsection{The case \texorpdfstring{$\Size{\ker(\gamma)}=p$}{ker = p} \texorpdfstring{and $\gamma(A)\leq \Inn(G)$}{inn}}
\label{ssec:G7_kerp-inn}

Since here $\Size{\gamma(G)}=pq$ and $\gamma(G)$ intersects $\iota(A)$ non-trivially, we have
$$\gamma(G) = \Span{\iota(c), \iota(d) \beta}$$ 
for some $c,d \in A$, with $A^{\sigma} = \Span{\iota(c)}$, and $\beta \in \GL(2,p)$, $\beta \neq 1$.

Let $b\in G$ (of order $q$) such that $\gamma(b)=\iota(d)\beta$. With respect to a suitable basis of $A$, the matrix associated to $\gamma(b)_{|A}$ is 
$$ T:=\begin{bmatrix}
\lambda^{x_1} & 0\\
0 & \lambda^{x_2}
\end{bmatrix} ,$$
where $x_1, x_2$ are not both zero. 
Denote by $\Set{a_1, a_2}$ such a basis, and keep in mind that when $[\gamma(b)_{|A}]$ is non-scalar there are $\frac{1}{2}p(p+1)$ choices for a pair $\Set{A_1, A_2}$ of distinct one-dimensional subspaces of $A$.

Following Subsection~\ref{ssec:G789_kerp}, and recalling that for $G$ of type 7 $k=1$, here we find the following cases:

\begin{itemize}
\item Case A: $\ker(\sigma)=\Span{a_1}$.
 \begin{enumerate}
  \item[(A1)] $\nu=0$, $\mu \neq 0$, $x_1-x_2= 1$;
  \item[(A2)] $\nu=1$, $\mu \neq 0$, $x_1=x_2$;
  \item[(A3)] $\nu=1$, $\mu = 0$.
 \end{enumerate}
\item Case B: $\ker(\sigma)=\Span{a_1a_2}$.
 \begin{enumerate}
  \item[(B2)] $\mu+1-\nu=0$, $x_1=x_2$.
  \end{enumerate}
\end{itemize}
As explained in Subsection~\ref{ssec:G789_kerp}, the results in the cases (A1$^\ast$), (A2$^\ast$) and (A3$^\ast$) can be obtained doubling the results we will obtain in the cases (A1), (A2) and (A3).

Notice that $p$ divides both $\Size{\ker(\gamma)}$ and $\Size{\ker(\tilde\gamma)}$ if and only if $\sigma $ has both $0$ and $1$ as eigenvalues, that is, in all the cases above except A1, where, since $\sigma$ has only $0$ as eigenvalue, $p\mid \Size{\ker(\gamma)}$ but $p \nmid\Size{\ker(\tilde\gamma)}$.

\subsubsection{Invariant Sylow $q$-subgroups}
\label{sssec:inv-q-Sylow-G7}

Following Subsubsection~\ref{sssec:inv-q-Sylow-kerp-G789} and taking $k=1$, we obtain the following.

\begin{prop}
\label{prop:inv_q-sylow_kerp-G7}
The number of invariant Sylow $q$-subgroups is
\begin{enumerate}
\item[\rm{(A1)}] $1$ when $x_1\neq 0, 1$ and $p$ otherwise.
\item[\rm{(A2)}] $1$ when $x_1 \neq -1$ and $p$ otherwise.
\item[\rm{(A3)}] $1$ when $x_1\neq 0$ and $x_2\neq -1$, $p^2$ when $x_1= 0$ and $x_2= -1$, and $p$ otherwise.
\item[\rm{(B2) \phantom{}}] $1$ when $x_1\ne -1$ and $p$ otherwise.
\end{enumerate}
\end{prop}

\subsubsection{Computations}
\label{sssec:computation-G7}

By Subsubsection~\ref{sssec:enumerate-GF-789} the action $Z_\circ$ of $b$ on $A$ in $(G, \circ)$ is given by
\begin{equation}
Z_\circ = (\sigma(1-Z)+Z)T,
\end{equation}
and we obtain the following.    
      
\textbf{Case A.} Here $\ker(\sigma)=\Span{a_1}$ and equality~\eqref{eq:Zcirc} yields 
$$Z_\circ=
\begin{bmatrix}
 \lambda^{x_1 + 1}& 0\\
\mu(1-\lambda)\lambda^{x_1} & \lambda^{x_2}(\nu(1-\lambda)+\lambda)\\
  \end{bmatrix}.
$$

\begin{enumerate}
\item[(A1)+(A1$^\ast$)] We have $p-1$ choices for $\sigma$, and
$$
 Z_\circ \sim \begin{bmatrix}
 \lambda^{x_1 +1}& 0\\
0 & \lambda^{x_1}\\
  \end{bmatrix}.
$$
We obtain the following groups $(G, \circ)$.
\begin{description}
\item[Type 5] does not arise.
\item[Type 6] if $x_1=0$ or $x_1=-1$. If $x_1=0$ for each of the $(p-1)$ choices for $\sigma$ we have $p^2/p$ choices for $B$ giving different GF's, so $p^2(p^2-1)$ groups. 
If $x_1=-1$, then $x_1 = 1$ if and only if $q=2$. If $q>2$ there are respectively $p^3(p^2-1)$ groups, otherwise if $q=2$ there are $p^2(p^2-1)$ groups. 
\item[Type 7] does not arise.
\item[Type 8] if $x_1 \ne 0$, ${-1}$, $(q-1)/2$, and these are always three independent conditions. 
We have $p^2(p^2-1)$ groups when $x_1=1$ and $p^3(p^2-1)(q-4)$ groups when $x_1\ne 1$. They split in $p^2(p^2-1)+p^3(p^2-1)$ groups isomorphic to $G_2$, and $2p^3(p^2-1)$ groups isomorphic to $G_s$, for every $s \neq 2, s \in \mathcal{K}$.
\item[Type 9] if $x_1=(q-1)/2$. 
Since $(q-1)/2=1$ if and only if $q=3$, we have $p^2(p^2-1)$ groups when $q=3$ and $p^3(p^2-1)$ groups when $q>3$.
\end{description}

\item[(A2)+(A2$^\ast$)] We have $p-1$ choices for $\sigma$, and
$$
 Z_\circ \sim \begin{bmatrix}
 \lambda^{x_1 +1}& 0\\
0 & \lambda^{x_1}\\
  \end{bmatrix}.
$$
We obtain the following groups $(G, \circ)$.
\begin{description}
\item[Type 5] does not arise.
\item[Type 6] when $x_1=-1$. Since $x_1 = -1$ there are $2p(p-1)$ groups. 
\item[Type 7] does not arise.
\item[Type 8] if $x_1\ne 0$, ${-1}$, $(q-1)/2$, and these are always three independent conditions. 
Since $x_1 \neq -1$, there are $2p^2(p-1)(q-3)$ groups. They split in $4p^2(p-1)$ groups isomotphic to $G_s$ for every $s \in \mathcal{K}$. 
\item[Type 9] if $x_1=(q-1)/2$. Since $x_1 \neq -1$, there are $2p^2(p-1)$ groups.
\end{description}

\item[(A3)+(A3$^\ast$)]
We have $1$ choice for $\sigma$, and
$$  Z_\circ \sim
\begin{bmatrix}
\lambda^{x_1+1}& 0\\
0 &\lambda^{x_2}\\
\end{bmatrix}.
$$
We obtain the following groups $(G, \circ)$.
\begin{description}
\item[Type 5] if $1+x_1=x_2 =0$. Since $x_1 \neq 0$ and $x_2 \neq -1$, there are $p^3(p+1)$ groups. 
\item[Type 6] if either $x_1={-1}$ and $x_2\ne 0$ or $x_1 \ne {-1}$ and $x_2 = 0$. In the first case, there are $2p$ groups when $x_2=-1$, otherwise, for $x_2 \neq -1$, there are $p^3(p+1)(q-2)$ groups. 
In the second case, since $x_2=0$, we have to take $x_1\ne 0$ and there there $p^3(p+1)(q-2)$ groups. 
\item[Type 7] when $x_2 =1+x_1\ne 0$. If $x_1=0$ and $x_2=-1$, then $q=2$ and there are $p(p+1)$ groups. In both the cases $x_1 \neq 0$, $x_2=-1$, and $x_1=0$, $x_2\neq -1$ there are $2p^2(p+1)$ groups. If $x_1 \neq 0$ and $x_2 \neq -1$ there are $p^3(p+1)(q-3)$ groups. 
\item[Type 8] when $x_1 \ne -1, x_2 -1, -x_2 -1$, $x_2\ne 0$.
The case $x_1= 0$ and $x_2=-1$ does not arise. 
If $x_1 \neq 0$ and $x_2=-1$, the four conditions on $x_1$ are actually three conditions and there are $p^2(p+1)(q-3)$ groups. 
If $x_1= 0$ and $x_2\neq -1$ we get further $p^2(p+1)(q-3)$ groups.
 
Suppose now $x_1\neq 0$, $x_2 \neq -1$.
There are always four independent conditions on $x_1$ except when $x_2=1$; in the latter case the conditions become three. If $x_2=1$ then there is one invariant Sylow $q$-subgroup and $\beta$ is scalar if and only if $x_1=1=x_2$, so there are $2p^2+ p^3(p+1)(q-4)$ groups. If $x_2 \neq 1$,
there is one invariant Sylow $q$-subgroup and $\beta$ can always be scalar except when $x_2=(q-1)/2$, thus there are $p^3(p+1)(q-4)$ groups when $x_2=(q-1)/2$ and $2p^2 (q-4)+p^3 (p+1)(q-5)(q-4)$ when $x_2\ne (q-1)/2$.

Summing up, we have just obtained
$2p^2(p+1)(q-3)+2p^2(q-3)+p^3(p+1)(q-4)(q-3) $ groups. They split in $4p^2(p+1)+4p^2+2p^3(p+1)(q-4) $ groups isomorphic to $G_s$ for every $s \in \mathcal{K}$.
\item[Type 9] if $x_1 \ne -1$, $x_2\ne 0$ and $1+x_1+x_2 =0$.
Here $x_1=x_2$ if and only if $x_2=(q-1)/2$.
If $x_2=-1$ then $x_1=0$ and there are $p(p+1)$ groups. If $x_2=(q-1)/2=x_1$ there are $2p^2$ groups. Otherwise $x_2\ne 0,-1,(q-1)/2$ and there are $p^3(p+1)(q-3)$ groups.
\end{description}
\end{enumerate}

\textbf{Case B.} Here $\ker(\sigma)=\Span{a_1a_2}$, $x_1=x_2$.
\begin{enumerate}
\item[(B2)] Here $\mu+1=\nu$, and we have $p(p-1)$ choices for 
$$\sigma=
\begin{bmatrix}
 -\mu & -\mu-1\\
  \mu & \mu+1\\
  \end{bmatrix}.$$
Equality~\eqref{eq:Zcirc} yields
$$Z_\circ=
\begin{bmatrix}
 \lambda^{x_1 + 1}-\lambda^{x_1}\mu(1-\lambda)& -\lambda^{x_1} (\mu+1) (1-\lambda)\\
 \lambda^{x_1}\mu(1-\lambda) & \lambda^{x_1 + 1}+\lambda^{x_1}(\mu+1)(1-\lambda)\\
  \end{bmatrix},
$$
so that 
$$ Z_\circ \sim \begin{bmatrix}
\lambda^{x_1+1} & 0 \\
0 & \lambda^{x_1} 
\end{bmatrix} .$$
We obtain the following groups $(G, \circ)$.
\begin{description}
\item[Type 5] does not arise.
\item[Type 6] if $x_1=-1$, and there are $p^2(p-1)$ groups.
\item[Type 7] does not arise.
\item[Type 8] when $x_1 \ne -1, 0, (q-1)/2$. These are three independent conditions and there are $p^3(p-1)(q-3)$ groups.
\item[Type 9] if $x_1\ne -1, 0$, $x_1=(q-1)/2$, and there are $p^3(p-1)$ groups, which split in $2p^3(p-1)$ groups isomorphic to $G_s$ for every $s$. 
\end{description}
\end{enumerate}

As for the conjugacy classes, here an automorphism of $G$ has the form $\phi=\iota(x)\delta$, where $x\in A$ and $\delta \in \GL(2,p)$. Therefore we can refer to Subsubsection~\ref{sssec:789-conj-classes} for the computation of the conjugacy classes.

Summing up all the results obtained for the kernel of size $p$ we have the following.

\begin{recap}
For $G$ of type 7 and $\gamma$ a GF on $G$ with kernel of size $p$ and such that $\gamma(A)\leq \Inn(G)$, we obtain the following.
\begin{enumerate}
\item $p^3(p+1)$ groups $(G,\circ)$ of type 5 form one class of length $p^3(p+1)$.
\item the groups $(G,\circ)$ of type 6 split in this way:
	 \begin{itemize}
	 \item[$-$] $p^2(p^2-1)$ groups form one class of length $p^2(p^2-1)$;
	 \item[$-$] $p^3(p^2-1)$ groups form one class of length $p^3(p^2-1)$;
	 \item[$-$] $p^2(p+1)$ groups form one class of length $p^2(p+1)$;
	 \item[$-$] $2p^3(p+1)(q-2)$ groups split in $2(q-2)$ classes of length $p^3(p+1)$;
	 \end{itemize}
	 In total they split in $2q-1$ classes.
\item the $2p^2(p+1)+p^3(p+1)(q-3)$ groups $(G,\circ)$ of type 7 split in 
	 \begin{itemize}
	 \item[$-$] $2$ classes of length $p^2(p+1)$;
	 \item[$-$] $q-3$ classes of length $p^3(p+1)$;
	 \end{itemize}
	 In total they split in $q-1$ classes.	 
\item the groups $(G,\circ)$ of type 8 split in this way:
	 \begin{itemize}
	 \item[$-$] $p^2(p^2-1)$ groups isomorphic to $G_2$, which form one class of length $p^2(p^2-1)$.
 	 \item[$-$] $p^3(p^2-1)$ groups isomorphic to $G_2$, which form one class of length $p^3(p^2-1)$, and $2p^3(p^2-1)$ groups isomorphic to $G_s$, which split in two classes of length $p^3(p^2-1)$, for every $s\neq 2, s \in \mathcal{K}$.	 
	 \item[$-$] $4p^2(p+1)$ groups isomorphic to $G_s$, which split in $4$ classes of length $p^2(p+1)$, for every $s \in \mathcal{K}$.	
	 \item[$-$] $2p^3(p+1)(q-3)$ groups isomorphic to $G_s$, which split in $2(q-3)$ classes of length $p^3(p+1)$, for every $s\in\mathcal{K}$.
 	 \end{itemize}
In total they split in $2q$ classes for every $G_s$.
 \item the groups $(G,\circ)$ of type 9 split in this way:
	 \begin{itemize}
	 \item[$-$] $p(p+1)$ groups form one class of length $p(p+1)$;
	 \item[$-$] $p^3(p+1)(q-2)$ groups split in $(q-2)$ classes of length $p^3(p+1)$;
	 \item[$-$] if $q=3$, $p^2(p^2-1)$ groups form one class of length $p^2(p^2-1)$;
 	 \item[$-$] if $q>3$, $p^3(p^2-1)$ groups form one class of length $p^3(p^2-1)$;
 	 \end{itemize}
In total they split in $q$ classes.
\end{enumerate}
\end{recap}

\subsection{The case \texorpdfstring{$\Size{\ker(\gamma)}= p$}{ker = p} \texorpdfstring{and $\gamma(A)\not\leq \Inn(G)$}{not inn}}
\label{sub:G7-non-int}

Let $\ker(\gamma)=\Span{a_1}$. 
We claim that $\gamma(G)\cap \iota(A)=\Set{1}$.
Indeed, since $q \mid p-1$, $\gamma(G)$ (of order $pq$) has a unique subgroup of order $p$. 
Moreover $A$ is the Sylow $p$-subgroup of both $G$ and $(G, \circ)$, so that $\gamma(A) \leq \gamma(G)$ and the order of $\gamma(A)$ is a divisor of $p^2$. Therefore the unique subgroup of order $p$ of $\gamma(G)$ is necessarily $\gamma(A)$. Now, either $\Size{\gamma(G)\cap \iota(A)}=1$, and we are done, or $\Size{\gamma(G)\cap \iota(A)}=p$. In the latter case $\gamma(G)\cap \iota(A)=\gamma(A)$, namely $\gamma(A)\leq \Inn(G)$, contradiction.

Now, since $\gamma(G)$ intersects $\iota(A)$ trivially, 
\begin{equation*}
  \gamma(G) = \Span{ \iota(c) \alpha, \beta},
\end{equation*}
where $c \in A$, $\alpha \in \Aut(G)$ has order $p$, and $\beta \in \Aut(G)$ has order $q$.
So $\alpha_{|A}$ is an element of order $p$ in $\GL(2,p)$, and $\beta_{|A}$ is an element of order $q$ in $\GL(2, p)$. By Subsection~\ref{sss:GL2_2} $\alpha_{|A}$ fixes $\Span{a_{1}}$.
Moreover, by Subsection~\ref{sss:GL2_4}, $\Span{a_{1}}$ is an eigenspace for $\beta_{|A}$ too, and if $\Span{a_2}$ is another eigenspace for $\beta_{|A}$ for a suitable choice of $a_2$ we can write 
\begin{equation*}
  \alpha_{|A}
  =
  \begin{bmatrix}
    1 & 0\\
    1 & 1
  \end{bmatrix},
  \quad
  \beta_{|A}
  =
  \begin{bmatrix}
    \lambda^{x_1} & 0\\
    0 & \lambda^{x_2}
  \end{bmatrix},
\end{equation*}
with respect to the basis $a_1, a_2$, where $\lambda \in \mathbb{F}_p^{\ast}$ of order $q$, and $0 \leq x_1, x_2 \leq q-1 $ not both zero.

The subgroup $\Span{\beta}$ of $\Aut(G)$, of order $q$, acts on the set $\mathcal{Q}$ of the Sylow $q$-subgroups of $G$, and since $\Size{\mathcal{Q}}=p^2$, this action has at least one fixed point. Suppose it is $\Span{b}$.
We can suppose that $\Span{b}$ is fixed by $\alpha$ too, in fact otherwise it will be fixed by $\alpha':=\iota(x)\alpha$ for a suitable $x \in A$, and up to an appropriate adjustment of $c$, $\gamma(G)=\Span{\iota(c)\alpha', \beta}$.

Thus we assume in the following that $\alpha_{|A}$ and $\beta_{|A}$ are in the same copy of $\GL(2,p)$.

Note that the Sylow $q$-subgroups $\Span{xb}$ fixed by $\beta$ are those for which $x$ satisfies $x^{1-\beta}=1$. Therefore there is a unique fixed Sylow $q$-subgroup when $x_1,x_2 \ne 0$, and there are $p$ when either $x_1 = 0$ or $x_2=0$. 

\subsubsection{The case $\gamma(G)$ abelian}
Suppose first $\gamma(G)$ is abelian. Then $[\alpha, \beta]=1$ modulo $\iota(A)$, so that $\beta_{|A}$ is a non-trivial scalar matrix, namely $x_1=x_2 \ne 0$. We can assume that $\beta_{|A}$ is scalar multiplication by $\lambda$. Moreover by the discussion above $\Span{b}$ is the unique $\beta$-invariant Sylow $q$-subgroup, as $x_1,x_2\ne 0$.

Since $\gamma(G)$ is abelian, $\beta$ has to centralise $\iota(c)\alpha$, and since
\begin{equation*} 
 b^{\beta\iota(c)\alpha} = c^{(-1+\lambda^{-1})\alpha}b, \qquad 
 b^{\iota(c)\alpha\beta} = c^{(-1+\lambda^{-1})\alpha\beta}b,  
\end{equation*}
and $\beta$ does not have fixed points in $A$, we get $c=1$.

Therefore 
$$ \gamma(G)=\Span{\alpha, \beta} ,$$
where both $\alpha$ and $\beta$ fix $\Span{b}$ pointwise. In particular $\Span{b}$ is $\gamma(G)$-invariant.

We will have $\gamma(a_2)=\alpha^{i}$ for $i\ne 0$. Moreover, since $b$ is fixed by $\gamma(G)$, $\gamma(b)^q = \gamma(b^{q})=1$, namely $\gamma(b)=\beta^{j}$. Now we show that necessarily $j=-1$.

In fact, in this case $(G, \circ)$ can be of type 5 or 6, as they are the only types which have an abelian quotient of order $pq$. We have
$$  
b^{\ominus 1} \circ a \circ b = b^{-1}a^{\gamma(b)} b
= a^{\beta^{j}\iota(b)} = a^{\lambda^{j+1}}.
$$
Denoting by $Z_\circ$ the action of $\iota(b)$ on $A$ in $(G, \circ)$, and taking into account that 
$$a_1^{\circ k}=a_1^k \ \text{  and  } \ a_2^{\circ k}=a_1^{i(\frac{k(k-1)}{2})}a_2^k ,$$ 
we have that $Z_\circ$ has to be scalar multpilication by $1$. In this case $j=-1$, $(G, \circ)$ is of type 5 and $\gamma(b)=\beta^{-1}=\iota(b^{-1})$.

Now we show that such an assignment extends to a gamma function, namely if $a=a_1^{s}a_2^{t}$ the maps defined by
$$ \gamma(ab^{k})=\alpha^{it\lambda^{k}}\beta^{-k}$$
satisfy the GFE. Let $a'=a_1^{x}a_2^{y}$. Then
$$\gamma(ab^{k})\gamma(a'b^{m})= \alpha^{it\lambda^{k}+iy\lambda^{m}}\beta^{-(k+m)}, $$
and
\begin{align*}
\gamma((ab^{k})^{\gamma(a'b^{m})}a'b^{m})
&= \gamma(a^{\alpha^{iy\lambda^{m}}\beta^{-m}} b^{k} a' b^{m}) \\
&= \gamma((a_2^{t})^{\alpha^{iy\lambda^{m}}\beta^{-m}} a_2^{y \lambda^{-k}} b^{k+m}) \\
&= \gamma(a_2^{t \lambda^{-m}+y\lambda^{-k}} b^{k+m}) \\
&= \alpha^{i(t\lambda^{-m}+y\lambda^{-k})\lambda^{k+m}}\beta^{-(k+m)} .
\end{align*}
Since there are $p+1$ choices for $\ker(\gamma)$, $p-1$ choices for $i$, and $p^2$ choices for the Sylow $q$-subgroup fixed by $\beta$, we obtain $p^2(p^2-1)$ groups $(G, \circ)$.

As to the conjugacy classes, $B=\Span{b}$ is the unique $\gamma(B)$-invariant Sylow $q$-subgroup, so that by Lemma~\ref{lemma:conjugacy}-\eqref{item:invariance-under-conjugacy} 
$\bar{B}=B^{\iota(a)}$ is the unique $\gamma^{\iota(a)}(\bar{B})$-invariant Sylow $q$-subgroup. Since $\iota(A)$ conjugates transitively the Sylow $q$-subgroups of $G$, all classes have order a multiple of $p^2$.

Suppose now $\delta \in \GL(2,p)$. Then $\gamma(a_1^{\delta^{-1}})=1$ if and only if $\delta_{12}=0$.
Moreover,
\begin{align*}
\gamma^{\delta}(a_2) &= \delta^{-1}\gamma(a_2^{\delta_{22}^{-1}})\delta 
= \delta^{-1}\alpha^{i\delta_{22}^{-1}} \delta ,
\end{align*}
and it is equal to $\alpha^{i}$ when $\delta_{11}=\delta_{22}^2$. Since every $\delta$ stabilises $\beta$, as $[\delta, \beta]=1$, we have that the stabiliser has order $p(p-1)$. Therefore there is one class of length $p^2(p^2-1)$.

\subsubsection{The case $\gamma(G)$ non-abelian} 
Suppose now $\gamma(G)$ is non-abelian.
Here $\beta$ normalises but does not centralise $\iota(c)\alpha$. Therefore,
\begin{equation*}
(\iota(c) \alpha)^{\beta } =
\iota(c^{\beta})\beta^{-1}\alpha\beta
= \iota(c^{\beta})\alpha^{\lambda^{x_1-x_2}} 
\end{equation*}
is an element of $\Span{\iota(c)\alpha}$, and $x_1 \ne x_2$.
Since $(\iota(c)\alpha)^k=\iota(c^{1+\alpha^{-1}+\cdots +\alpha^{-(k-1)}})\alpha^k$, we have that $(\iota(c)\alpha)^{\beta} \in \Span{\iota(c)\alpha}$ if and only if
\begin{equation*}
c^{\beta} =
c^{1 + \alpha^{-1} + \dots + \alpha^{-(\lambda^{x_1-x_2} -1)}} .
\end{equation*}
Writing $c=a_1^{u}a_2^{v}$, the latter yields
\begin{equation}
\label{eq:system}
\begin{cases}
u(1-\lambda^{-x_2}) = \frac{1}{2}v\lambda^{-x_2}(1-\lambda^{x_1-x_2}) \\
v\lambda^{x_2} = v\lambda^{x_1-x_2} 
\end{cases}.
\end{equation}
From the second equation we obtain either $v = 0$ or $x_1 \equiv 2x_2 \bmod q$.

\textbf{First case.} 
If $v=0$ the first equation yields $u=0$ or $x_2=0$. 

If $x_2=0$ and $u \ne 0$, we have
$$ \gamma(G)= \Span{ \iota(a_1^{u})\alpha, \beta } , $$
where $\alpha_{|\Span{b}}, \beta_{|\Span{b}}=1$. 
In this case there are $p$ Sylow $q$-subgroups fixed by $\beta$, namely $\Span{a_2^t b}$. Among these, those fixed by $\iota(a_1^{u})\alpha$ too are
$$ a_2^t b = (a_2^t b)^{\iota(a_1^{u})\alpha} = (a_2^t)^{\alpha} a_1^{u(-1+\lambda^{-1})\alpha} b = (a_2^t)^{\alpha} a_1^{u(-1+\lambda^{-1})} b,$$
that is, those with $t$ such that $(a_2^t)^{1-\alpha}=a_1^{u(-1+\lambda^{-1})}$, namely $t=u(1-\lambda^{-1})$.

Since $\langle a_2^{u(1-\lambda^{-1})}b \rangle$ is fixed by both $\iota(a_1)\alpha$ and $\beta$, 
 we can suppose that $\gamma(G) = \Span{\alpha, \beta}$ with $\alpha$ and $\beta$ fixing the same Sylow $q$-subgroup (which will still denote by $\Span{b}$).

By the discussion above, when $v = 0$ we can always suppose
$$\gamma(G)=\Span{\alpha, \beta } ,$$
with $\alpha, \beta$ fixing $\Span{b}$. Moreover a Sylow $q$-subgroup $\Span{yb}$ is $\gamma(G)$-invariant if and only if $y^{\alpha}=y^{\beta}=y$. The latter has $y=1$ has unique solution except when $x_1=0$, which gives the $p$ solutions $y \in \Span{a_1}$.

We will have $\gamma(a_2)=\alpha^{i}$, $i\neq 0$, and replacing $\beta$ with another element of order $q$ in $\gamma(G)$ we can suppose $\gamma(b)=\beta$. 

Now we show that such an assignment extends to a gamma function if and only if $x_1=2x_2 +1$. If $a=a_1^{s}a_2^{t}$, consider the maps defined by
$$ \gamma(ab^{k})=\alpha^{it\lambda^{-k x_2}}\beta^{k} .$$
With computations similar to the previous case, for $a'=a_1^{x}a_2^{y}$ we find that
$$\gamma(ab^{k})\gamma(a'b^{m})= \alpha^{it\lambda^{-k x_2}+iy\lambda^{-m x_2 -k(x_1-x_2)}}\beta^{k+m} ,$$
and
$$
\gamma((ab^{k})^{\gamma(a'b^{m})}cb^{m})
= \alpha^{i(t\lambda^{-kx_2}+y\lambda^{-k-(k+m)x_2})}\beta^{k+m} ,
$$
so that they are equal precisely when $x_1=2x_2 +1$.

Since there are $p+1$ choices for the kernel $\Span{a_1}$, $p$ choices for $\Span{a_2}$, $p-1$ for the image of $a_2$, $q-1$ for $x_2$ such that $\beta_{|A}$ is non-scalar, and $p^2$ for the Sylow $q$-subgroup $\Span{b}$ fixed by $\alpha$ and $\beta$, we have $p^3(p^2-1)(q-1)$ GF. They are all distinct if $\Span{b}$ is the unique $\gamma(G)$-invariant Sylow $q$-subgroup. Otherwise, when $x_2=\frac{q-1}{2}$, the $p$ choices for a $\gamma(G)$-invariant Sylow $q$-subgroups yield the same GF, so there are $p^2(p^2-1)$ distinct GF.

Note that we do not consider the choices for the images of $b$, as if $\gamma(b)=\alpha^{i}\beta^{k}$, then the choices for $k$ correspond to the choices for $x_2$, and since $\alpha^{i}\beta^{k}$ will have eigenspaces $\Span{a_1}, \Span{a_3}$, for $a_3 \in A\setminus \Span{a_1}$, the choices for $i$ correspond to the choices for the second eigenspace, namely to the choice for $\Span{a_2}$.

Since 
$$ b^{\ominus 1} \circ a \circ b = a^{\gamma(b)\iota(b)} ,$$
and $a_1^{\circ k}=a_1^k$, $a_2^{\circ k}=a_2^k $ modulo $\Span{a_1}$, denoting by $Z_\circ$ the action of $\iota(b)$ on $A$ in $(G, \circ)$, we have
$$ Z_\circ \sim \begin{bmatrix}
    \lambda^{2x_2+2} & 0\\
    0 & \lambda^{x_2+1}
  \end{bmatrix}. $$
If $q=2$ the unique choice for $x_1, x_2$ compatible with the conditions $\beta$ non-scalar and $x_1=2x_2+1$ is $x_1=1$ and $x_2=0$. Therefore $Z_\circ=\diag(1, \lambda)$, and $(G, \circ)$ is of type 6. Therefore:
\begin{description}
\item[Type 6] is for $q=2$ and $x_2=0$. Since $x_1 \ne 0$ we obtain $p^3(p^2-1)$ groups.
\end{description}

Suppose now $q>2$. In this case if $x_2 = -1$, then $x_1=-1$ too so that $\beta$ is scalar, against the assumption $\gamma(G)$ non-abelian. So suppose $x_2 \ne -1$, so that the groups $(G, \circ)$ can have only types 8 or 9.

\begin{description}
\item[Type 9] is when $\det(Z_\circ)=1$, namely for $3(x_2+1)=0$. If $q>3$ then there are no groups of type 9. If $q=3$ then for $x_2=0$ there is a unique invariant Sylow $q$-subgroup, and we have $p^3(p^2-1)$ groups. If $x_2=1$ there are $p$ invariant Sylow $q$-subgroups, and there are $p^2(p^2-1)$ groups. 
Thus in total $p^2(p^2-1) + p^3(p^2-1)$ groups.
\item[Type 8] is when $q>3$ and $\det(Z_\circ)\ne 1$, namely $x_2 \ne -1$. For each choice of $x_2$ there are $p^3(p^2-1)$ groups, except when $x_2=\frac{q-1}{2}$, in which there are $p^2(p^2-1)$ groups. Thus in total $p^2(p^2-1) + p^3(p^2-1)(q-2)$ groups.

Note that in this case $Z_\circ \sim \diag(\mu^{2},\mu)$, where $\mu=\lambda^{x_2+1}$, therefore the groups $(G, \circ)$ are all isomorphic to $G_2$.
\end{description}

As to the conjugacy classes, when there is a unique Sylow $q$-subgroup $B$ which is $\gamma(B)$-invariant, as in the previous case, all classes have order a multiple of $p^2$. 
Otherwise $x_1=0$ and there are $p$ invariant Sylow $q$-subgroups. In this case $\iota(x)$ stabilises $\gamma$ if and only if $\iota(x)$ commutes with both $\alpha$ and $\beta$, namely when $x \in \Span{a_1}$.

Now suppose $\delta=(\delta_{ij})\in \GL(2,p)$. As above, $\gamma^{\delta}(a_1)=1$ and $\gamma^{\delta}(a_2)=\gamma(a_2)$ if and only if $\delta_{12}=0$ and $\delta_{11}=\delta_{22}^2$.
Moreover 
$$ \gamma^{\delta}(b) = \delta^{-1}\gamma(b)\delta
= \delta^{-1} \beta \delta,
$$
and $\beta$, which is non-scalar, is centralised by $\delta$ when $\delta$ is a diagonal matrix. 

Therefore the orbits have length $p^3(p^2-1)$ if $x_1 \ne 0$ and $p^2(p^2-1)$ if $x_1=0$.

\textbf{Second case.} 
Suppose now $v \ne 0$ and $x_1 \equiv 2x_2 \bmod q$. Then~\eqref{eq:system} yields $v=-2u$ and we have
$$\gamma(G)=\Span{\iota(c)\alpha, \beta} ,$$
where $\alpha$ and $\beta$ fix $\Span{b}$, and $c=a_1^{u}a_2^{-2u}$.

replacing $a_1$ with a suitable element in $\Span{a_1}$ we can suppose $u=\frac{1}{2}$.

We will have $\gamma(a_2)=(\iota(c)\alpha)^{i}$ for some $i\ne 0$, and $\gamma(b)=(\iota(c)\alpha)^{j}\beta^{k}$, with $k \ne 0$, as it is an element of order $q$ in $\gamma(G)$.
replacing $\beta$ with a suitable element in $\Span{\beta}$ we can suppose that $k=1$. 

Now we show that if the assignment above extends to a GF, then $i=1$.
In fact, denoting by $M_i$ the matrix $1+\alpha^{-1}+\cdots+\alpha^{-(i-1)}$, we will have
\begin{align*}
 b^{\ominus 1} \circ a_{2} \circ b  
 &= (b^{\gamma(b)^{-1}\gamma(a_2)\gamma(b)})^{-1}a_2^{\gamma(b)}b \\
 &= (b^{\iota(c^{M_{i}})\alpha^{i}\beta})^{-1}a_2^{\alpha^{j}\beta}b \\
 &= b^{-1} c^{(1-\lambda^{-1})M_i\alpha^i\beta} a_2^{\alpha^{j}\beta}b \\
 &= c^{(\lambda-1)M_i\alpha^i\beta} a_2^{\lambda \alpha^{j}\beta} \\
 &= a_1^{(\frac{1}{2}(1-\lambda)i^2+j\lambda)\lambda^{2 x_2}} a_2^{((1-\lambda)i + \lambda)\lambda^{ x_2}} .
\end{align*}
Applying $\gamma$ to both the sides we obtain
$$ \gamma(b)^{-1}\gamma(a_2)\gamma(b)=\gamma(a_2)^{((1-\lambda)i + \lambda)\lambda^{x_2}}, $$
and comparing with
$$ \gamma(b)^{-1}\gamma(a_2)\gamma(b) = \beta^{-1}\iota(c^{M_i})\alpha^i = \iota(c^{M_i \beta})\alpha^{i\lambda^{x_2}}
$$
we obtain $(1-\lambda)i+\lambda=i$, so that $i=1$.

Now we show that if the map $\gamma$ extends to a GF then there always exist at least one Sylow $q$-subgroup $B$ which is $\gamma(B)$-invariant. 

Write $x=a_1^{w}a_2^{z}$ for an element in $A$. If $\gamma$ extends to a GF, then 
$$ \gamma(xb)= \gamma(a_2^{z\lambda^{- x_2}})\gamma(b) = \iota(c^{M_K})\alpha^{K}\beta, $$
where $K:=j+z\lambda^{x_2}$. 

Since
\begin{equation*}
(xb)^{\gamma(xb)} 
= (xb)^{\iota(c^{M_K})\alpha^{K}\beta} 
= (x c^{(-1+\lambda^{-1})M_K})^{\alpha^{K}\beta} b ,
\end{equation*}
$(xb)^{\gamma(xb)} $ belongs to $\Span{xb}$ if and only if 
\begin{equation}
\label{eq:G-7-inv-q-syl}
 x^{1-\alpha^{K}\beta} = c^{(-1+\lambda^{-1})M_K \alpha^K \beta} .
\end{equation}
Writing $x$ and $c$ in the basis $\Span{a_1, a_2}$ and looking at their second component in~\eqref{eq:G-7-inv-q-syl} we find
\begin{equation} 
\label{eq:G-7-inv-q-syl-z}
z(\lambda^{-1}-\lambda^{x_2})= j\lambda^{x_2} (1-\lambda^{-1}) .
\end{equation}

If $x_2 \ne -1$ then there is a unique solution for $z$ and in this case the first component~\eqref{eq:G-7-inv-q-syl} yields
$$ w(1-\lambda^{2 x_2})=j^2 \lambda^{2 x_2} \frac{(1-\lambda^{-1})}{2(\lambda^{-1}-\lambda^{x_2})^2} (1-\lambda^{2 x_2}) ,$$
so that, since $q>2$ (as $x_2 \ne 0, -1$), there is a unique invariant Sylow $q$-subgroup.

Suppose now $x_2=-1$. By induction one can show that in this case if the map $\gamma$ is a GF then
$$ \gamma(b^m)=\gamma(a_2)^{j(m-(1+\lambda+\cdots +\lambda^{m-1}))}\gamma(b)^{m} .$$
Looking at the exponent of $\alpha$ in $\gamma(b^q)=1$ we obtain that $jq=0$, namely $j=0$.

Therefore~\eqref{eq:G-7-inv-q-syl-z} yields that there are $p$ solutions for $z$. Moreover in this case 
$$
w(1-\lambda^{-2})= \frac{1}{2} z^2 (1+\lambda^{-1}),
$$
so that there are $p$ invariant Sylow $q$-subgroups when $q>2$ and $p^2$ when $q=2$.

Now, since the Sylow $q$-subgroup $\Span{b}$ is invariant, $\gamma(b)=\beta^k$.

With computations similar to the previous cases one can show that the assignment
\begin{equation*}
\begin{cases}
\gamma(a_2)=\iota(c)\alpha \\
\gamma(b)=\beta
\end{cases}
\end{equation*}
extends to a GF, namely if $a=a_1^{s}a_2^{t}$ then the map defined as
$$ \gamma(ab^{k})=\iota(c^{M_{t\lambda^{-kx_2}}})\alpha^{t\lambda^{-k x_2}}\beta^{k} $$
satisfies the GFE.

Since there are 
$p+1$ choices for $\Span{a_1}$, 
$p$ choices for $\Span{a_2}$,
$p-1$ choices for $a_2$ in $\Span{a_2}$,
$q-1$ choices for $x_2$, and
$p^2$ choices for the Sylow $q$-subgroup fixed by $\alpha$ and $\beta$, we have $p^3(p^2-1)(q-1)$ GF. They are all distinct if $\Span{b}$ is the unique invariant Sylow $q$-subgroup. Otherwise, when $x_2=-1$, the $p$ (respectively $p^2$ when $q=2$) choices for an invariant Sylow $q$-subgroup yield the same GF, and so there are $p^2(p^2-1)$ (respectively $p(p^2-1)$) distinct GF.

We have
\begin{align*}
 b^{\ominus 1} \circ a_{1} \circ b  &= a_1^{\lambda^{2 x_2+1}}, \\
 b^{\ominus 1} \circ a_{2} \circ b  
 &= a_1^{\frac{1}{2}(1-\lambda)\lambda^{2 x_2}} a_2^{\lambda^{x_2}} ,
\end{align*}
and since $a_1^{\circ k}= a_1^k$ and $a_2^{\circ k}= a_2^{k}$ modulo $\Span{a_1}$,
denoting as usual by $Z_{\circ}$ the action of $\iota(b)_{|A}$ in $(G, \circ)$, we have
\begin{equation*} 
Z_\circ \sim 
\begin{bmatrix}
    \lambda^{2 x_2 + 1} & 0\\
    0 & \lambda^{x_2}
  \end{bmatrix}.
\end{equation*} 
\begin{description}
\item[Type 6] when $Z_{\circ}$ has an eigenvalue $1$, namely for $x_2=\frac{q-1}{2}$. In this case there are $p^3(p^2-1)$ groups.
\item[Type 7] when $Z_{\circ}$ is scalar, namely for $x_2=-1$. In this case there are $p^2(p^2-1)$ groups if $q>2$ and $p(p^2-1)$ if $q=2$.
(Note that for $x_1=-1$,  $b^{\ominus 1} \circ a_{2} \circ b = a_2^{\circ \lambda^{x_2}}$, so that $Z_\circ$ is actually a scalar matrix.)
\item[Type 8] when $q>3$ and $x_2 \ne -1, \frac{q-1}{2}, \frac{q-1}{3}$, so there are $p^3(p^2-1)(q-4)$ groups of type 8. 

Here each group $(G, \circ)$ is isomorphic to $G_{2+x_2^{-1}}$ for a certain $x_2$. For each $s\ne 2$, $s \in \mathcal{K}$, there are $2p^3(p^2-1)$ groups isomorphic to $G_s$, namely those obtained for $x_2$ such that $2+x_2^{-1}=s$ and $2+x_2^{-1}=s^{-1}$, while there are $p^3(p^2-1)$ groups isomorphic to $G_2$, as they can be obtained just for $x_2$ such that $2+x_2^{-1}=2^{-1}$. 

\item[Type 9] when $x_2 \ne -1, \frac{q-1}{2}$ and $Z_{\circ}$ has determinant equal to $1$, namely when $x_2=\frac{q-1}{3}$ and $q>3$. In this case we obtain $p^3(p^2-1)$ groups.
\end{description}

As to the conjugacy classes, with computations similar to the previous cases we obtain orbits of length $p^3(p^2-1)$ when $x_2 \ne -1$, otherwise $x_2=-1$ and there is unique orbit of length $p^2(p^2-1)$ if $q>2$, and $p(p^2-1)$ if $q=2$.

\subsection{The cases \texorpdfstring{$\Size{\ker(\gamma)}=q$ and $\Size{\ker(\gamma)}=1$ when $q=2$}{ker q, ker 1}}
\label{G7-and-q-is-2}

The case $q=2$ was treated by Crespo in~\cite{Cre2021}. In the following we show that we obtain the same results using the gamma functions.

As explained in Subsection~\ref{ssec: duality-G-7}, in the case $\gamma(A)\leq \Inn(G)$ and $q=2$ we can not assume that $p\mid \Size{\ker(\gamma)}$. But, except for the case when both $\gamma$ and $\gammatilde$ have kernel of size not divisible by $p$, we will use duality to swich to a more convenient kernel as we done for the type 9. In particular, here we have to consider also the kernels of size $q$ and $1$.

\subsubsection{The case \texorpdfstring{$\Size{\ker(\gamma)}=q=2$}{ker q}}
\label{sssec:G7-kernel-q}
With exactly the same argument used for the type 9 (see Subsection~\ref{ssec:G9-kernel q}), we obtain $p^2$ groups of type 5, which form a unique orbit of length $p^2$.

\subsubsection{The case \texorpdfstring{$\Size{\ker(\gamma)}=1$}{ker = 1}}
\label{sssec:G7-ker-1}

Also in this case we proceed as for the type 9 (see Subsection~\ref{ssec:G9-ker-1}), namely we divide the GF's of this case according to the size of  $\ker(\gammatilde)$, and for those for which $\Size{\ker(\gammatilde)}\ne1$ we use the previous computations applied to $\gammatilde$. 
The others, for which $\Size{\ker(\gammatilde)}=1$, can be obtained using Proposition~\ref{prop:more-lifting}.

Arguing as in Subsection~\ref{ssec:G9-ker-1}, we obtain the following.
\begin{itemize}
\item When $\Size{\ker(\gammatilde)}=p^2q$, $\gamma=\tilde{\gammatilde}$ corresponds to the left regular representation, and this gives one group of type 7. 
\item In the remaining cases for which $q\mid \Size{\ker(\gammatilde)}$, none of the corresponding $\gamma$'s has trivial kernel.
\item If $\Size{\ker(\gammatilde)}=p^2$, the $p^2$ GF's $\gammatilde$ corresponding to $(G,\circ)$ of type 5 are such that the corresponding $\gamma$ have kernel of size $q$, and all the others $\gammatilde$ correspond to $\gamma$ with kernel of size $1$. Therefore, in the case of kernel of size $p^2$ we double all the GF's except those corresponding to $(G,\circ)$ of type 5.
\item If $\Size{\ker(\gammatilde)}=p$, again, $\ker(\gamma)$ can have size $1$ or $q$.
By Subsection~\ref{ssec:G7_kerp-inn}, the $\gammatilde$'s such that $p\mid \Size{\ker(\gammatilde)}$ and $p \nmid \Size{\ker(\gamma)}$ are those of the cases A1 and A1$^{\ast}$. Moreover, for every $\gammatilde$ belonging to these cases the corresponding $\gamma$ are injective. Therefore, in the case of kernel of size $p$ we double the GF's of cases A1 and A1$^{\ast}$.
\item When $\Size{\ker(\gammatilde)}  = 1$, following the argument in Subsection~\ref{ssec:G9-ker-1}, we obtain $2p(p-1)(p-2)\cdot \frac{p(p+1)}{2}=p^2(p^2-1)(p-2)$ groups $(G, \circ)$ of type 6, which split in $p-2$ classes of length $p^2(p^2-1)$. 
\end{itemize}

\subsection{Results}

Taking into account what we said in Subsection~\ref{ssec: duality-G-7}, we obtain the following.

\begin{prop}
\label{prop:G7}
Let $G$ be a group of order $p^2q$, $p>2$, of type 7.
Then in $\Hol(G)$ there are:
\begin{enumerate}
\item $p^{3}(3p+1)$ groups of type 5, which split in two conjugacy classes of length $p^2$, one conjugacy class of length $p^3(p+1)$ and two conjugacy classes of length $p^2(p^2-1)$;
\item if $q=2$
\begin{enumerate}
	\item $p^3(p+1)(3p + 1)$ groups of type 6, which split in 
$4$ conjugacy classes of length $p^{2}(p+1)$, 
$p+4$ conjugacy classes of length $p^2(p^2-1)$, and
two conjugacy classes of length $p^{3}(p^2-1)$;

in total $10+p$ conjugacy classes;
	\item $2+p(p+1)(2p-1)$ groups of type 7, which split in 
two conjugacy classes of length $1$, 
two conjugacy classes of length $p(p^2-1)$, and
one conjugacy class of length $p(p+1)$;

in total $5$ conjugacy classes;
\end{enumerate}
\item if $q>2$
\begin{enumerate}
	\item $4p^2(p+1)(p^2+pq-2p)$ groups of type 6, which split in 
$4$ conjugacy classes of length $p^{2}(p+1)$, 
$4(q-2)$ conjugacy classes of length $p^3(p+1)$, 
$4$ conjugacy classes of length $p^2(p^2-1)$, 
and $4$ conjugacy classes of length $p^3(p^{2}-1)$;

in total $4(q+1)$ conjugacy classes;
\item $2+p^2(2p^2+pq+2q-4)$
groups of type 7, which split in 
two conjugacy classes of length $1$, 
$2(q-2)$ conjugacy classes of length $p^2$,
two conjugacy classes of length $p^2(p+1)$,
two conjugacy classes of length $p^2(p^2-1)$, 
and $q-3$ conjugacy classes of length $p^3(p+1)$;

in total $3q-1$ conjugacy classes;
\item if $q=3$, $p(p+1)(2p^3+3p^2-2p+1)$ groups of type 9, which split in 
one conjugacy class of length $p^3(p+1)$,
two conjugacy classes of length $p^2(p+1)$, 
one conjugacy class of length $p(p+1)$, 
two conjugacy classes of length $p^3(p^2-1)$,
and $4$ conjugacy classes of length $p^2(p^2-1)$;

in total $10$ conjugacy classes;
\item if $q>3$, 
\begin{itemize}
\item[$-$] $2p^2(p+1)(p^2q+pq-4p+2)$ groups of type 8 isomorphic to $G_2$, which split in 
$4(q-3)$ conjugacy classes of length $p^3(p+1)$, 
$8$ conjugacy classes of length $p^2(p+1)$, 
$4$ conjugacy classes of length $p^2(p^2-1)$, 
and $2q$ conjugacy classes of length $p^3(p^2-1)$;
\item[$-$] for every $s \in \mathcal{K}$, $s \ne 2$, $4p^2(p+1)(2p^2-5p+pq+2)$ groups of type 8 isomorphic to $G_s$, which split in 
$4(q-3)$ conjugacy classes of length $p^3(p+1)$, 
$8$ conjugacy classes of length $p^2(p+1)$, 
and $8$ conjugacy classes of length $p^3(p^2-1)$;
\end{itemize}

in total $6q$ conjugacy classes of groups isomorphic to $G_2$ and $4(q+1)$ conjugacy classes of groups isomorphic to $G_s$ for every $s\ne 2$;
\item if $q>3$, $4p^5+p^4(q-2)+p^3(2q-7)+3p^2+p$
groups of type 9, which split in 
$2q-5$ conjugacy classes of length $p^3(p+1)$,
two conjugacy classes of length $p^2(p+1)$, 
one conjugacy class of length $p(p+1)$, 
and $4$ conjugacy classes of length $p^3(p^2-1)$;

in total $2(q+1)$ conjugacy classes;
\end{enumerate}
\end{enumerate}
\end{prop}

\section{Type 10}
\label{sec : 10}

In this case $q\mid p+1$, where $q>2$, and $G=(\cC_{p} \times \cC_{p}) \rtimes_{C} \cC_{q}$. The Sylow $p$-subgroup $A=\Span{a_1,a_2}$ is characteristic  and a generator $b$ of a Sylow $q$-subgroup acts on $A$ as a suitable power $Z$ of a Singer cycle, namely $a^b=a^Z$ for $a\in A$. We know that $Z$ has determinant 1 and two (conjugate) eigenvalues $\lambda, \lambda^p=\lambda^{-1} \in\F_{p^2}\setminus\F_p$.

The divisibility condition on $p$ and $q$ implies that $(G,\circ)$ can only be of type 5 or 10.

According to Subsections 4.1 and 4.4 of~\cite{classp2q}, we have 
\begin{equation*}
\Aut(G)=(\cC_{p} \times \cC_{p}) \rtimes ( \cC_{p^{2}-1} \rtimes \cC_{2})  ,
\end{equation*}
where $\cC_p \times \cC_p = \iota(A)$, and for $\mu \in \cC_{p^2-1}$ and $\psi\in \cC_2$ we write
\begin{equation}
\label{eq:aut-G-10}
\mu:
\begin{cases}
a \mapsto a^{M} \\
b \mapsto b
\end{cases}, \ 
\psi:
\begin{cases}
a \mapsto a^{S} \\
b \mapsto b^{r}
\end{cases},
\end{equation}
where $M=uI+vZ \in \GL(2,p)$, for $u,v \in \mathbb{F}_p$ not both zero, and $S$, $r$ are such that either $r=1$ and $S=1$, or $r=-1$ and 
$$
S=\begin{bmatrix}
0 & 1 \\
1 & 0 
\end{bmatrix}.
$$

The Sylow $p$-subgroup of $\Aut(G)$ has order $p^2$ and is characteristic and a Sylow $q$-subgroup is cyclic, so $\gamma(G)$ of order a divisor of $p^2q$ is always contained in $\Inn(G)$. 

Moreover
\begin{itemize}
\item  since $A$ is characteristic, it is also a Sylow $p$-subgroup of $(G,\circ)$, so $\gamma(A)$ is a subgroup of $\iota(A)$, the Sylow $p$-subgroup of $\Aut(G)$. 
\item $\gamma_{|A}: A\to\Aut(G)$ is a morphism, as for each $a\in A$ the automorphism $\iota(a)$ acts trivially on the abelian group $A$, and so
\begin{equation*}
\gamma(a)\gamma(a')=
\gamma(a^{\gamma(a')}a')=
\gamma(aa').
\end{equation*}
Therefore, $\gamma(a)=\iota(a^{-\sigma})$ for each $a\in A$, where $\sigma\in \End(A)$.
\end{itemize}
In the following assume $\gamma\ne 1$.

\subsection{The case \texorpdfstring{$\Size{\ker(\gamma)} = p$}{ker = p}} 
\label{G-10-ker-p}
This case does not arise, in fact the group $(G,\circ)$ can not have type 5, since $\Inn(G)$ does not contain an abelian subgroup of order $pq$.
$(G,\circ)$ can neither be of type 10, since a group of type 10 has no normal subgroups of order $p$.

\subsection{The case \texorpdfstring{$\Size{\ker(\gamma)}=p^2$}{ker = p2}} 
\label{G-10-ker-p^2}

Here $\ker(\gamma)=A$ and  $\Size{\gamma(G)}=q$, so $\gamma(G)=\Span{\iota(b)}$ where $b$ is a $q$-element of $G$. In this case $B=\Span{b}$ is the unique $\gamma(G)$-invariant Sylow $q$-subgroup, therefore by 
Proposition~\ref{prop:lifting} 
each $\gamma$ is the lifting of exactly one RGF defined on the unique $\gamma(G)$-invariant Sylow $q$-subgroup. 
So, for each choice of $B=\Span{b}$ ($p^2$ possibilities), we can define $\gamma(b)=\iota(b^{-s})$, with $1\le s \leq q-1$ ($q-1$ choices). 

Since $[B,\gamma(B)]=1$, by Lemma~\ref{Lemma:gamma_morfismi} the RGF's correspond to the morphisms.

For $s=1$ we obtain a group of type 5 and for $s\ne 1$ ($q-2$ choices) we obtain a group of type 10.
In conclusion there are
\begin{itemize}
\item[(i)]$p^2$ groups of type 5; 
\item[(ii)] $p^2(q-2)$ groups of type 10.
\end{itemize}

As to the conjugacy classes, here the kernel $A$ is characteristic, so that every $\phi \in \Aut(G)$, stabilises $\gamma_{|A}$.

All orbits here have length a multiple of $p^2$, as
$$
\gamma^{\iota(x)}(b)
=\iota(x^{-1})\gamma(b)\iota(x) \\ 
=\iota(x^{-1+Z^{s}}b^{-s})=\iota(x^{-1+Z^{s}})\gamma(b).
$$
Now, let $\phi=\mu\psi$, where $\mu$ and $\psi$ are as in~\eqref{eq:aut-G-10}. Then
$$
\gamma^{\phi}(b)
=\phi^{-1}\iota(b^{-sr})\phi 
=\psi^{-1}\iota(b^{-sr})\psi
=\iota(b^{-s})
=\gamma(b),
$$
so that the orbits have length exactly $p^2$.

\subsection{The case \texorpdfstring{$q\mid \Size{\ker(\gamma)}$}{ker = q}}
\label{G-10-ker-q}

In this case $(G,\circ)$ can only be of type 5, as a group of type 10 has no normal subgroups of order $q$ or $pq$. 

Let $B\leq \ker(\gamma)$. Since $A$ is characteristic, by 
Proposition~\ref{prop:lifting} 
each GF on $G$ is the lifting of a RGF defined on $A$, and a RGF on $A$ can be lifted to $G$ if and only if $B$ is invariant under $\{\gamma(a)\iota(a)\mid a\in A\}$. 

 Now, for each $a\in A$, $\gamma(a)=\iota(a^{-\sigma})$, where $\sigma\in\End(A)$, so that $\gamma(a)\iota(a)=\iota(a^{1-\sigma})$. Since every Sylow $q$-subgroup of $G$ is self-normalising, necessarily $\sigma=1$, so that for each $a\in A$
 $$\gamma(a)=\iota(a^{-1}).$$
 
Since $[A,\gamma(A)]=1$, by Lemma~\ref{Lemma:gamma_morfismi} the RGF's correspond to the morphisms. 
So, for each of the $p^2$ choices for $B$  there is a unique RGF on $A$ which lifts to $G$.

In conclusion we obtain $p^2$ groups of type 5. Note that for all the GF's of this case $\Size{\ker(\gamma)}=q$, namely there are no GF's on $G$ with $\Size{\ker(\gamma)}=pq$.

As to the conjugacy classes, as in Subsection~\ref{ssec:G9-kernel q}, since $\iota(A)$ conjugates transitively the $p^2$ Sylow $q$-subgroups of $G$, the $p^2$ GF's are conjugate.

\subsection{The case \texorpdfstring{$\ker(\gamma)=\{1\}$}{ker = 1}} 
\label{ssec:G-10-ker-1}

As in Subsection~\ref{ssec:G9-ker-1}, the GF's of this case can be divided into subclasses according to the size of  $\ker(\gammatilde)$.

In this case $\gamma(G)=\Inn(G)\cong(G,\circ)$, so that all the GF's correspond to groups of type 10.

Let $\gamma(a)=\iota(a^{-\sigma})$ for some $\sigma\in\GL(2,p)$.

Consider first the case $\sigma=1$, namely $\gamma(a)=\iota(a^{-1})$. 
In this case $p^2\mid\Size{\ker(\tilde\gamma)}$, since $\gammatilde(x)=\gamma(x^{-1})\iota(x^{-1})$ for all $x\in G$.
By Subsection~\ref{G-10-ker-p^2}, $\gammatilde(b)=\iota(b^{-s})$, therefore $\gamma(b)=\gammatilde(b^{-1})\iota(b^{-1})=\iota(b^{s-1})$, and $q \mid \Size{\ker(\gamma)}$ precisely when $s=1$. Therefore, the $p^2$ GF's $\gammatilde$ corresponding to $(G,\circ)$ of type 5 are such that the corresponding $\gamma$ have kernel of size $q$, and these have already been considered in Subsection~\ref{G-10-ker-q}. All the other $\gammatilde$ correspond to $\gamma$ with kernel of size $1$.
The latter are $p^2(q-2)$, plus the right regular representation, and all of them correspond to groups of type 10. 
  
Let now $\sigma\ne1$. In this case $1$ is not an eigenvalue of $\sigma$, in fact otherwise $p\mid \Size{\ker(\tilde\gamma)}$, but, as seen before, this implies $p^2\mid \Size{\ker(\tilde\gamma)}$, and hence $\sigma=1$. 
Therefore here both $\sigma$ and $(1-\sigma^{-1})$ are invertible.

Let $b$ be a $q$-element and let $\gamma(b)=\iota(a_0b^{-s})$ for some $a_0\in A$ and $s\not\equiv0\bmod q$. Then Subsection~\ref{subsec:sigma_1-sigma_inv} yields~\eqref{eq:conjugation}, which in our notation here is
\begin{equation}
\label{eq:conjG10}
(\sigma^{-1}-1)^{-1}Z^{-s}(\sigma^{-1}-1)=Z^{1-s}.
\end{equation}
Recall that $Z$ has order $q$ and has two conjugate eigenvalues $\lambda$ and $\lambda^{-1}$ in the extension $\F_{p^2}\setminus\F_p$. An easy computation shows that the corresponding eigenspaces are $\Span{v_1}$ and $\Span{v_2}$ with $v_1=a_1+\lambda a_2$, $v_2=a_1+\lambda^{-1} a_2$.
 
From~\eqref{eq:conjG10} we get $\{\lambda^{-s}, \lambda^{s}\}=\{\lambda^{1-s}, \lambda^{-1+s}\}$, which is possible only for $\lambda^{-s}= \lambda^{-1+s}$: this  gives the condition $2s\equiv 1\bmod q$ and means that $\sigma^{-1}-1$ exchanges the two eigenspaces of $Z$. Therefore, with respect to the basis $\{v_1,v_2\}$,
\begin{equation*}
\sigma^{-1}-1=
  \begin{bmatrix}
 0 &\nu^p \\
 \nu & 0 
  \end{bmatrix},
\end{equation*} 
where $\nu\in\F^*_{p^2}.$
The condition that $\sigma$ has not $0$ and $1$ as eigenvalues reads here as $\det(\sigma^{-1}-1)=\nu^{p+1}\ne 0$ and $\det(\sigma^{-1})=1-\nu^{p+1}\ne0$, so we have $p^2-1-(p+1)=(p+1)(p-2)$ choice for $\sigma^{-1}-1$ and hence for $\sigma$. Since $2s\equiv1\bmod q$, there are $(p+1)(p-2)$ choices for the couple $(\sigma, s)$.

The next Proposition shows that all the GF's of this case can be constructed via gluing. 

\begin{prop}
\label{prop:invariantG10}
Let $\gamma$ be a GF on a group $G$ of type 10. If $\Size{\ker(\gamma)}=1$, then there is a unique Sylow $q$-subgroup $B$ of $G$ invariant under $\gamma(B)$. 
\end{prop}
\begin{proof}
Let $B=\Span{b}$ a Sylow $q$-subgroup of $G$. Then, a Sylow $q$-subgroup $\Span{b^x}$, where $x\in A$, is invariant when $(b^{x})^{\gamma(b^{x})}\in\Span{b^{x}}$, that is, 
$$\gamma(b^{x})\in \Norm_{\Aut(G)}(\iota(b^x)) .$$
Since $\gamma(b^{x})$ is a $q$-element, the latter means that 
\begin{equation}
\label{eq:invariance}
\gamma(b^{x})\in\Span{\iota(b^x)}.
\end{equation}
Since
 \begin{align*}
 \gamma(b^x) &= \gamma(x^{-1+Z^{-1}}b) \\
  &= \gamma(x^{(-1+Z^{-1})\iota(b^{s})})\gamma(b)\\
  &= \iota(x^{(1-Z^{-1})Z^s\sigma} )\iota(a_0b^{-s})\\
  &= \iota(b^{-s})\iota(x^{(1-Z^{-1})Z^s\sigma Z^{-s}}a_0^{  Z^{-s}}),
\end{align*}
~\eqref{eq:invariance} becomes
$$
 \iota(b^{-s})\iota(x^{(1-Z^{-1})Z^s\sigma Z^{-s}}a_0^{  Z^{-s}})=\iota(x^{-1+Z^{-1}}b)^{-s},
$$ 
which is equivalent to
 $$
   b^{-s}x^{(1-Z^{-1})Z^s\sigma Z^{-s}}a_0^{  Z^{-s}}=b^{-s}x^{-Z^{-s}+1}.
 $$
Therefore, we are left with show that the equation 
\begin{equation}
\label{eq:solx}
x^{(1-Z^{-1})Z^s\sigma Z^{-s}}a_0^{  Z^{-s}}=x^{-Z^{-s}+1}
\end{equation}
has a unique solution $x$ for each choice of $a_0\in A$, namely that the matrix
$$ D=(1-Z^{-1})Z^s +(1-Z^s)\sigma^{-1} $$
is invertible. 
One can easily compute $D$ and $\det(D)=(1-\lambda^s)(1-\lambda^{-s})(1-\nu^{p+1})$, and since the latter is non-zero, $D$ is invertible and~\eqref{eq:solx} has a unique solution $x$.
\end{proof}

Summarising, each $\gamma$ admits a unique invariant Sylow $q$-subgroup, so that it is a gluing of a RGF $\gamma_B: B\to \Aut(G)$ and a RGF $\gamma_A$ determined by $\sigma$, with the condition that equation~\eqref{eq:conjG10} holds. Necessarily, $\gamma_B(b)=\iota(b^{-s})$ for some $s$, and by~\eqref{eq:conjG10} we get $2s\equiv1\bmod q$. Therefore for each $B$ ($p^2$ choices), we have only one RGF on $B$ and $(p+1)(p-2)$ choices for $\sigma$, so there are $p^2(p+1)(p-2)$ distinct GF, corresponding to groups of type 10.

As to the conjugacy classes, since each $\gamma$ has a unique Sylow $q$-subgroup $B$ which is $\gamma(B)$-invariant, by Lemma~\ref{lemma:conjugacy}-\eqref{item:invariance-under-conjugacy},
for $a\in A$, $\gamma^{\iota(a)}$ has $\bar{B}=B^{\iota(a)}$ as $\gamma(\bar{B})$-invariant Sylow $p$-subgroup. Now $\iota(A)$ conjugates transitively the Sylow $p$-subgroups of $G$, so that all classes have order a multiple of $p^2$.

Consider $\phi=\mu\psi \in \Aut(G)$, where $\mu$ and $\psi$ are as in~\eqref{eq:aut-G-10}. $\mu$ centralises both $b$ and $\gamma(b)$, so that it stabilises $\gamma_{|B}$. Moreover $\psi$ has order $2$, and $b^{\psi}=b^{r}$, $\iota(b)^{\psi}=\iota(b^{r})$, so that $\psi$ stabilises $\gamma_{|B}$ as well.

As for $\gamma_{|A}$, we have $a^{\phi^{-1}}=a^{\psi\mu^{-1}}=a^{SM^{-1}}$, so that
\begin{align*}
\gamma^{\phi}(a)=\phi^{-1}\gamma(a^{SM^{-1}})\phi
=\phi^{-1}\iota(a^{-SM^{-1}\sigma})\phi,
\end{align*}
and it coincides to $\gamma(a)=\iota(a^{-\sigma})$ if and only if $a^{SM^{-1}\sigma MS}=a^{\sigma}$, namely if and only if $\sigma MS \sigma^{-1} = MS $. The latter can be written as 
\begin{equation} 
\label{eq:G10_conj_class_ker_1}
(\sigma^{-1}-1)^{-1} MS (\sigma^{-1}-1) = MS.
\end{equation}
If $S=1$,~\eqref{eq:G10_conj_class_ker_1} yields $u+v\lambda^{-1}=u+v\lambda$, namely $v=0$, so that there are $p-1$ choices for $\mu$.
If $S\ne 1$,~\eqref{eq:G10_conj_class_ker_1} yields $(u+v\lambda^{-1})\nu^{p-1}=u+v\lambda$, namely
$$ \frac{u}{v}=\frac{\lambda-\lambda^{-1}\nu^{p-1}}{\nu^{p-1}-1} .$$
Since it is fixed by the Frobenius endomorphism, it is actually in $\mathbb{F}_p$, and there are $p-1$ choices for $\mu$.

Therefore, the stabiliser has order $2(p-1)$, and there are $p-2$ orbits of length $p^2(p+1)$.

We summarise, including the right and left regular representations.

\begin{prop}
\label{prop:G10}
Let $G$ be a group of order $p^2q$, $p>2$, of type 10.
Then in $\Hol(G)$ there are:
\begin{enumerate}
\item $2p^2$ groups of type 5, which split in two conjugacy classes of length $p^2$;
\item $2+p^2(2(q-2)+(p+1)(p-2))$ groups of type 10, which split in two conjugacy classes of length $1$, $2(q-2)$ conjugacy classes of length $p^2$, and $p-2$ conjugacy classes of length $p^2(p+1)$. 
\end{enumerate}
\end{prop}

\section{Type 11}
\label{sec:mess}

In this case $p\mid q-1$ and $G=\cC_p \times (\cC_p \ltimes \cC_q)$. Let $Z=\Span{z}$ be the center of $G$, and $B=\Span{b}$ the Sylow $q$-subgroup. 

According to Subsection~4.6 of ~\cite{classp2q}, 
$$\Aut(G) = \Hol(C_{p})\times \Hol(C_{q}),$$ 
so that a Sylow $p$-subgroup of $\Aut(G)$ is of the form $\cC_p \times \mathcal{P}$, where $\cC_p$ is generated by a central automorphism and $\mathcal{P}$ is a Sylow $p$-subgroup of $\Hol(\cC_q)$. Therefore, a subgroup of order $p^2$ in $\Aut(G)$ is generated by an inner automorphism $\iota(a)$, for some non-central element $a$ of order $p$, and the central automorphism
\begin{equation}
\label{eq:def-central-psi}
  \psi :
  \begin{cases}
    z \mapsto z,\\
    a \mapsto a z,\\
    b \mapsto b
  \end{cases}.
\end{equation}

By 
Proposition~\ref{prop:duality} (see also \cite[Corollary 2.25]{p2qciclico})
in counting the GF’s we can suppose $B \leq \ker(\gamma)$.
Therefore the image $\gamma(G)$ is contained in a subgroup of $\Aut(G)$ of order $p^2$, that is,
\begin{equation*}
  \gamma(G) \leq  \Span{ \iota(a), \psi },
\end{equation*}
for $a\in A\setminus Z$, and $\psi$ as in~\eqref{eq:def-central-psi}.

In this case 
there exists at least one Sylow $p$-subgroup $A$ of $G$ which is $\gamma(G)$-invariant (see \cite[Theorem 3.3]{p2qciclico}). More precisely $A=\Span{a, z}$ is $\gamma(G)$-invariant, and it is the unique $\gamma(G)$-invariant Sylow $p$-subgroup if $\gamma(G)\cap \Inn(G) \neq \Set{1}$; otherwise $\gamma(G)\leq \Span{\psi}$, and every Sylow $p$-subgroup is $\gamma(G)$-invariant. 

We may thus apply Proposition~\ref{prop:lifting}, and look for the functions
\begin{equation*}
  \gamma' : A \to \Aut(G)
\end{equation*}
which  satisfy   the  GFE  (we   will  just  wrote  $\gamma$   in  the
following). Since $(A, \circ)$ is abelian, we have
\begin{equation*}
  a^{\gamma(z)} z   
  =  a \circ z 
  =  z \circ a 
  =  z^{\gamma(a)} a 
  =  z a,
\end{equation*}
so that $a^{\gamma(z)} = a$, namely
\begin{equation}
\label{eq:G11-gamma-z}
  \gamma(z) = \iota(a)^{s},
\end{equation}
for some $0 \leq s \leq p-1$. We also have
\begin{equation}
\label{eq:G11-gamma-a}
  \gamma(a) =  \iota(a)^{t} \psi^{u},
\end{equation}
for some $0\leq t,u \leq p-1$. 

If both $s=0$ and $u=t=0$, then $\ker(\gamma) = G$ and we get the right regular representation. 

Proposition~\ref{prop:lifting} yields also that the RGF's on $A$ with kernel of size $1$, respectively $p$, correspond to the GF's on $G$ with kernel of size $q$, respectively $qp$. In the following we suppose $\gamma(G) \ne \Set{1}$.

\subsection{The case \texorpdfstring{$\Size{\ker(\gamma)} = q$}{ker = q}}
\label{ssec:G11kerq}

Here $\gamma(G) = \Span{ \iota(a), \psi }$ and $A= \Span{a, z}$ is the unique Sylow $p$-subgroup of $G$ which is $\gamma(G)$-invariant. 
By the discussion above, we look for the RGF's $\gamma$ on $A$ extending the assignments~\eqref{eq:G11-gamma-z}, ~\eqref{eq:G11-gamma-a}, and with trivial kernel, namely $s\neq0$ and $u\neq 0$.
By Lemma~\ref{remark:gammaGFkerq} in the Appendix there is a unique RGF $\gamma$ like that, and since there are $q$ choices for $A$, we get $q p (p-1)^{2}$ maps. 

As to the circle operation, for every $x\in A$, $x^{\ominus 1 } \circ b \circ x = b^{\gamma(x)\iota(x)}$, so that
$$ a^{\ominus 1 } \circ b \circ a = b^{\iota(a^{t+1})\psi^{u}}, \ 
   z^{\ominus 1 } \circ b \circ z = b^{\iota(a^{s} z)}.
$$
Since $b^{\circ k}=b^k$ and $s \ne 0$, all groups $(G, \circ)$ are of type 11.

As to the conjugacy classes, let $\phi \in \Aut(G)$. Write $\phi=\phi_1 \phi_2$, where $\phi_1 \in \Hol(\cC_p)$ and $\phi_2\in \Hol(\cC_q)$, so that
\begin{equation}
\label{eq:aut-G-11}
  \phi_1 :
  \begin{cases}
    z \mapsto z^i \\
    a \mapsto a z^j \\
    b \mapsto b 
  \end{cases} ,
   \phi_2 :
  \begin{cases}
    z \mapsto z \\
    a \mapsto b^m a  \\
    b \mapsto b^k
  \end{cases}.
\end{equation}
Since the kernel $B$ is characteristic, then $\gamma_{|B}$ is stabilised by every automorphism of $G$.

Moreover 
\begin{align*}
\gamma^{\phi}(a)
&= \phi^{-1}\gamma(az^{-ji^{-1}})\phi \\ 
&= \phi^{-1}\gamma(a)\gamma(z)^{-ji^{-1}}\phi \\
&= \phi^{-1}\iota(a^{t-sji^{-1}})\psi^{u}\phi \\
&= (\iota(a^{t-sji^{-1}}))^{\phi_2}(\psi^{u})^{\phi_1} ,
\end{align*}
and
\begin{align*}
\gamma^{\phi}(z)
&= \phi^{-1}\gamma(z^{i^{-1}})\phi \\ 
&= \phi^{-1}\iota(a^{si^{-1}})\phi \\ 
&= (\iota(a^{si^{-1}}))^{\phi_2} ,
\end{align*}
so that $\phi$ stabilises $\gamma$ if and only if $\phi_1 = 1$ and $\phi_2 \in \cC_{q-1}$.

Therefore, the stabiliser has order $q-1$ and there are $p-1$ orbits of length $qp(p-1)$.

\subsection{The case \texorpdfstring{$\Size{\ker(\gamma)} = p q$}{ker = pq}}
\label{ssec:G11kerpq}

Here $\gamma(G)$ is a subgroup of order $p$ of $\Span{ \iota(a), \psi }$. 
We look for the RGF's $\gamma$ on $A$ extending the assignments~\eqref{eq:G11-gamma-z},~\eqref{eq:G11-gamma-a}, and with kernel of size $p$, namely $s=0$ or $u=0$.

Suppose first that $s=0$, so that the kernel is $ZB$ and $\gamma(a)=\iota(a)^{t}\psi^{u}$. By Lemma~\ref{lem:kernel-pq-gamma-morphisms} in the Appendix
the RGF's on $A$ with kernel of size $p$ are precisely the morphisms.
\begin{enumerate}
\item If $t=0$, then $\gamma(a)=\psi^{u}$ and every Sylow $p$-subgroup is $\gamma(G)$- invariant. Therefore here we obtain $p-1$ groups, and they are all of type 11 as $b$ is $\gamma(b)$-invariant and
$$ a^{\ominus 1} \circ b \circ a = b^{\iota(a)} .$$
\item If $t\neq 0$, then $\gamma(a)=\iota(a)^t\psi^{u}$, and $A= \Span{a, z}$ is the unique $\gamma(G)$-invariant Sylow $p$-subgroup, so that we have $q$ choices for $A$, $p-1$ for $t$ and $p$ for $u$, namely $qp(p-1)$ functions. Since
$$ a^{\ominus 1} \circ b \circ a = b^{\iota(a)^{t+1}} ,$$
they correspond to $qp$ groups of type 5 and $qp(p-2)$ groups of type 11.
\end{enumerate} 

As to the conjugacy classes, here the kernel $ZB$ is charactertistic, so that $\gamma_{|ZB}$ is stabilised by every automorphism of $G$. 

Now, since $a^{\phi} \equiv a \bmod \ker(\gamma)$, we have 
\begin{align*}
\gamma^{\phi}(a)
&= \phi^{-1}\gamma(a)\phi 
= \phi^{-1}\iota(a^{t})\psi^{u}\phi 
= (\iota(a^{t}))^{\phi_2}(\psi^{u})^{\phi_1},
\end{align*}
so that $\phi$ stabilises $\gamma$ if and only if it centralises $\gamma(a)$.

If $t=0$, the last condition is equivalent to say that $\phi_1 \in \Span{\psi}$ and $\phi_2 \in \Hol(\cC_q)$, so that the $p-1$ groups of type 11 form one orbit of length $p-1$.

If $t \neq 0$, then $\phi$ centralises $\gamma(a)$ if and only if $\phi_2 \in \cC_{q-1}$, and either $u\neq 0$ and $\phi_1\in\Span{\psi}$, or $u=0$ and $\phi_1 \in \Hol(\cC_p)$.
In the first case the stabiliser has order $p(q-1)$, and there is one orbit of length $q(p-1)$ for the type 5, and $p-2$ orbits of length $q(p-1)$ for the type 11. In the second case the stabiliser has order $p(p-1)(q-1)$, and there is one orbit of length $q$ for the type 5, and $p-2$ orbits of length $q$ for the type 11.

Suppose now $u=0$, so that $\gamma(a)=\iota(a^t)$ and $\gamma(z)=\iota(a^s)$. Here $\ker(\gamma)=\Span{v}$, where $v=z^{e}a^{f}$ is such that $tf+se=0$. Up to change the basis of $A$, we can appeal again to Lemma~\ref{lem:kernel-pq-gamma-morphisms}, which yields that the RGF's here are exactly the morphisms.
Again, $A= \Span{a, z}$ is the unique $\gamma(G)$-invariant Sylow $p$-subgroup, and we obtain $qp(p-1)$ functions. Since
$$ z^{\ominus 1 } \circ b \circ z = b^{\iota(a^{s})} $$
they correspond to groups of type 11.

As to the conjugacy classes, since $B\leq \ker(\gamma)$ is characteristic, $\gamma_{|B}$ is stabilised by every automorphism $\phi$. Moreover , let $\phi=\phi_1\phi_2$ where $\phi_1$, $\phi_2$ are as in~\eqref{eq:aut-G-11}. We have
\begin{align*}
\gamma^{\phi}(a)
&= \phi^{-1}\gamma(az^{-ji^{-1}})\phi = (\iota(a^{t-sji^{-1}}))^{\phi_2}, \\
\gamma^{\phi}(z)
&= \phi^{-1}\gamma(z^{i^{-1}})\phi = (\iota(a^{si^{-1}}))^{\phi_2},
\end{align*}
so that $\phi$ stabilises $\gamma$ if and only if $\phi_1=\id$ and $\phi_2\in \cC_{q-1}$, namely the stabiliser has order $q-1$, and there is one orbit of length $qp(p-1)$.

We summarise, including the right and left regular representations.

\begin{prop}
\label{prop:G11}
Let $G$ be a group of order $p^2q$, $p>2$, of type 11. Then in $\Hol(G)$ there are:
\begin{enumerate}
\item $2pq$ groups of type 5, which split in two conjugacy classes of length $q$, and two conjugacy classes of length $q(p-1)$;
\item $2p(1+q(p^2-2))$ groups of type 11, which split in 
\begin{enumerate}
\item two conjugacy classes of length $1$;
\item two conjugacy classes of length $p-1$;
\item two conjugacy classes of length $qp(p-1)$;
\item $2(p-1)$ conjugacy classes of length $qp(p-1)$;
\item $2(p-2)$ conjugacy classes of length $q(p-1)$;
\item $2(p-2)$ conjugacy classes of length $q$.
\end{enumerate}
\end{enumerate}
\end{prop}

\section{Conclusions}

The proofs of Theorem~\ref{number_regular} and Theorem~\ref{th:number_skew} are obtained
piecing together the results of Propositions~\ref{prop:G5}, \ref{prop:G6}, \ref{prop:G9}, \ref{prop:G8}, \ref{prop:G7}, \ref{prop:G10}, \ref{prop:G11}, 
and recalling that if the Sylow $p$-subgroups of the groups $\Gamma$ and $G$ are not isomorphic, then $e'(\Gamma, G)=e(\Gamma, G)=0$ (\cite[Corollary 3.4]{p2qciclico}).

To prove Theorem~\ref{number_hgs}, we use Theorem~\ref{th:hgs-regular-skew}~\ref{item:hgs-byott96}, therefore, for each pair of finite groups $\Gamma, G$ with $\Size{\Gamma} = \Size{G}$, we have
\begin{equation*}
  e(\Gamma, G) =  \frac{\Size{\Aut(\Gamma)}} {\Size{\Aut(G)}} e'(\Gamma, G).
\end{equation*}
The values of $e'(\Gamma, G)$ computed in Propositions~\ref{prop:G5}, \ref{prop:G6}, \ref{prop:G9}, \ref{prop:G8}, \ref{prop:G7}, \ref{prop:G10}, \ref{prop:G11} and the cardinalities of the automorphism groups given in Table~\ref{table:grp_aut} yield the values of $e(\Gamma,G)$.

\appendix

\section{}

The following Lemma proves that the maps found in Subsections~\ref{sssec:Gcirc11},~\ref{ssec:G11kerq} in the case $\Size{\ker(\gamma)}=q$ are gamma functions.

\begin{lemma}
\label{remark:gammaGFkerq}
Let $G$ be a group of type 5 or 11, $B$ its Sylow $q$-subgroup, and $A=\Span{a_1, a_2}$ a Sylow $p$-subgroup. 

Let $\gamma: A \to \Aut(G)$ a map such that
\begin{equation}
\label{eq:app-1-ker-q}
\begin{cases}
\gamma(a_1)&=\eta_1 \\
\gamma(a_2)&=\eta_2 ,
\end{cases}
\end{equation}
where ${\eta_1}_{|A}=1$, $a_1^{\eta_2}=a_1$, $a_2^{\eta_2}=a_2a_1^k$, $1\leq k <p$.

Then 
$$ \gamma(a_1^{n} a_2^{m})= \eta_1^{n-k(1+\cdots+(m-1))}\eta_2^{m}$$
is the unique RGF extending the assignment above.
\end{lemma}

\begin{proof}
By our assumptions $A$ is clearly $\gamma(A)$-invariant. Moreover
\begin{align*}
\gamma((a_1^{n} a_2^{m})^{\gamma(a_1^{e} a_2^{f})} a_1^{e} a_2^{f})
&=\gamma({(a_1^{n} a_2^{m})}^{\eta_1^{e-k(1+\cdots+(f-1))}\eta_2^{f}} a_1^{e} a_2^{f})\\
&=\gamma(a_1^{n} {(a_2^{m})}^{\eta_2^{f}} a_1^{e} a_2^{f})\\
&=\gamma(a_1^{n+e+kfm} a_2^{m+f})\\
&=\eta_1^{n+e+kfm-k(1+\cdots+(m+f-1))}\eta_2^{m+f},
\end{align*}
and, on the other hand,
\begin{align*}
\gamma(a_1^{n} a_2^{m})\gamma(a_1^{e} a_2^{f})&=
\eta_1^{n-k(1+\cdots+(m-1))}\eta_2^{m}\eta_1^{e-k(1+\cdots+(f-1))}\eta_2^{f} \\
&=\eta_1^{n-k(1+\cdots+(m-1))+e-k(1+\cdots+(f-1))}\eta_2^{m+f}.
\end{align*}
Therefore $\gamma$ satisfies the GFE if and only if 
$$ -k(\sum_{s=1}^{m+f-1}s)+fkm \equiv -k(\sum_{s=1}^{m-1}s+\sum_{s=1}^{f-1}s) \bmod p ,$$
that is, 
$$\sum_{s=m}^{m+f-1}s -fm \equiv \sum_{s=1}^{f-1}s \bmod p. $$ 
Since $ m+(m+1)+\cdots +(m+f-1)=fm+(1+\cdots +f-1)$, the last condition holds true, and $\gamma$ is a RGF on $A$.

Now let $\gamma'$ be a RGF on $A$ extending the assignment~\eqref{eq:app-1-ker-q}. Since ${\eta_1}_{|A}=1$, necessarily $ a_1^{\circ n}=a_1^{n}$, so
\begin{align*}
\gamma'(a_1^{n} a_2^{m})&=\gamma'((a_2^{m})^{\gamma'(a_1^n)^{-1}})\gamma'(a_1^{n})
=\gamma'(a_2^{m})\gamma'(a_1)^{n} .
\end{align*}
Moreover
\begin{align*} 
\gamma'(a_2^{m})&= \gamma'((a_2^{m-1})^{\gamma'(a_2)^{-1}})\gamma'(a_2) \\
&=\gamma'(a_1^{-k(m-1)}a_2^{m-1})\gamma'(a_2) \\
&= \gamma'(a_1)^{-k(m-1)}\gamma'(a_2^{m-1})\gamma'(a_2).
\end{align*}
By induction we obtain 
$\gamma'(a_2^{m})=\gamma'(a_1)^{-k((m-1)+(m-2)+\cdots + 1)}\gamma'(a_2)^m$, so that
$$\gamma'(a_1^{n} a_2^{m})=\gamma'(a_1)^{n-k((m-1)+(m-2)+\cdots + 1)}\gamma'(a_2)^{m},  $$
namely $\gamma'=\gamma$.
\end{proof}

The following Lemma proves that the maps $\gamma$ in Subsections~\ref{sssec:Gcirc11}, \ref{ssec:G7-kernel-pq} and~\ref{ssec:G11kerpq}, in the case $\Size{\ker(\gamma)}=pq$, satisfy the assumptions of Lemma~\ref{Lemma:gamma_morfismi}.

\begin{lemma}
\label{lem:kernel-pq-gamma-morphisms}
Let $G$ be a group of order $p^2q$, $A=\Span{a_1, a_2}$ a Sylow $p$-subgroup of $G$, and $\gamma:A \to \Aut(G)$ a map such that
\begin{equation*}
\begin{cases}
\gamma(a_1)=\phi  \text{  (possibly modulo $\iota(A)$) }\\
\gamma(a_2)=1
\end{cases},
\end{equation*} 
where $a_2^{\phi}=a_2$, $a_1^{\phi}=a_1a_2^{k_\phi}$ for a certain $k_\phi$.

Then $\gamma$ extends to a unique RGF on $A$ if and only if $\gamma$ is a morphism.
\end{lemma}

\begin{proof}
We show that 
$$ [A,\gamma(A)]=[A, \gamma(\Span{a_1})]\subseteq \Span{a_2}, $$
and then, by Lemma \ref{Lemma:gamma_morfismi}, the RGF's on $A$ with kernel $\Span{a_2}$ correspond to the morphisms $A\to\Aut(G)$.

Note that if $\gamma$ is a RGF or a morphism, then $\gamma(a_1^s)=\gamma(a_1)^s$, as $\ker(\gamma)=\Span{a_2}$.
Thus we have
\begin{align*}
(a_2^ma_1^t)^{-1}(a_2^ma_1^t)^{\gamma(a_1^s)}&=(a_2^ma_1^t)^{-1+\phi^{s}}=
(a_1^t)^{-1+\phi^{s}}  = a_2^{tk_{\phi^s}}\in \Span{a_2}.
\end{align*}
\end{proof}

 

\providecommand{\bysame}{\leavevmode\hbox to3em{\hrulefill}\thinspace}
\providecommand{\MR}{\relax\ifhmode\unskip\space\fi MR }
\providecommand{\MRhref}[2]{%
  \href{http://www.ams.org/mathscinet-getitem?mr=#1}{#2}
}
\providecommand{\href}[2]{#2}

\end{document}